\documentclass[11pt]{amsart}
\usepackage[pdftex]{graphicx}
\usepackage{amssymb}
\usepackage{amsmath}
\usepackage{amsfonts}
\usepackage{amsthm}
\usepackage{amssymb}
\usepackage{mathrsfs}
\usepackage{enumitem}
\usepackage{xcolor}
\newtheorem{theorem}{Theorem}[section]

\newtheorem{proposition}[theorem]{Proposition}
\newtheorem{corollary}[theorem]{Corollary}

\newtheorem{lemma}[theorem]{Lemma}
\newtheorem{conjecture}[theorem]{Conjecture}
\newtheorem{open}[theorem]{Open problem}
\theoremstyle{definition}
\newtheorem{definition}[theorem]{Definition}
\theoremstyle{remark}
\newtheorem{remark}[theorem]{Remark}

\newtheorem{example}[theorem]{Example}

\newcommand{\vol}{\mathrm{vol}}

\newcommand{\Sp}{\mathbb{S}^2}

\newcommand{\mE}{\mathcal E}
\newcommand{\mC}{\mathcal C}

\newcommand{\mM}{\mathcal M}

\newcommand{\mZ}{\mathcal Z}

\newcommand{\dist}{\mathrm{dist}}

\newcommand{\eps}{\epsilon}
\newcommand{\R}{\mathbb{R}}

\newcommand{\goto}{\underset{\epsilon\rightarrow 0}{\longrightarrow}}
\newcommand{\Met}{\mathrm{Met}}
\newcommand{\Diff}{\mathrm{Diff}}
\newcommand{\can}{\mathrm{can}}
\newcommand{\conf}{\mathrm{conf}}
\newcommand{\Conf}{\mathrm{Conf}}

\DeclareMathOperator{\area}{\mathrm{Area}}
\DeclareMathOperator{\ind}{\mathrm{ind}}
\DeclareMathOperator{\nul}{\mathrm{nul}}
\numberwithin{equation}{section}

\begin{document}

\title[Stability of isoperimetric inequalities]{Stability of  isoperimetric inequalities for  Laplace eigenvalues on surfaces}

\author[M. Karpukhin]{Mikhail Karpukhin}
\address[Mikhail Karpukhin]{Mathematics 253-37, California Institute of Technology, 
Pasadena, CA 91125, USA}
\email{mikhailk@caltech.edu}
\author[M. Nahon]{Micka\"el Nahon}
\address[Micka\"el Nahon]{Univ. Savoie Mont Blanc, CNRS, LAMA, 73000 Chamb\'ery, France}
\email{mickael.nahon@univ-smb.fr}
\author[I. Polterovich]{Iosif Polterovich}
\address[Iosif Polterovich]{D\'epartement de math\'ematiques et de 
statistique, Universit\'e de Montr\'eal, CP 6128 Succ. Centre-Ville, 
Montr\'eal, Qu\'ebec, H3C 3J7, Canada}
\email{iosif.polterovich@umontreal.ca}
\author[D. Stern]{Daniel Stern}
\address[Daniel Stern]{Department of Mathematics, University of Chicago,
5734 S University Ave
Chicago IL, 60637, USA}
\email{dstern@uchicago.edu}
\date{}
\maketitle
\begin{abstract}
We prove stability estimates for the isoperimetric inequalities for the first and the second nonzero Laplace eigenvalues on surfaces, both globally and in a  
fixed conformal class.  We employ the notion of eigenvalues of measures and show that if a normalized eigenvalue is close
to its maximal value, the corresponding measure must be close in the Sobolev space $W^{-1,2}$ to the set of maximizing measures. In particular, this implies a qualitative stability result: metrics almost maximizing the normalized eigenvalue must be $W^{-1,2}$--close to a maximal metric.
Following this
approach, we prove sharp quantitative stability of the celebrated Hersch's inequality for the first eigenvalue on the sphere, as well as of its
counterpart for the second eigenvalue. Similar results are also obtained
for the precise isoperimetric eigenvalue inequalities on the projective
plane, torus, and Klein bottle. The square of the $W^{-1,2}$ distance to a maximizing measure in these stability estimates
is controlled by the difference between the normalized eigenvalue and its 
maximal value, indicating that the maxima are in a sense nondegenerate.
We construct examples showing that the power of the distance can not be improved,
and that the choice of the 
Sobolev space $W^{-1,2}$ is optimal.

\end{abstract}
\section{Introduction and main results}
\subsection{Stability of  isoperimetric inequalities} 
Let $M$ be a compact smooth surface without boundary. Given a conformal class $\mC$ on $M$,  we set
\begin{equation}
\label{eq:confmax}
\Lambda_k(M, \mC) = \sup_{g\in {\mC}} \bar\lambda_k(M,g),
\end{equation}
where $g$ denotes a Riemannian metric, and 
\begin{equation}
\label{eq:normalized}
\bar{\lambda}_k(M,g)=\lambda_k(M,g) \area(M,g)
\end{equation}
is the $k$-th nonzero normalized Laplace eigenvalue on the Riemannian manifold $(M,g)$. The quantities $\Lambda_k(M,\mC)$ are sometimes referred to as the {\it conformal spectrum}, see \cite{CE}.
We also set 
\begin{equation}
\label{eq:globalmax}
\Lambda_k(M)=\sup_g \bar \lambda_k(M,g),
\end{equation}
where the supremum runs over all Riemannian metrics $g$ on $M$.
It is well-known  (see \cite{Kor93, Has11,  KNPP2})  that $\Lambda_k(M) < +\infty$ for all surfaces. Moreover, as was shown for $k=1$ in \cite{Petrides} and \cite{MS}, 
the suprema in \eqref{eq:confmax}  and \eqref{eq:globalmax} are  attained on metrics  which are  smooth except at a possibly  finite number of conical singularities.
We say that a metric is  {\it conformally maximal}  for the eigenvalue $\lambda_k$ if it  realizes the supremum in \eqref{eq:confmax} for the given conformal class $\mC$.
Similarly, 
a metric is called  {\it globally maximal} for $\lambda_k$  if it realizes the supremum in \eqref{eq:globalmax}.

In  the present paper  we focus on the  following question: suppose that $g\in \mC$ is a Riemannian metric such that the difference $\Lambda_1(M,\mC)-\bar \lambda_1(M,g)$
(respectively, the difference  $\Lambda_1(M)-\bar \lambda_1(M,g)$ )  is small. In which sense the metric $g$ is close to a conformally (respectively, globally) maximal metric?
In other words, we would like to find an appropriate distance on the space of  Riemannian metrics in which the isoperimetric inequalities for $\lambda_1$  on surfaces are {\it stable}. We also discuss some related  results for higher eigenvalues, notably for $k=2$.  To our knowledge,   stability of spectral isoperimetric inequalities have not  been previously investigated  in a general Riemannian setting.

The prototypical problem of this kind is the stability of the classical (geometric) isoperimetric inequality.  It is well-known (see  \cite{FusMagPrat08} and references therein)  that 
the isoperimetric deficit of a Euclidean domain  $\Omega$  (i.e. the difference
$\vol_{d-1}(\partial\Omega)-\vol_{d-1}(\partial \mathbb{D})$, where $\Omega \subset  \mathbb{R}^d$ and  $\vol_d(\Omega)=\vol_d(\mathbb{D})$)
controls the Fraenkel asymmetry of $\Omega$, i.e. a certain measure of its deviation from a ball.  Fraenkel asymmetry also  turns out  to be the right  quantity in the stability results for the Faber--Krahn and Szeg\H{o}--Weinberger inequalities for Dirichlet and Neumann eigenvalue problems on Euclidean domains. Fraenkel asymmetry can be used as well as a measure of 
stability for the Brock inequality, stating that the first nonzero Steklov eigenvalue of a Euclidean domain attains its maximum if the domain is a ball. We refer to \cite{BrasDePhil16} for
a recent  overview of  these and other spectral stability results in the Euclidean case.
 
At the same time, as was recently shown in \cite{BuNa}, the Weinstock inequality stating that the disk maximizes the first nonzero Steklov eigenvalue among all simply-connected planar domains of fixed {\it perimeter} is not stable with respect to Fraenkel asymmetry. However, it is stable with respect to the Sobolev $W^{-\frac{1}{2},2}(\partial \mathbb{D})$--distance between  the boundary measures.
Here $\mathbb{D}$ is a unit disk, and to each simply connected domain $\Omega$ we associate the appropriate  boundary measure  arising from a  conformal map $\Omega \to \mathbb{D}$.
While this distance a priori does not appear to be natural, it turns out to be indicative of what happens in the case of surfaces. 
\subsection{Eigenvalues of measures}  
 Let us recall the definition of the eigenvalues corresponding to Radon measures on surfaces.
Given a  conformal class $\mC$ on a surface  $M$,  we fix a background  metric $g_\mC$ with the corresponding Riemannian volume density $dv_{g_\mC}$. 
 Let $\mu$ be a Radon measure on $M$. Following \cite{KS} we say that $\mu$ is {\it admissible} if  
the identity map on $C^\infty(M)$ extends to a compact operator $W^{1,2}(M,g_\mC) \to L^2(M,\mu)$; it is easy to check that this property depends only on the conformal class but not on the choice of a particular background metric. Define 
\begin{equation}
\label{eq:eigmes}
\lambda_k(M,\mC,\mu)=\inf_{E_{k+1}} \,\, \sup_{0 \ne f \in E_{k+1}}\frac{\int_M |\nabla f|_{g_\mC}^2 dv_{g_\mC}}{\int_M f^2 d\mu},
\end{equation}
where the infimum is taken over all  $(k+1)$-dimensional subspaces $E_{k+1} \subset C^\infty (M)$ which are also $(k+1)$-dimensional in $L^2(M,\mu)$.
We refer to $\lambda_k(M,\mC,\mu)$ as the $k$-th eigenvalue of the measure $\mu$ in the conformal class $\mC$. We also set 
$$
\bar\lambda_k(M,\mC,\mu) = \lambda_k(M,\mC,\mu)\mu(M).
$$ 

Let  $g=\rho(x) g_\mC \in \mC$  be a Riemannian metric, where $\rho(x)>0$ is a smooth conformal factor, and set $d\mu:=dv_g=\rho(x) dv_{g_\mC}$.
Then by the conformal invariance of the Dirichlet energy,  the eigenvalues $\lambda_k(M,\mC,\mu)$ are precisely the Laplace eigenvalues on the Riemannian manifold $(M,g)$.  

This approach goes back to \cite{Kokarev},  and it has been  immensely useful in the study of extremal metrics for Laplace eigenvalues. In particular, it is shown in~\cite{KS} that for $k=1,2$ one has
$$
\Lambda_k(M,\mC) = \sup_{\mu}\bar\lambda_k(M,\mC,\mu).
$$
This allows us to consider distances between measures in order to investigate the stability of isoperimetric eigenvalue inequalities.

\begin{remark}
\label{rem:conformal}
There is a natural conformal invariance associated with admissible measures. Let $\Phi\colon(M,\mC)\to (M,\mC)$ be a conformal automorphism and $\mu$ be an admissible measure. Since the Dirichlet integral is conformally invariant, it is easy to see that $\Phi_*\mu$ is also admissible, $\lambda_k(M,c,\mu) = \lambda_k(M,\mC,\Phi_*\mu)$ and $\bar\lambda_k(M,\mC,\mu) = \bar\lambda_k(M,\mC,\Phi_*\mu)$. Furthermore, if $\phi_k$ is the $k$-th eigenfunction of $\mu$, then $(\Phi^{-1})^*\phi_k$ is the $k$-th eigenfunction of $\Phi_*\mu$. In particular, the set of $\bar\lambda_k$-conformally maximal measures is invariant under the action of the group of conformal automorphisms. 
\end{remark}

\subsection{Stability of eigenvalue inequalities on the sphere} 
\label{subsec:stabsphere}
To introduce  the main results of the present paper, consider first the isoperimetric 
eigenvalue inequalities on the sphere $\mathbb{S}^2$. Note that that  there is a unique conformal structure on $\mathbb{S}^2$,
and therefore a conformally maximal metric for any eigenvalue is also  a globally maximal one.

The celebrated Hersch's inequality \cite{Hersch} states that 
\begin{equation}
\label{eq:hersch}
\bar \lambda_1(\mathbb{S}^2,g) \leqslant 8\pi
\end{equation}
with the equality attained  if and only if $g$ is a round metric.

Recall that the Sobolev space $W^{-1,2}(M,g)$ is defined as the dual space 
to $W^{1,2}(M,g)$. Given a Radon measure $\mu \in \left(C^0(M)\right)^*$, we say that $\mu\in W^{-1,2}(M,g)$ if
\begin{equation}
\label{eq:sobolnorm}
\|\mu\|_{W^{-1,2}(M,g)} := \sup_{ \|f\|_{W^{1,2}(M,g)}=1}\,\, \int_M f(x) d\mu<\infty,
\end{equation}
where the supremum is taken over all $f\in C^\infty(M)$. Here $\left(C^0(M)\right)^*$ denotes the space of all Radon measures, i.e. the dual to the space of continuous functions on $M$. One of the principal findings of the present paper is that the distance between measures in the space $W^{-1,2}$ induced by the norm \eqref{eq:sobolnorm} can be used to  control the stability of isoperimetric inequalities for Laplace eigenvalues on surfaces. Moreover, the choice of this distance is essentially optimal.

Note that for $g\in \mC$ the norm $W^{1,2}(M,g)$ is not conformally invariant.  At the same time, for $M\ne\mathbb{S}^2$ the group of conformal automorphisms is compact and, moreover, it is known that the set of $\bar\lambda_1$-conformally maximal measures of fixed area  is compact as well \cite[Section 6.1]{Kokarev}. Thus, given a fixed metric $g \in \mC$,  one can use the norm $W^{1,2}(M,g)$ for stability estimates near {\em any} $\bar\lambda_1$-conformally maximal metric at the expense of constants depending on $g$. However, for $M=\mathbb{S}^2$ the group of conformal automorphisms is not compact and, therefore, by Remark~\ref{rem:conformal} neither is the space of $\bar\lambda_1$-conformally maximal metrics. As a result, 
one has to either be more precise with the choice of  the conformally maximal metric  or take the action of  the conformal group into account, see~\eqref{eq:S2stabnew1} and~\eqref{eq:S2stabnew2} below. 

Let us now state the first main result. In what follows, we set $\lambda_k(\mathbb{S}^2,\mu):=\lambda_k(\mathbb{S}^2,\mC,\mu)$, given that
the conformal structure on the sphere is unique.
\begin{theorem}[Stability of Hersch's inequality]
\label{S2stability:thm}
 Let $g_0$ be a round metric of curvature one on $\mathbb{S}^2$. Then for any admissible measure $\mu$ on $\mathbb{S}^2$ such that $\lambda_1(\mathbb{S}^2,\mu)=2$, 
there exists a conformal automorphism $\Phi$ of 
$\mathbb{S}^2$ such that 
\begin{equation}
\label{eq:S2stabnew1}
8\pi - \bar \lambda_1(\mathbb{S}^2, \mu) \geqslant  2 \| \Phi_* \mu - dv_{g_0}\|^2_{W^{-1,2}(\mathbb{S}^2, g_0)}.
\end{equation}
\end{theorem}
Note that \eqref{eq:S2stabnew1} can be equivalently rewritten as
\begin{equation}
\label{eq:S2stabnew2}
8\pi - \bar \lambda_1(\mathbb{S}^2,\mu) \geqslant 2 \|  \mu - dv_{\Phi^*g_0}\|^2_{W^{-1,2}(\mathbb{S}^2, \Phi^*g_0)}.
\end{equation}
The proof of  Theorem \ref{S2stability:thm} is presented in subsection \ref{sec:S2proof}.
\begin{remark}
\label{rem:normaliz0}
 Here and further on,  we normalize the measure $\mu$ by the first eigenvalue. Alternatively, one can normalize the area of $\mu$: 
 up to a constant factor  it yields the same stability estimates, see Remark \ref{rem:normaliz} for details.
\end{remark}
Note that inequality \eqref{eq:S2stabnew1} gives a  sharp {\it quantitative stability} result: the power $2$ on the right hand side  of  \eqref{eq:S2stabnew1} can not be replaced by any smaller power, see Proposition \ref{prop:optpower}. The choice of the Sobolev $W^{-1,2}$ norm  is sharp as well, see Theorem \ref{thm:spacesharp} and Remark \ref{rem:stabcont} below. 
\begin{remark}[{\em Stability vs. continuity}]
\label{rem:stabcont}
It was shown in \cite{GKL} that the functional $\mu \mapsto \bar \lambda_1(\mathbb{S}^2,\mu)$ is continuous in the dual of the Orlicz-Sobolev space 
$W^{1,2,-\frac{1}{2}}:=W^{1,2,-\frac{1}{2}}(\mathbb{S}^2 ,g_0)$, which lies between $W^{1, 2-\epsilon}$ and $W^{1,2}$.
 Note that $(W^{1,2, -\frac{1}{2}})^* \subset W^{-1,2}$ and the eigenvalue functional above  is not continuous in  $W^{-1,2}$. In other words, the choice of the space  $(W^{1,2, -\frac{1}{2}})^*$ is essentially optimal for the continuity of eigenvalues.  At the same time, we show in Theorem \ref{thm:spacesharp}  that Hersch's inequality is {\it not} stable in the space  $(W^{1,2, -\frac{1}{2}})^*$.  It therefore appears that we may have either stability or continuity but not both. 
This dichotomy seems to be a general phenomenon for spectral isoperimetric inequalities: in particular, it is easy to check that the Dirichlet and Neumann eigenvalues are not continuous with respect to Fraenkel's asymmetry.
\end{remark}
\begin{remark}[{\em Stability in the Wasserstein distance}]
\label{rem:wasser}
As was shown in \cite{Peyre}, the $W^{-1,2}$ distance between measures can be bounded below (and under some additional assumptions, can be bounded above as well) in terms of the {\it quadratic Wasserstein distance}. Given two unit area measures $\mu, \nu \in W^{-1,2}(M,g)$, 
\[2\Vert \mu-\nu\Vert_{W^{-1,2}(M,g)}\geqslant W_2(\mu,\nu):=\inf_{\gamma}\Vert d_g(\cdot,\cdot)\Vert_{L^2(M\times M,\gamma)} ,\]
where 
$\gamma$ is taken among all Borel measures on $M \times M$ with marginals $\mu$ and $\nu$, and $d_g$ is the distance induced by the metric $g$.

As a consequence, stability  in $W^{-1,2}$  norm  implies stability in the quadratic Wasserstein distance.  The Wasserstein distance is one of the most commonly used metrics between measures, notably in the optimal transport theory. In the spectral context it has previously appeared, for instance, in \cite{KurBurIv, SS}. 
\end{remark}
Let us now state the analogue of Theorem \ref{S2stability:thm} for the second nonzero Laplace eigenvalue on the sphere. It was shown in \cite{NadirashviliS2, PetridesS2} that 
$\bar \lambda_2(\mathbb{S}) \leqslant 16 \pi$ with the equality attained in the limit as  a   maximizing sequence of metrics tends to  a disjoint union of identical round spheres. 
This phenomenon is called ``bubbling'', see \cite{NaSi2, Petrides2, KNPP2}. In the language of  measures it can be viewed as a sum of the Riemanninan measure on a round sphere and a Dirac measure  (i.e. a  bubble) of the same mass. 
Our next result is the following quantitative stability estimate.
\begin{theorem}
\label{thm:stab2S2}
Let $g_0$ be a round metric of curvature one on $\mathbb{S}^2$. Then for any admissible measure $\mu$ 
such that $\lambda_2(\Sp, \mu)=2$, there exists a conformal automorphism 
$\Phi\colon\mathbb{S}^2\to\mathbb{S}^2$ and a point $p\in\mathbb{S}^2$ such that
\begin{equation}
16\pi-\bar\lambda_2(\mathbb{S}^2, \mu) \geqslant c\, \|dv_{g_0}+ 4\pi \delta_p- \Phi_*\mu\|^2_{(C^1(\mathbb{S}^2))^*}
\end{equation}
for some constant $c>0$. Here $\delta_p$ is the unit Dirac measure at $p$, and  $(C^1(\mathbb{S}^2))^*$ denotes the dual space to $C^1(\mathbb{S}^2)$.
\end{theorem}
Theorem \ref{thm:stab2S2} is proved in subsection \ref{subsec:stab2S2}. We remark that the constant $c$ on the right-hand side can  be computed explicitly. We also note the stability is proved in the  
norm of the  space $(C^1(\mathbb{S}^2))^*$  which is larger than $W^{-1,2}(\mathbb{S}^2,g)$, and hence the stability estimate is weaker. While we do not know whether the choice of the space $(C^1(\mathbb{S}^2))^*$ is optimal, due to the presence of Dirac measures  the isoperimetric inequality for the second eigenvalue can not be stable in $W^{-1,2}(\mathbb{S}^2,g)$, as $\delta_p\not\in W^{-1,2}(\mathbb{S}^2,g)$.

As was shown in \cite{KNPP1}, $\bar \lambda_k(\mathbb{S}^2,g) \le 8 \pi k$ for any $ k\ge 2$, with the equality attained in the limit as a maximizing sequence of metrics tends to $k$ disjoint identical round spheres (or, equivalently, a sphere with $k-1$ bubbles). It would be interesting to extend the stability result  to any $k\geqslant 2$. However, the proof of  Theorem \ref{thm:stab2S2} relies on the min-max energy characterization of maximal metrics (see the key  Lemma \ref{main:lemma}) which has been so far proved only for $k=1,2$ \cite{KS}.  Still, we propose the following
\begin{conjecture}
\label{conj:highersphere}
Let $g_0$ be a round metric of curvature one  on $\mathbb{S}^2$. Given $k \ge 2$, 
let $\mu$ be an admissible measure such that $\lambda_k(\Sp,\mu)=2$.
Then there exists a conformal automorphism 
$\Phi \colon\mathbb{S}^2\to\mathbb{S}^2$ and $m \leqslant k-1$ distinct points   $p_1, \dots, p_m \in\mathbb{S}^2$ such that
\begin{equation}
8 \pi k -\bar\lambda_k(\mathbb{S}^2, \mu) \geqslant c \left\|dv_{g_0}+ \sum_{i=1}^m 4 \pi \alpha_i \delta_{p_i}- \Phi_*\mu \right\|^2_{(C^1(\mathbb{S}^2))^*}
\end{equation}
for some constant $c>0$ and positive integers $\alpha_i$ satisfying the condition $\sum_{i=1}^m \alpha_i=k-1$.
\end{conjecture}
\begin{remark} The condition $ m\le k-1$ arises  since the bubbles can be attached at the same point (or, in a sense, one on top of the other, cf.  the 
``bubble tree'' construction in  \cite{KNPP2}).  Hence the number of distinct points at which the Dirac measures arise may be smaller than $k-1$. 
\end{remark}
\subsection{Maximal metrics, harmonic maps and minimal surfaces}  In order to provide the context for our results on a general surface $M$, we recall the connection between isoperimetric eigenvalue inequalities and harmonic maps to spheres. The map $u\colon (M,\mC)\to\mathbb{S}^n$ is called harmonic if it is a critical point of the Dirichlet energy functional
\begin{equation}
\label{eq:diren}
E(u)=\frac{1}{2} \int_M |du|_g^2 dv_g
\end{equation}
or, equivalently,
$$
\Delta_g u = |du|^2_gu,
$$
where $g\in\mC$ is any metric in the conformal class. The conformal covariance of $\Delta$ implies $\Delta_{g_u}u = 2u,$ where $g_u = \frac{1}{2}|du|^2_gg$. Therefore, $\lambda_k(M,g_u) = 2$ for some $k>0$ and $\bar\lambda_k(M,g_u) = 2E(u)$. Note that $g_u$ is smooth outside of finitely many conical singularities at zeroes of $du$. We say that the harmonic map $u\colon (M,\mC)\to\mathbb{S}^n$ is a {\em $\bar\lambda_k$-conformally maximal map} if $g_u$ is a $\bar\lambda_k$-conformally maximal metric or, equivalently, if components of $u$ are $k$-th eigenfunctions of $\Delta_{g_u}$ and $\Lambda_k(M,\mC) = 2E(u)$. If $g$ is a smooth (up to a finite number of conical singularities) $\bar\lambda_k$-conformally maximal metric, then $g=\alpha g_u$ for some constant $\alpha>0$ and a $\bar\lambda_k$-conformally maximal harmonic map $u$, see e.g~\cite{ESIextremal}. For $k=1$ the existence of such metrics is established in~\cite{Petrides}. At the same time, as was shown in \cite{KS}, if $\mu$ is a $\bar\lambda_1$-conformally maximal admissible measure, then $\mu = \alpha dv_{g_u}$ for some $\bar\lambda_1$-conformally maximal map $u$.
 
An analogous characterization holds for the globally $\bar\lambda_1$-maximal metrics. If $u\colon M\to\mathbb{S}^n$ is a branched minimal immersion and  $\mC_u = [u^*g_{\mathbb{S}^n}]$, then $u\colon (M,\mC_u)\to\mathbb{S}^n$ is a harmonic map and $u^*g_{\mathbb{S}^n} = g_u$. We say that a branched minimal immersion $u$ is a {\em $\bar\lambda_k$-maximal map} if $g_u$ is $\bar\lambda_k$-maximal metric. Similarly to the maximizers in a fixed conformal class, if $g$ is a smooth (up to a finite number of conical singularities) $\bar\lambda_k$-maximal metric, then $g=\alpha g_u$ for some constant $\alpha>0$ and a $\bar\lambda_k$-maximal immersion $u$, see~\cite{ESIextremal}. For $k=1$ the existence of such metrics is established in~\cite{MS17}. 

\subsection{Stability in the conformal class} Consider now surfaces other than the sphere. 
We have  the following qualitative stability result for  the maximizers  of the first Laplace eigenvalue in any conformal class.
\begin{theorem}
\label{thm:qualconf}
Let $M\ne \mathbb{S}^2$ be a closed surface with a fixed conformal class $\mC$ and  $g_\mC \in \mC$ be a background metric. Let $\mu_j$ be a sequence of admissible measures of unit area, such that $\bar\lambda_1(M, \mC,\mu_j)\to\Lambda_1(M, \mC)$. Then there exists a  smooth  $\bar\lambda_1$-conformally maximal measure $\mu$ of unit area  such that  up to a choice of a subsequence $\mu_j \to \mu$ strongly in $W^{-1,2}(M, g_\mC)$.
\end{theorem}
Theorem \ref{thm:qualconf} is proved in Section \ref{qual:sec}. We also note that under certain  assumptions a similar result can be  proved for the second nonzero eigenvalue, see Theorem~\ref{thm:higher_conf_stab}.

In order to state the results on the  quantitative stability  for maximizers of the first eigenvalue we will need the following definition. 
\begin{definition}
\label{def:quantstab}
 Let $M\ne \mathbb{S}^2$ be a surface.  We say that  the isoperimetric inequality
\begin{equation}
\label{eq:isopconf}
\bar \lambda_1(M, g) \leqslant \Lambda_1(M,\mC)
\end{equation}
 in a  conformal class $\mC=[g]$ on $M$ is {\it quantitatively stable} if there exist positive constants $\delta$ and $c$ with the following property:  for any admissible measure $\mu$ satisfying 
$\Lambda_1(M,\mC) - \bar \lambda_1(M,\mC,\mu)<\delta$ , there exists a $\bar\lambda_1$-conformally maximal metric $g_{\max}\in\mC$ (possibly with finitely many conical singularities) such that
\begin{multline}
\label{eq:quantstabkey}
\Lambda_1(M,\mC) - \bar \lambda_1(M,g_{\max})  \geqslant \\ \geqslant c \|\lambda_1(M,g_{\max})dv_{g_{\max}}-
\lambda_1(M,\mC,\mu)\mu\|^2_{W^{-1,2}(M,g)}.
\end{multline}
\end{definition}
The theorem below provides a sufficient condition for the quantitative stability of conformal maximizers for the first eigenvalue. 
\begin{theorem}
\label{thm:quantstab}
Let $(M, g)$ be a surface with a Riemannian metric, possibly with finitely many conical singularities.  Suppose there exists a branched minimal immersion $u\colon M \to \mathbb{S}^n$, $n \ge 3$, 
 by the first eigenfunctions of the Laplacian  $\Delta_g$, such that its image is not contained in a totally geodesic submanifold  $\mathbb{S}^2 \subset \mathbb{S}^n$.
Then the isoperimetric inequality \eqref{eq:isopconf} is quantitatively stable in the conformal class $\mC=[g]$. 
\end{theorem}
Note that  $g$ is a {\it globally} extremal metric for the first eigenvalue if and only if  the corresponding branched immersion by the first eigenfunctions is minimal (see \cite{NadirashviliT2, ESIextremal}). This is the case, for instance 
for the standard metric on the real projective plane that gives rise to the Veronese immersion into $\mathbb{S}^4$. Note that the conformal maximizer is also a global maximizer in this case, since  there is a unique conformal structure on on $\mathbb{RP}^2$. Also, by  the results of \cite{ ElIl00, NadirashviliT2, JNP, EGJ, CKM},  the square and equilateral flat metrics are the only globally extremal metrics for the first eigenvalue on the torus, and  the Lawson bipolar $\tilde\tau_{3,1}$ surface is the unique extremal metric  on the Klein bottle. In all those cases the images of the minimal immersions are spheres of dimension at least three. The following corollary is  immediate. 
\begin{corollary}
\label{cor:quantstab}
The inequality \eqref{eq:isopconf} is quantitatively stable on the real projective plane,   in the conformal classes of the square and equilateral tori, as well as in the conformal class of the unique globally 
$\lambda_1$-maximal metric on the Klein bottle.
\end{corollary}
The proof  of Theorem \ref{thm:quantstab} is given in Section \ref{quant_stability:sec}. In fact, we prove a more general result that also implies stability of the inequality \eqref{eq:isopconf}
in some conformal classes on a torus, for which 
the conformally maximal metrics for the first eigenvalue are flat \cite{ESIR}, see Proposition \ref{prop:rhombus}.  Moreover, in Section \ref{quantII:sec} we show that the quantitative stability holds under some conditions on the Jacobi fields along $\bar\lambda_1$-conformally maximal harmonic maps,  see also Section \ref{subsec:disc} for a discussion.
\begin{remark}
\label{rem:genustwo} The quantitative stability remains open for the conformal class of the Bolza surface on a surface of genus two, since the branched immersion by the first eigenfunctions is in this case into the sphere $\mathbb{S}^2$. Moreover, the required condition on the Jacobi fields is not satisfied either, see Example \ref{Jacobi_Bolza:ex}. 
Note that the genus two case is particularly difficult from the stability standpoint, because there is a continuous family of maximal metrics, see \cite{JLNNP, NayataniShoda};
in all examples covered by Corollary~\ref{cor:quantstab} the conformal maximizers are unique.
\end{remark}
\subsection{Stability of global maximizers}

Given a closed surface $M$, denote by $\Met_{\can}(M)$ the space of all constant curvature Riemannian metrics on $M$  of  unit area.
By the uniformization theorem, $\Met_{\can}(M)$ is in one-to-one correspondence with the space of conformal classes of metrics on $M$. Note that the diffeomorphism group ${\rm Diff}(M)$ acts naturally on pairs $(g,\mu)\in \Met_{\can}(M)\times C(M)^*$, by
$$\Phi\cdot (g,\mu)=(\Phi^*g, (\Phi^{-1})_*\mu),$$
such that
$$\bar{\lambda}_k(M, [\Phi^*g],  (\Phi^{-1})_*\mu)=\bar{\lambda}_k(M, [g],\mu).$$
It was shown in \cite{MS} that the equality in the isoperimetric inequality
\begin{equation}
\label{eq:isopglob}
\bar \lambda_1(M, g) \leqslant \Lambda_1(M)
\end{equation}
is attained on any surface $M$ by a metric $g=g_{\max}$ which is smooth  possibly except  a finite number of conical singularities. The following qualitative stability result holds.
\begin{theorem}\label{glob:qual:stab} Given a sequence $g_j\in \Met_{\can}(M)$ on $M\neq S^2$ and admissible measures of unit area $\mu_j$ such that
$$\lim_{j\to\infty}\lambda_1(M,[g_j],\mu_j)=\Lambda_1(M),$$
there exists a subsequence (un-relabelled) $(g_j,\mu_j)$, a sequence of diffeomorphisms $\Phi_j\in \Diff(M)$, and a globally $\bar{\lambda}_1$-maximizing unit area metric $g_{\max}$ on $M$ conformal to $g_0\in \Met_{\can}(M)$, such that the pairs 
$$(\tilde{g}_j,\tilde{\mu}_j):=\Phi_j\cdot (g_j,\mu_j)$$ 
satisfy
$$\tilde{g}_j\to g_0\text{ smoothly }$$
and
$$\tilde{\mu}_j\to dv_{g_{\max}}\text{  strongly  in }W^{-1,2}(M,g_0).$$
\end{theorem}
Theorem \ref{glob:qual:stab} is proved in Section~\ref{sec:glob_stab}.

Similar to  Definition \ref{def:quantstab}, let us introduce the notion of the global quantitative stability.
\begin{definition}
\label{def:quantstabglob} We say that  the isoperimetric inequality \eqref{eq:isopglob} for the first eigenvalue is {\it globally  quantitatively stable} on a surface $M \neq \mathbb{S}^2$  if  there exist constants 
$\delta, C>0$ with the following property. Given a metric  $g\in \Met_{\can}(M)$ and an admissible measure 
$\mu$ on $M$  satisfying 
$\bar{\lambda}_1(M, [g],\mu)\geqslant \Lambda_1(M)-\delta$,  there exists a 
$\bar{\lambda}_1$-maximizing metric $g_{\max} \in [g_0]$, $g_0 \in  \Met_{\can}(M)$,  
such that 
\begin{equation}
\label{eq:globstabf}
\begin{split}
&\|g-g_0\|_{C^1(g_0)}^2 + \|\lambda_1(M,[g],\mu)\mu-\lambda_1(M,g_{\max})dv_{g_{\max}}\|_{W^{-1,2}(g_0)}^2\leqslant \\
&\leqslant C\left(\Lambda_1(M)-\bar{\lambda}_1(M, [g],\mu)\right).
\end{split}
\end{equation}
\end{definition}

\medskip

Note that if $\mu=dv_h$ for some metric $h\in[g]$, then \eqref{eq:globstabf} implies
$$
\|\lambda_1(M,h)h-\lambda_1(M,g_{\max})g_{\max}\|_{W^{-1,2}(M,g_{\max})}^2\leqslant C \left( \Lambda_1(M)-\bar{\lambda}_1(M,h)\right).
$$
Here and above the $C^1$ and $W^{-1,2}$ distances between metrics are understood in the sense of  the corresponding norms on tensors, see Remark \ref{rem:tensornorms} for a formal definition.

\begin{remark}
Since $g,g_0\in\Met_{\can}(M)$, the $C^1$ norm in the left-hand side can be replaced by $C^k$ norm at the expense of a possibly different constant~$C$.
\end{remark}
Note that  the globally maximal metrics for the first eigenvalue are known only for the sphere, the real projective plane, the torus, the Klein bottle and the surface of genus two.  The first two cases have already been covered in the previous subsections, since the conformal maximizers coincide with the global ones.  The case of the surface of genus two is beyond our reach for the reasons explained in Remark \ref{rem:genustwo}.  At the same time, for the remaining cases of the torus and the Klein bottle, the quantitative stability of global maximizers can be shown. 
\begin{theorem} 
\label{thm:globstabtorus}
The isoperimetric inequality \eqref{eq:isopglob} is globally quantitatively stable on the torus and on the Klein bottle.
\end{theorem}
Theorem \ref{thm:globstabtorus} is proved in  Section~\ref{sec:glob_stab}.
\subsection{Ideas of the proofs} 
\label{subsec:disc}
Our approach to a large extent relies on the characterization of the conformally maximal metrics in terms of the min-max energy of harmonic maps, see \cite{KS}. Suppose we have an admissible measure $\mu$  on a surface $M$ with a fixed conformal class  $\mC=[g]$,  
and assume that for some $n\ge 1$ there exists  a  sphere-valued map  $u \in W^{1,2}(M,\mathbb{S}^n)$ satisfying a ``balancing'' assumption, i.e. all of its components are  
orthogonal to constants in $L^2(M,\mu)$. Consider  the Dirichlet energy $E(u)$ of the map $u$ defined by \eqref{eq:diren}.  
It turns out that  the difference $2E(u)-\bar \lambda_1(M,\mC,\mu)$ is positive and, moreover, if it is small, then 
the map $u$ is in an appropriate sense close to being  harmonic, and its  components are close to being the first eigenfunctions corresponding to 
$\lambda_1(M,\mC,\mu)$. This result is proved  in Lemma \ref{main:lemma} which is the key technical statement of the paper. Recall that a map is harmonic if and only if its tension field
(see, for instance, \cite{EeSa64} for a definition)
vanishes. Informally, Lemma \ref{main:lemma} states that if $2E(u)-\bar \lambda_1(M,\mC,\mu)$ is small, then the tension field of  $u$ is small in the appropriate norm, see Remark \ref{rem:tensionfield} for further details.

At the same time,  as was shown in \cite{KS},  $\mu$ is a conformally maximal measure for the first eigenvalue if and only if it arises from the energy density of a harmonic map $u$ by the first eigenfunctions of $\Delta_{g_u}$, such that  
\begin{equation}
\label{eq:enerin}
2E(u)- \Lambda_1(M,\mC)=0. 
\end{equation}
In view of this characterization of conformally maximal measures, Lemma \ref{main:lemma} provides a tool to prove stability estimates. The main challenge is to find for each admissible measure $\mu$ an appropriate balanced sphere-valued map $u$, such that $2E(u)\leqslant \Lambda_1(M,\mC)$. This way
$$
2E(u)-\bar \lambda_1(M,\mC,\mu)\leqslant \Lambda_1(M,\mC)-\bar\lambda_1(M,\mC,\mu)
$$
and the stability estimate follows from Lemma~\ref{main:lemma}.
 For particular cases, it is possible to arrange these maps in a {\it comparison family}, see Definition \ref{def:comparfam}. For example, to prove Theorem \ref{S2stability:thm} we use Hersch's lemma: for  any measure $\mu$  there exists a balanced conformal automorphism  $\Phi\colon{\mathbb S}^2 \to {\mathbb S}^2$. Setting $u= \Phi$ and  applying Lemma \ref{main:lemma} yields the result. 

A more elaborate argument is required to prove Theorem \ref{thm:quantstab}. We construct a {\it canonical} comparison family (see Example \ref{canon:ex}) by composing the minimal immersion $u$ satisfying the assumptions of the theorem with conformal automorphisms of the target sphere 
$\mathbb{S}^n$.  Once again, Hersch's lemma is used to show that for each measure $\mu$ the automorphism can be chosen in such a way that that the resulting map is balanced.  The equality in \eqref{eq:enerin} is achieved, because a minimal immersion into a sphere by the first eigenfunctions (which, as was mentioned earlier, yields a {\it globally} extremal metric for the first eigenvalue) corresponds to a conformally maximal metric. Indeed,  by~\cite{LiYau, ESIconfvolume}, see also~\cite[Proposition 3.1]{CKM}, conformal automorphisms reduce the area of minimal surfaces, 
which equals the energy (since the map is conformal), and the energy in turn provides an upper bound for the normalized first eigenvalue. Moreover, one can show that inside the 
canonical comparison family, the maximal metric given by the minimal immersion is a {\it non-degenerate} maximum, and therefore  quantitative stability in the sense of \eqref{eq:quantstabkey} holds inside the canonical comparison family. Theorem \ref{thm:quantstab} then follows by an application of Lemma \ref{main:lemma}.

Intuitively, the quadratic dependence on the right-hand side of the stability estimate \eqref{eq:quantstabkey}  is of the same nature as the stability of the ``model'' nondegenerate maximum of the function $f(x)=-x^2$.  Note that the assumption  in the more general quantitative stability Theorem \ref{Jacobi:cor} that there are no  nontrivial Jacobi fields  along a conformally maximal harmonic map is of similar  nature: essentially, it provides a certain nondegeneracy condition on the maximum. The same can be said about the assumption on the maximality of the Morse index in the global quantitative stability result presented in Theorem \ref{max.ind.thm}. Note that for the surface of genus two the maximum is degenerate (see Remark \ref{rem:genustwo}) and hence the quantitative stability can not be shown by our methods. Still, we believe that the quantitative stability in a conformal class is a  rather general phenomenon (see also Remark \ref{rem:Jaco}):
\begin{open} 
\label{open:stab}
Let $M$ be a surface of negative Euler characteristic and $\mC$ be a conformal class such that no $\bar\lambda_1$-conformally maximal harmonic map is a conformal branched cover of $\mathbb{S}^2$. Show that the  inequality \eqref{eq:isopconf} is quantitatively stable in $\mC$.
\end{open}
\begin{remark}
We suspect that quantitative stability also holds for the conformal class of the Bolza surface and, more generally, for conformal classes with branched covers as  $\bar\lambda_1$-conformally maximal harmonic maps. For such conformal classes the maximal metric is not unique,  and more a appropriate model seems to be that of a ``plateau'' $f(x) = \min\{1-|x|,0\}$. However, 
new ideas are required to tackle this case.
\end{remark}
At the same time, it is possible to obtain a qualitative stability result in full generality. Theorem \ref{thm:qualconf} holds in any conformal class on any surface except the sphere (which has been treated separately). The analog of comparison family is constructed using the techniques of \cite{KS}, and the proof of the stability follows the approach of \cite[Section 4]{KNPP2}. In particular, we show in Proposition \ref{convergence:prop} that   to  each conformally maximizing sequence of admissible measures we can associate a sequence of  appropriately balanced sphere-valued maps that converge strongly in $W^{1,2}(M,\mathbb{S}^n)$ to a $\bar \lambda_1$-maximal harmonic map. Along the way we simplify  some of the arguments of \cite{KNPP2}: the main novelty here is Proposition \ref{wk.lim} that allows to avoid using the language of quasi-open sets, which significantly shortens the proof. The global qualitative stability result, Theorem \ref{glob:qual:stab}, 
is a relatively straightforward consequence of Theorem \ref{thm:qualconf} and the results of \cite{Petrides, MS} establishing compactness of  the moduli space of  conformal classes with $\Lambda_1(M,\mC)$ sufficiently close to $\Lambda_1(M)$. 

The stability of the isoperimetric inequality for the second eigenvalue on the sphere (Theorem \ref{thm:stab2S2}), as well as more general qualitative and quantitative stability results for the second eigenvalue obtained in Section \ref{sec:higher} are proved using the same set of ideas. The main difference is that in this case we have to take into account the bubbling phenomenon,
which dictates a more careful  choice of spaces in which stability estimates are shown, see  discussion following Theorem \ref{thm:stab2S2}. Lemma~\ref{kokarev:lemma} is the key new ingredient. It can be seen as a generalization of a non-concentration estimate for
 $\lambda_1$~\cite{Girouard,Kokarev} and could be of independent interest.  
 As we have already noted, 
our approach is based on the characterization of the conformally maximal metrics via min-max energy, which was proved in \cite{KS}  only for $k=1,2$. Since Lemma \ref{main:lemma} holds for any $k \geqslant 1$, extending the results of \cite{KS} to $k >2$  will open the path to proving stability of isoperimetric inequalities for higher eigenvalues.
\subsection{Plan of the paper} The paper is organized as follows. In Section \ref{quantI:sec}  we obtain quantitative stability results for conformally maximal metrics.
In particular, we prove the quantitative stability of Hersch's inequality for the first eigenvalue on the sphere, and introduce the notion of stable comparison families.  Qualitative stability of conformally maximal metrics is investigated in Section \ref{qual:sec}.
In Section \ref{quantII:sec}  we revisit the quantitative stability, and show that conformally maximal metrics are quantitatively stable provided the corresponding conformally maximal harmonic maps do not admit non-trivial Jacobi fields.  In Section \ref{sec:higher}  we explore stability for the second Laplace eigenvalue and, in particular, investigate stability of bubbling sequences. Results on global quantitative stability for the first eigenvalue are obtained in Section \ref{sec:glob_stab}.  Finally, Section \ref{sec:sharp}  is concerned with the sharpness of the quantitative stability estimates and the optimality of the choice of the Sobolev space $W^{-1,2}$.
\subsection*{Acknowledgments}  Research of MK is partially supported by the NSF grant DMS-2104254. 
Research of MN is partially supported by ANR SHAPO (ANR-18-CE40-0013);  this paper is part of his Ph.D. research at the Universit\'e Savoie Mont Blanc under the supervision of Dorin 
Bucur.  Research of IP is  partially supported by NSERC and FRQNT.  Research of DS is partially supported by the 
NSF fellowship DMS-2002055.

\section{Quantitative stability of conformally maximal metrics I: comparison families}
\label{quantI:sec}
\subsection{Key lemma}
\label{S2stability:sec}
In this section we assume that $M$ is a closed surface with a fixed conformal class $\mC$ and a distinguished choice of the background metric $g\in \mC$. 
To simplify notation, throughout this section we set  $W^{1,2}(M):=W^{1,2}(M,g)$. 
Consider an admissible measure $\mu$ on $M$ with the corresponding map $T_\mu\colon W^{1,2}(M)\to L^2(\mu)$ and denote the $L^2(\mu)$-normalized eigenfunctions of $(M,\mC,\mu)$ by $\phi_0\equiv \frac{1}{\sqrt{\mu(M)}}$, $\phi_1$, $\ldots$ $\phi_k$, $\phi_i\in W^{1,2}(M)$. In what follows,   for $f\in W^{1,2}(M)$ we write $\int f\,d\mu$ instead of $\int T_\mu(f)\,d\mu$ whenever it does not cause confusion. This way
 the collection $\{\phi_0,\phi_1,\ldots\}$ can be considered an orthonormal basis of $L^2(\mu)$.

\begin{lemma}
\label{main:lemma}
Let $u\in W^{1,2}(M,\mathbb{S}^n)$ be a sphere-valued map such that
$$
\int_M \phi_j u\,d\mu=0\in \mathbb{R}^{n+1}\text{ for each }0\leqslant j\leqslant k-1
$$
for some $k\in \mathbb{N}$. Then
$$
2E(u)\geqslant \bar\lambda_k(M,\mC,\mu)
$$
and for any $v\in W^{1,2}(M,\mathbb{R}^{n+1})$,
\begin{equation}
\label{main:eq}
\int_M \langle du,dv\rangle\,dv_g-\lambda_k(M,\mC,\mu)\int_M\langle u,v\rangle\, d\mu \leqslant [2E(u)-\bar\lambda_k(M,\mC,\mu)]^{1/2}\|dv\|_{L^2(M)}.
\end{equation}
In particular, 
\begin{equation}
\label{main:eq3}
\| |du|^2_g\,dv_g - \lambda_k(M,\mC,\mu)\mu \|_{(C^0\cap W^{1,2}(M))^*}\leqslant \|u\|_{W^{1,2}(M)}[2E(u)-\bar\lambda_k(M,\mC,\mu)]^{1/2}
\end{equation}
Moreover, if $u\in W^{1,\infty}(M,\mathbb{S}^n)$, then
\begin{equation}
\label{main:eq2}
\| |du|^2_g\,dv_g - \lambda_k(M,\mC,\mu)\mu \|_{W^{-1,2}(M)} \leqslant \|u\|_{W^{1,\infty}(M)}[2E(u)-\bar\lambda_k(M,\mC,\mu)]^{1/2}
\end{equation}
\end{lemma}

\begin{remark} 
\label{rem:tensionfield}
A harmonic map is characterized by vanishing of its tension field.
The tension field $\tau(u)$ of a map $u\colon (M,g)\to\mathbb{S}^n$ is given by $\tau(u) = \Delta_gu - |du|_g^2u$. Combining~\eqref{main:eq} with~\eqref{main:eq3} yields
$$
\|\tau(u)\|_{(C^0\cap W^{1,2}(M))^*}\leqslant (1+\|u\|_{W^{1,2}(M)})[2E(u)-\bar\lambda_k(M,\mC,\mu)]^{1/2}
$$ 
Thus, if $|2E(u)-\bar\lambda_k(M,\mC,\mu)|$ is small, then the lemma implies that the tension field is small in $(C^0\cap W^{1,2}(M))^*$ and, therefore, one can think of $u$ as an ``almost'' harmonic map.
\end{remark}

\begin{proof} 
Denote by $Q_k\colon W^{1,2}(M,\mathbb{R}^{n+1})\times W^{1,2}(M,\mathbb{R}^{n+1})\to \mathbb{R}$ the quadratic form
$$
Q_k(v_1,v_2):=\int_M\langle dv_1,dv_2\rangle dv_g-\lambda_k(M,\mC,\mu)\int_M\langle v_1,v_2\rangle d\mu.
$$
Letting
$$
V_k:=\left\{v\in W^{1,2}(M,\mathbb{R}^{n+1}) \left| \, \int_M\phi_jvd\mu=0\in \mathbb{R}^{n+1}\text{ for }0\leqslant j\leqslant k-1 \right.\right\},
$$
it is clear from the definition of $\lambda_k(M,\mC,\mu)$ that $Q_k$ is positive semi-definite on $V_k$, and the Cauchy-Schwarz inequality therefore gives
$$
Q_k(v_1,v_2)\leqslant \sqrt{Q_k(v_1,v_1)}\sqrt{Q_k(v_2,v_2)}\text{ for all }v_1,v_2\in V_k.
$$
Since $u\in W^{1,2}(M,\mathbb{S}^n)$ lies in $V_k$ by assumption, for any $v\in W^{1,2}(M,\mathbb{R}^{n+1})$, decomposing $v$ as $v=v_0+v_1$ where $v_1\in V_k$ and
$$
v_0=\sum_{j=0}^{k-1}\langle \phi_j, v\rangle_{L^2(\mu)}\phi_j,
$$
we see that $Q_k(u,v_0)=0$, so that
$$
Q_k(u,v)=Q_k(u,v_1)\leqslant \sqrt{Q_k(u,u)}\sqrt{Q_k(v_1,v_1)}.
$$
In particular, noting that
$$
Q_k(u,u)=\int_M|du|^2dv_g-\lambda_k(M,\mC,\mu)\int_M |u|^2d\mu=2E(u)-\bar{\lambda}_k(M,\mC,\mu)
$$
(since $|u|\equiv 1$) and
$$
Q_k(v_1,v_1)\leqslant \int_M |dv_1|^2\,dv_g\leqslant \int_M |dv|^2\,dv_g,
$$
it follows that
$$
\int_M\langle du,dv\rangle\,dv_g-\lambda_k(M,\mC,\mu) \int_M\langle u,v\rangle d\mu\leq \sqrt{2E(u)-\bar\lambda_k (M,\mC,\mu)} \|dv\|_{L^2(M)},
$$
as claimed in~\eqref{main:eq}.

To obtain~\eqref{main:eq2}, for any $\varphi\in W^{1,2}(M)$ one sets $v = \varphi u$ in~\eqref{main:eq}, which is possible due to the fact that $u\in W^{1,\infty}$ implies $\varphi u\in W^{1,2}(M,\mathbb{R}^{n+1})$. Since $|u|^2=1$,~\eqref{main:eq} reads 
\begin{equation}
\label{vineq}
\int_M \langle d(\varphi u),d u\rangle\,dv_g - \lambda_1(M,\mC,\mu )\int_{M}\varphi\,d\mu\leqslant \sqrt{2E(u)-\bar\lambda_1(\mu)}\|d(\varphi u)\|_{L^2(M)}.
\end{equation}
For the left-hand side one has 
$$
\int_{M} \langle d(\varphi u),d u\rangle\,dv_g = \int_{M} \left(\varphi|du|_g^2 + \frac{1}{2}\langle d\varphi, d|u|^2\rangle\right)\,dv_g = \int_{M} \varphi|du|_g^2\,dv_g.
$$
For the right-hand side one has
\begin{equation*}
	\begin{split}
		\int_{M} |d(\varphi u)|^2_g\,dv_g &= \int_M |d \varphi|^2_g 
		+ \langle \varphi d(\varphi),d|u|^2\rangle + \varphi^2|du|^2_g\,dv_g\\
		&\leqslant \|u\|^2_{W^{1,\infty}(M)}\|\varphi\|^2_{W^{1,2}(M)}.
	\end{split}
\end{equation*}
Substituting these two expressions into~\eqref{vineq} yields~\eqref{main:eq2}

The proof of~\eqref{main:eq3} is similar. One sets $v=\psi u$, where $\psi\in C^0\cap W^{1,2}(M)$. It is easy to see $v\in W^{1,2}(M,\mathbb{S}^n)$. One has the inequality similar to~\eqref{vineq} with the l.h.s equal to the l.h.s. of~\eqref{main:eq3}. For the r.h.s. one has 
\begin{equation*}
	\begin{split}
		\int_{M} |d(\psi u)|^2_g\,dv_g &= \int_M |d \psi|^2_g |u|^2
		+ \langle \psi d(\psi),d|u|^2\rangle + \psi^2|du|^2_g\,dv_g\\
		&\leqslant \|u\|^2_{W^{1,2}(M)}\left(\|\psi\|^2_{C^0(M)} + \|d\psi\|^2_{L^2(M)}\right).
	\end{split}
\end{equation*}
\end{proof}

In the following Lemma~\ref{main:lemma} is applied in the situation, where $2E(u)$ is close to $\Lambda_k(M,\mC)$, so that the r.h.s. of~\eqref{main:eq} is small when the measure $\mu$ is almost $\bar\lambda_k$-conformally maximal. For general conformal classes, the existence of such maps $u$ satisfying the conditions of Lemma~\ref{main:lemma} is far from obvious. In Section~\ref{existence:sec} we show how the min-max characterization of~\cite{KS} can be applied to this problem. For now, we focus on the particular examples, where the existence can be shown directly.

\subsection{Proof of Theorem \ref{S2stability:thm}} 
\label{sec:S2proof}
 By Hersch's lemma, there exists a conformal automorphism $\Phi\colon\mathbb{S}^2\to\mathbb{S}^2$ such that
$$
\int_{\mathbb{S}^2} \Phi\,d\mu=0\in\mathbb{R}^3.
$$

Thus, Lemma~\ref{main:lemma} can be applied with $k=1$, $u=\Phi$, $g=g_0$.  The rest of the proof is almost identical to the arguments after~\eqref{vineq}.
Namely, since $\Phi$ is smooth, we can set $v = \varphi\Phi\in W^{1,2}(\mathbb{S}^2,\mathbb{R}^3)$. As a result, since $|\Phi|^2\equiv 1$ the inequality~\eqref{main:eq} reads
\begin{multline}
\label{S2ineq1:eq}
\int_{\mathbb{S}^2} \langle d(\varphi\Phi),d\Phi\rangle\,dv_g - \lambda_1(\mathbb{S}^2, \mu)\int_{\mathbb{S}^2}\varphi\,d\mu \\ \leqslant \sqrt{2E(\Phi)-\bar\lambda_1(\mathbb{S}^2, \mu)}\|d(\varphi\Phi)\|_{L^2(\mathbb{S}^2)}.
\end{multline}
Furthermore, $\Phi$ is conformal, so one has
$$
\int_{\mathbb{S}^2} \langle d(\varphi\Phi),d\Phi\rangle\,dv_g = \int_{\mathbb{S}^2} \varphi|d\Phi|_g^2\,dv_g = 2\int_{\mathbb{S}^2} \varphi dv_{\Phi^*g}.
$$
Similarly,
\begin{equation*}
	\begin{split}
		\int_{\mathbb{S}^2} |d(\varphi\Phi)|^2_g\,dv_g &= \int_{\mathbb{S}^2} |d\varphi|^2_g 
		+ \langle \varphi d(\varphi),d|\Phi|^2\rangle + \varphi^2|d\Phi|^2_g\,dv_g\\
		&= \int_{\mathbb{S}^2}|d\varphi|^2_{\Phi^*g} + 2\varphi^2\,dv_{\Phi^*g}.
	\end{split}
\end{equation*}
Substituting these two equalities into \eqref{S2ineq1:eq} yields
\begin{multline}
\label{S2stability:ineq1}
\langle \varphi,2dv_{\Phi^*g}-\lambda_1(\mathbb{S}^2, \mu)\mu\rangle \\ \leqslant \sqrt{8\pi-\bar\lambda_1( (\mathbb{S}^2, \mu)}\cdot \left(\int_{\mathbb{S}^2}(|d\varphi|_{\Phi^*g}^2+2\varphi^2)dv_{\Phi^*g}\right)^{1/2}
\end{multline}
Applying this inequality to both $\varphi$ and $-\varphi$ and taking into account the normalization $\lambda_1(\mathbb{S}^2, \mu)=2$ we arrive at  \eqref{eq:S2stabnew2}.
The estimate~\eqref{eq:S2stabnew1} is obtained by applying~\eqref{S2stability:ineq1} directly to $\Phi_*\mu$ and noticing that the components of the identity map have vanishing $\Phi_*\mu$-average.
\qed
\begin{remark}
\label{rem:normaliz}
As was mentioned in Remark \ref{rem:normaliz0}, one can equivalently normalize the measure $\mu$ by the area instead of the first eigenvalue. Indeed,  by Cauchy-Schwarz inequality, 
$$
\left| \langle \varphi, dv_{\Phi^*g} \rangle \right| \le \left(4\pi \int_{\mathbb{S}^2} \varphi^2 dv_{\Phi^*g}\right)^{1/2},
$$
and combining this with \eqref{S2stability:ineq1} we get that 
$$
\| \lambda_1(\mathbb{S}^2, \mu) \mu \|_{W^{-1,2}(\mathbb{S}^2)} \le c_1,
$$
for some explicitly computable constant $c_1$.  Hence, if $\mu$ has area $4\pi$, from estimate   \eqref{eq:S2stabnew2} and  the triangle inequality we obtain
\begin{multline*}
\| 2 \mu - 2 dv_{\Phi^*g}\|_{W^{-1,2}(\mathbb{S}^2)} \le \| 2 dv_{\Phi^*g} - \lambda_1(\mathbb{S}^2,\mu)\mu\|_{W^{-1,2}(\mathbb{S}^2)} +\\  (2-\lambda_1(\mathbb{S}^2,\mu))\|\mu\|_{W^{-1,2}(\mathbb{S}^2)} \le \frac{c_2}{\lambda_1(\mathbb{S}^2, \mu)} \sqrt{8\pi-\bar\lambda_1( (\mathbb{S}^2, \mu)}
\end{multline*}
for some other explicitly computable constant $c_2$. Note that  the new normalization was used in the last inequality. Since we are interested in the case when 
$\lambda_1(\mathbb{S}^2, \mu)$ is close to $2$, the denominator on the right-hand side may be essentially absorbed in the constant.
\end{remark}
\subsection{Comparison families} 
\label{quant_stability:sec}
In order to effectively apply Lemma~\ref{main:lemma} to the stability of $\bar\lambda_1$-maximal measures it is convenient to have an explicit family of maps in $W^{1,2}(M,\mathbb{S}^n)$, such that for any admissible measure $\mu$ there is a member of the family with vanishing $\mu$-average. This way, by inequality~\eqref{main:eq2} the stability of a general measure $\mu$ follows from stability properties of the explicit family of absolutely continuous measures. Theorem~\ref{S2stability:thm} is a particular example of this principle, where the family of maps consists entirely of $\bar\lambda_1$-maximal harmonic maps and, as a result, the argument is very straightforward. 

\begin{definition}
\label{def:comparfam}
Let $Z$ be a smooth manifold, possibly with boundary. We say that a family $F\in C^0(Z,W^{1,\infty}(M,\mathbb{S}^n))$ is a {\em comparison family} if 
\begin{enumerate}
\item for any admissible measure $\mu$ there exists $z$ such that $\int_M F_z\,d\mu=0$;
\item $\max\limits_{z\in Z} E_g(F_z) = \frac{1}{2}\Lambda_1(M,[g])$.
\end{enumerate}
\end{definition} 
In what follows, we take $Z=\mathbb{B}^{n+1}$.
\begin{example}[Canonical family]
\label{canon:ex}
Recall the assumptions of Theorem \ref{thm:quantstab}: let $(M,g)$ be such that there exists a branched minimal immersion $u\colon M \to \mathbb{S}^n$, $n\geqslant 3$ by the first eigenfunctions, such that its image is not contained in the equatorial $\mathbb{S}^2\subset\mathbb{S}^n$. Let $G_a(x) = \frac{1-|a|^2}{|x+a|^2}(x+a)+a$, $a\in\mathbb{B}^{n+1}$ be a conformal automorphism of the unit sphere $\mathbb{S}^n$. Then the {\em canonical family} $\{G_a\circ u\}_{a\in \mathbb{B}^{n+1}}$ is a comparison family. Indeed, property $1)$ follows from the Hersch's trick~\cite{Hersch, ESIconfvolume, LiYau}. Property $2)$ is a consequence of the fact that conformal automorphisms decrease the area of minimal surfaces~\cite{LiYau, ESIconfvolume}, \cite[Proposition 3.1]{CKM}, see also subsection \ref{subsec:disc} for a discussion.
\end{example}

To any comparison family $F$ one can associate the corresponding comparison family of measures $\{|dF_z|^2_g\,dv_g\}_{z\in Z}$. Note that by~\eqref{main:eq2} this family of measures has to contain all $\bar\lambda_1$-maximal measures. In particular, for each $\bar\lambda_1$-maximal measure, it has to contain at least one corresponding $\bar\lambda_1$-maximal harmonic map. 
Furthermore,~\eqref{main:eq2} yields a stability estimate in terms of the distance to the comparison family of measures. In order to obtain the stability of $\bar\lambda_1$-maximal measures the comparison family has to satisfy additional assumptions.  

\begin{definition}
\label{stable:def}
Let $F$ be a comparison family. Then $F$ is called {\em stable} if the $\bar\lambda_1$-maximal measures are stable in the family $\{|dF_z|^2_g\,dv_g\}_{z\in Z}$ in the following sense. There exists $C,\delta_0>0$ such that as soon as $\Lambda_1(M,[g])-2E(F_z)\leqslant \delta_0$ then $z\in K$ --- a compact subset of $Z$, and there exists a $\bar\lambda_1$-maximal harmonic map $u\in F$ satisfying
\begin{equation}
\label{stabledef:ineq}
\||du|^2_g\,dv_g-|dF_z|^2_g\,dv_g\|_{W^{-1,2}(M,g)}\leqslant C\sqrt{\Lambda_1(M,c)-2E(F_z)}.
\end{equation}
\end{definition}

The following theorem is a straightforward application of Lemma~\ref{main:lemma}.

\begin{theorem}
\label{stable:thm}
Let $c=[g]$ be a conformal class, such that there exists a stable comparison family. Then there exist $\delta_0, C>0$ such that if $\mu$ is an admissible measure satisfying $\Lambda_1(M,c)-\bar\lambda_1(M,c,\mu)<\delta_0$, then there exists a $\lambda_1$-maximal measure $\mu_0$ such that
$$
\|\lambda_1(M,c,\mu_0)\mu_0 - \lambda_1(M,c,\mu)\mu\|_{W^{-1,2}(M,g)}\leqslant C\sqrt{\Lambda_1(M,c)- \bar\lambda_1(M,c,\mu)}. 
$$
\end{theorem}
\begin{proof}
Let $F$ be a stable comparison family and let $\delta_0>0$, $K\Subset Z$ be as in Definition~\ref{stable:def}. Consider an admissible measure $\mu$ such that $\Lambda_1(M,c)-\bar\lambda_1(M,c,\mu)<\delta_0$ and let $z\in Z$ be such that
$\int_M F_z\,d\mu = 0$. By Lemma~\ref{main:lemma} and the definition of a comparison family one has $\bar\lambda_1(M,c,\mu)\leqslant 2E(F_z)\leqslant\Lambda_1(M,c)$. Therefore, $\Lambda_1(M,[g])-2E(F_z)\leqslant \delta_0$ and, in particular, $z\in K$. Thus, combining~\eqref{stabledef:ineq} with~\eqref{main:eq2}, one has that there exists a $\bar\lambda_1$-maximal measure $\mu_0$ such that 
\begin{equation*}
\begin{split}
&\|\lambda_1(\mu_0)\mu_0-\lambda_1(\mu)\mu\|_{W^{-1,2}}\leqslant \\
&\leqslant \|\lambda_1(\mu_0)\mu_0-|dF_z|^2_g\,dv_g\|_{W^{-1,2}}+\||dF_z|^2_g\,dv_g-\lambda_1(\mu)\mu\|_{W^{-1,2}} \leqslant\\
&\leqslant C\sqrt{\Lambda_1(M,c)-2E(F_z)} + \|F_z\|_{W^{1,\infty}}\sqrt{2E(F_z)-\bar\lambda_1(\mu)}\leqslant\\
&\leqslant 2(C+\max_{y\in K}\|F_y\|_{W^{1,\infty}})\sqrt{\Lambda_1(M,c)-\bar\lambda_1(\mu)}.
\end{split}
\end{equation*} 
In this computation we use a simplified notation $\lambda_1(\mu):=\lambda_1(M,\mC,\mu)$, and denote by $C$ different constants.
We also note that $\max_{y\in K}\|F_y\|_{W^{1,\infty}}$ can  be absorbed in the constant because $K$ is compact and $F_y \in W^{1,\infty}(M,\mathbb{S}^n)$,
see Definition \ref{def:comparfam}. This completes the proof of the theorem.
\end{proof}
In order to apply Theorem \ref{stable:thm} we need to have examples of comparison families. 
\begin{proposition}
\label{stable:prop}
The canonical family of Example~\ref{canon:ex} is a stable comparison family.
\end{proposition}
\begin{proof}
As we discussed in Example~\ref{canon:ex} the conformal automorphisms decrease the area. More precisely, the results in~\cite{CKM,ESIconfvolume} imply that the area strictly decreases in the radial direction, i.e. if $\xi\in\mathbb{S}^n$, then 
$\area(F_0(M))>\area(F_{s\xi}(M))>\area(F_{t\xi}(M))$ whenever $0<s<t<1$. Moreover, since the family $F_a:=\{G_a\circ u\}$ consists of conformal maps, area agrees with energy and, therefore, one has $\Lambda_1(M,c)=2E(F_0)>2E(F_{s\xi})>2E(F_{t\xi})$ for $0<s<t<1$. 

We first claim that $a=0$ is a non-degenerate critical point of $E(F_a)$. This fact is implicitly proved in~\cite[page 11]{CKM}; we sketch below a more direct proof for completeness. The explicit formula for $G_a$ yields
$$
E(F_a) = \frac{1}{2}\int_M \frac{(1-|a|^2)^2}{(|a|^2+2\langle F_0, a \rangle + 1)^2}|dF_0|_g^2\,dv_g.
$$
Thus, the Hessian of $E(F_a)$ at $a=0$ is the quadratic form
$$
H_0(v,v) = 4\int_M (3\langle v, F_0\rangle^2 - |v|^2)|dF_0|^2_g\,dv_g.
$$
Assume for now that $F_0$ does not have branch points. Let $h = F_0^*g_{\mathbb{S}^n}$, $h\in [g]$ since $F_0$ is conformal. Denote by $v^{\shortparallel}$ the projection of $v$ onto the tangent space of $F_0(M)$, then one has
\begin{equation*}
\begin{split}
&\int_M \langle v, F_0\rangle^2|dF_0|^2_g\,dv_g =\\
&\text{(since $F_0$ is harmonic $\Delta_g F_0 = |dF_0|_g^2 F_0$)}= \int_M(\Delta_g\langle v, F_0\rangle)\langle v, F_0\rangle\,dv_g = \\
&\text{(integrate by parts)} = \int_M \langle d\langle v, F_0\rangle,d\langle v, F_0\rangle \rangle_g\,dv_g = \\
&\text{(conformal invariance)} = \int_M \langle d\langle v, F_0\rangle,d\langle v, F_0\rangle \rangle_h\,dv_h = \\
&=\int_M |v^\shortparallel|^2\,dv_h = \frac{1}{2}\int_M |v^\shortparallel|^2|dF_0|^2_h\,dv_h = \frac{1}{2}\int_M |v^\shortparallel|^2|dF_0|^2_g\,dv_g.
\end{split}
\end{equation*}
Thus, 
\begin{equation*}
\begin{split}
H_0(v,v) &= 4\int_M (3\langle v, F_0\rangle^2 - |v|^2)|dF_0|^2_g\,dv_g =\\
&= 4\int_M(\langle v, F_0\rangle^2 + |v^\shortparallel|^2 - |v|^2)|dF_0|_g^2\,dv_g = -4\int_M|v^\perp|^2|dF_0|_g^2\,dv_g \le 0,
\end{split}
\end{equation*}
where $v^\perp$ is a projection of $v$ onto the normal bundle of $F_0(M)$ in $\mathbb{S}^n$. If $F_0$ has branch points, this formula still holds, see~\cite{CKM}. In this case, one can define the tangent plane (and, hence, normal bundle) at branch points as a limit of nearby tangent planes, see~\cite{GOR}. Finally, it is easy to see that $v^\perp \equiv 0$ for some $v$ implies that $F_0(M)$ is contained in an equatorial $\mathbb{S}^2$, see~\cite{CKM, ESIconfvolume}, which is ruled out by our assumptions. Hence, $H_0(v,v)<0$ and this completes the proof of the claim.

Since $a=0$ is a non-degenerate maximum, there exists $C>0$, $1/2>r_0>0$ such that $|a|<2r_0$ implies that $E(F_0)-E(F_a)\geqslant C|a|^2$. Set $\delta_0 = 2Cr^2_0$. This way if $|a|<2r_0$ and $\Lambda_1(M,[g])-2E(F_a)\leqslant \delta_0$, then $|a|<r_0$. Furthermore, since the energy decreases radially, we have that $\Lambda_1(M,[g])-2E(F_a)\leqslant \delta_0$ implies $|a|<r_0$, i.e. we can take $K=\{|a|\leqslant r_0\}$ in Definition~\ref{stable:def}.

To show~\eqref{stabledef:ineq} we use the explicit formula for $G_a$ to obtain 
\begin{equation*}
\begin{split}
&\left ||dF_0|_g^2 - |dF_a|_g^2\right| = \left|1-\frac{(1-|a|^2)^2}{|F_0+a|^4}\right| |dF_0|^2_g =\\
&=4\left| \frac{\langle F_0,a\rangle + \langle F_0,a\rangle^2 + \langle F_0,a\rangle|a|^2 + |a|^2}{|F_0+a|^4}\right||dF_0|_g^2\leqslant C'|a||dF_0|_g^2,
\end{split}
\end{equation*}
where in the last step we use $|a|<r_0<1/2$ and $|F_0|\equiv 1$. Finally, assume $\Lambda_1(M,[g])-2E(F_a)\leqslant \delta_0$. Then $|a|<r_0$, $\frac{1}{2} \Lambda_1(M,[g])-E(F_a)=E(F_0)-E(F_a)\geqslant C|a|^2$ and $\forall\varphi\in W^{1,2}(M,g)$ one has
\begin{equation*}
\begin{split}
&\left|\int_M\varphi\left (|dF_0|_g^2 - |dF_a|_g^2\right)\,dv_g\right|\leqslant C'|a|\int_M|\varphi||dF_0|_g^2\,dv_g\leqslant \\
&\leqslant \frac{C'}{\sqrt{2C}}\left(\int_M|dF_0|^4_g\,dv_g\right)^{1/2}\sqrt{\Lambda_1(M,c)-2E(F_a)}\|\varphi\|_{L^2(M,g)}.
\end{split}
\end{equation*}
In particular, after relabelling the constants one has
$$
\||dF_0|_g^2\,dv_g - |dF_a|_g^2\,dv_g\|_{W^{-1,2}(M,g)}\leqslant C\sqrt{\Lambda_1(M,c)-2E(F_a)},
$$
i.e. $\{F_a\}$ is a stable comparison family.
\end{proof}
\begin{proof}[Proof of Theorem \ref{thm:quantstab}] The result then immediately follows from Proposition \ref{stable:prop} and Theorem \ref{stable:thm}.
\end{proof} 

\begin{remark} Let $\Sigma_2$ be a surface of genus $2$ and let $\mC$ be a conformal class of the Bolza surface. Then the only branched minimal immersion by the first eigenfunctions is the hyperelliptic projection $\Pi\colon (\Sigma_2,c)\to \mathbb{S}^2$.  However, the corresponding canonical family is not a stable comparison family, since the energy of all maps in the family equals $8\pi$. Therefore, in order to establish a quantitative stability estimate for $\Lambda_1(\Sigma_2,c)$ one needs to come up with an alternative argument, cf. Remark \ref{rem:genustwo}.
\end{remark}


Consider the following modification of Example~\ref{canon:ex}. It turns out that for some harmonic (not necessarily minimal) maps $u$ the associated canonical family is a stable comparison family.
Let $\Gamma\subset \mathbb{R}^2$ be a rhombic lattice:  $\Gamma = \mathbb{Z}(1,0) \oplus \mathbb{Z}(c,d)$, where $0~\le~c~\le \frac{1}{2}$, $d>0$ and $c^2+d^2=~1$. Let $g_{c,d}$ be the corresponding flat metric of unit area on $\mathbb{T}^2$.  It was shown in  \cite{ESIR} that  for any $g \in g_{c,d}$ we have
\begin{equation}
\label{eq:rhombus}
\bar \lambda_1(\mathbb{T}^2, g) \le  \lambda_1(\mathbb{T}^2, g_{c,d})= \Lambda_1(\mathbb{T}^2,[g_{c,d}]).
\end{equation}
The following proposition holds.
\begin{proposition} 
\label{prop:rhombus}
 The isoperimetric inequality \eqref{eq:rhombus} is quantitatively stable.
\end{proposition}
\begin{proof}
 Consider the map $\Phi_{c,d}\colon (\mathbb{T}^2,[g_{c,d}])\to\mathbb{S}^3$ given by 
$$
\Phi_{c,d}(x,y) = \frac{1}{\sqrt{2}}\left(\sin\frac{2\pi y}{d}, \cos\frac{2\pi y}{d},\sin2\pi \left(\frac{c}{d}y - x\right),\cos2\pi \left(\frac{c}{d}y - x\right) \right).
$$
It is proven in~\cite{ESIR} that $\Phi_{c,d}$ is a $\bar\lambda_1$-maximal harmonic map. It is achieved by showing that the conformal automorphisms of $\mathbb{S}^3$ decrease the energy and, thus, the associated canonical family $\{G_a\circ\Phi_{c,d}\}$ is a comparison family. In particular, it is  shown that 
$$
\frac{E_{g_{c,d}}(G_a\circ \Phi_{c,d})}{E_{g_{0,1}}(G_a\circ \Phi_{0,1})} = \frac{E_{g_{c,d}}(\Phi_{c,d})}{E_{g_{0,1}}(\Phi_{0,1})}.
$$
Hence all the monotonicity properties of $E_{g_{0,1}}(G_a\circ \Phi_{0,1})$ also hold for $E_{g_{c,d}}(G_a\circ \Phi_{c,d})$. Since $\Phi_{0,1}$ is the minimal immersion of the Clifford torus, one has that $a=0$ is the non-degenerate global maximum of $E_{g_{c,d}}(G_a\circ \Phi_{c,d})$ and that $E_{g_{c,d}}(G_a\circ \Phi_{c,d})$ strictly decreases in the radial direction. One can then follow the proof of Proposition~\ref{stable:prop} to deduce that $\{G_a\circ\Phi_{c,d}\}$ is a stable comparison family and, as a result, the quantitative stability estimate holds for the inequality \eqref{eq:rhombus}. 
\end{proof}

\section{Qualitative stability of conformally maximal metrics}
\label{qual:sec}
\subsection{Min-max characterization}
\label{existence:sec}
The goal of this section is to prove the existence of suitable maps satisfying the conditions of Lemma~\ref{main:lemma}. Namely, we have the following proposition.

\begin{proposition}
\label{existence:prop}
Let $M$ be a surface with a fixed conformal class $\mC$. Then there exists $n>0$ such that for any admissible measure $\mu$ there is a map $u\in W^{1,2}(M,\mathbb{S}^n)$ such that 
$$
\int_M u\,d\mu = 0\in \mathbb{R}^{n+1}, 
$$
$$
\bar\lambda_1(M,\mC,\mu)\leqslant 2E(u)\leqslant \Lambda_1(M,\mC).
$$
\end{proposition}
In order to prove it we recall the min-max characterization of $\Lambda_1(M,\mC)$ obtained  in~\cite[Section 3.1]{KS}. 

For each $n\geqslant 2$ we define the collection $\Gamma_n(M)$ of continuous families 
$$
\overline{\mathbb{B}}^{n+1}\ni a\mapsto F_a\in W^{1,2}(M,\mathbb{R}^{n+1})\text{ such that } F_a\equiv a\text{ for }a\in\mathbb{S}^n. 
$$
For any $\varepsilon>0$ and any $u\in W^{1,2}(M,\mathbb{R}^{n+1})$ we further define the $\varepsilon$-energy
$$
E_\varepsilon(u) = \frac{1}{2}\left[\int_M|du|^2_g + \frac{1}{2\varepsilon^2}(1-|u|^2)^2\,dv_g \right]
$$
and the associated min-max quantities
$$
\mE_{n, \varepsilon}(M,g) = \inf_{F\in\Gamma_n}\sup_{a\in \overline{\mathbb{B}}^{n+1}} E_\varepsilon(F_a)
$$
$$
\mE_{n}(M,\mC) = \lim_{\varepsilon\to 0}\mE_{n,\varepsilon}(M,g).
$$
\begin{theorem}[~\cite{KS}] 
\label{KS:thm}
There exists $N=N(M,\mC)$ such that for all $n\geqslant N$ one has $2\mE_n(M,\mC)=\Lambda_1(M,\mC)$.
\end{theorem}

\begin{proof}[Proof of Proposition~\ref{existence:prop}]
Let $n$ be as in Theorem~\ref{KS:thm}. For each $\varepsilon>0$ let $F^\varepsilon\in\Gamma_n$ be such that $\sup_aE_\varepsilon(F^\varepsilon_a)\leqslant \mE_{n,\varepsilon}+\varepsilon$.
By Hersch's lemma there exists $a_\varepsilon\in\overline{\mathbb{B}}^{n+1}$ such that $u_\varepsilon:=F^\varepsilon_{a_\varepsilon}\in W^{1,2}(M,\mathbb{R}^{n+1})$ satisfies 
$$
\int_M u_\varepsilon\,d\mu=0, \quad E(u_\varepsilon)\leqslant E_\varepsilon(u_\varepsilon)\leqslant \mE_{n,\varepsilon}+\varepsilon\to \mE_n,
$$
$$
 \int_M|1-|u_\varepsilon||\,dv_g \leqslant \int_M|1-|u_\varepsilon|^2|\,dv_g\leqslant 2\varepsilon(\area(M,g) E_\varepsilon(u_\varepsilon))^{1/2}\to 0.
$$
Thus, there exists a subsequence $u_{\varepsilon_i}$ which converges weakly in $W^{1,2}(M,\mathbb{R}^{n+1})$ and strongly in $L^2(M,\mathbb{R}^{n+1})$ to $u\in W^{1,2}(M,\mathbb{R}^n)$. Since $\mu$ is admissible, the subsequence can be chosen to also converge strongly in $L^2(\mu)$ and, thus, 
\begin{equation}
\label{eq0}
\int_M u\,d\mu = 0.
\end{equation}
Since $|1-|u||\leqslant |1-|u_n||+|u_n-u|$, one also has
$$
\int_M|1-|u||\,dv_g = 0,
$$
and, therefore, $u\in W^{1,2}(M,\mathbb{S}^n)$. Since the energy is upper-semicontinuous with respect to weak convergence in $W^{1,2}$ one has
$$
E(u)\leqslant\lim_{\varepsilon\to 0}E(u_\varepsilon)\leqslant \mE_n = \frac{1}{2}\Lambda_1(M,\mC),
$$
where Theorem~\ref{KS:thm} was used in the last step. Finally, the inequality $2E(u)\geqslant \bar\lambda_1(M,c,\mu)$ follows from~\eqref{eq0} and Lemma~\ref{main:lemma}. 
\end{proof}

For the remainder of this section we study the behaviour of almost $\bar\lambda_1$-conformally maximal measures. Let $\mu_\delta$ be a conformally maximizing family of admissible measures, such that $\bar\lambda_1(M,\mC,\mu_\delta) =  \Lambda_1(M,\mC)-\delta^2$. By Proposition~\ref{existence:prop} there exist corresponding maps $u_\delta\in W^{1,2}(M,\mathbb{S}^n)$ with vanishing $\mu_\delta$-average and $\bar\lambda_1(M,\mC,\mu_\delta)\leqslant 2E(u_\delta)\leqslant\Lambda_1(M,\mC)$. First, we show that the sequence $u_\delta$ converges strongly in $W^{1,2}(M,\mathbb{S}^n)$ to a $\bar\lambda_1$-conformally maximal harmonic map $u$. Second, we show that the measures $\lambda_1(M,\mC,\mu_\delta)\mu_\delta$ converge in $W^{-1,2}$ to the $\bar\lambda_1$-conformally maximal measure $|du|_g^2\,dv_g$. This constitutes a qualitative stability result for $\Lambda_1(M,\mC)$.

\subsection{Convergence of $u_\delta$}
Our goal in this section is to prove the following.
\begin{proposition}
\label{ujconvergence:prop}
Let $M\ne\mathbb{S}^2$. Then, up to a choice of subsequence, $u_\delta$ converge strongly in $W^{1,2}(M,\mathbb{S}^n)$ to a $\bar\lambda_1$-conformally maximal harmonic map $u$.
\end{proposition}
The proof is heavily inspired by the ideas of~\cite{KNPP2} with only slight modifications. Moreover, we provide an additional observation which allows us to simplify some of the technical points in~\cite{KNPP2}, see Remark~\ref{simplification:rmk} below.

First of all, up to a choice of a subsequence $u_\delta$ converges weakly in $W^{1,2}(M,\mathbb{S}^n)$ and strongly in $L^2(M,\mathbb{S}^n)$ to some $u\in W^{1,2}(M,\mathbb{S}^n)$. Our goal is to show that the convergence is in fact strong and that $u$ is $\bar\lambda_1$-conformally maximal. We start with the following Lemma.
\begin{lemma}
\label{ujconvergence:lemma}
There exists a point $p\in M$ such that up to a choice of a subsequence for any compact $K\Subset M\setminus\{p\}$ the sequence $u_\delta|_K$ converges to $u|_K$ strongly in $W^{1,2}(K,\mathbb{S}^n)$.
\end{lemma}

We present two versions of the proof of this lemma. One is a direct application of the ideas in~\cite[Proposition 4.7]{KNPP2}, but it uses the lesser known notions of quasi-continuous representatives of Sobolev functions and quasi-open sets. The other is a variation of the first proof, where we show that it is possible to use approximation methods in order to avoid appealing to these lesser known concepts.

\subsubsection*{First proof of Lemma~\ref{ujconvergence:lemma}} 
By inequality~\eqref{main:eq} for any $v\in W^{1,2}(M,\mathbb{R}^{n+1})$ one has
\begin{equation}
\label{udelta:ineq}
\int_M\langle du,dv\rangle\,dv_g - \lambda_1(M,c,\mu_\delta)\int_M\langle u, v\rangle\,d\mu_\delta\leqslant \delta \|dv\|_{L^2(M)}
\end{equation}

In the remainder of this proof we identify $u_\delta$, $u$ with their quasi-continuous representatives, see~\cite[Section 3]{KNPP2}. 

\begin{lemma}
\label{lemma_H0}
Let $\Omega\subset M$ be open. Let $A\subset \Omega$ be quasi-open and let $\mu$ be an admissible measure on $M$. Suppose that $\phi_\delta\in W^{1,2}(\Omega)$ is such that $\forall f\in W^{1,2}_0(A)$
\begin{equation}
\label{quasimode:ineq}
\int_\Omega\langle d\phi_\delta,df\rangle\,dv_g - \int_\Omega \phi_\delta f\,d\mu\leqslant C\delta \|df\|_{L^2(M)},
\end{equation}
\begin{equation}
\label{positive:ineq}
\int_\Omega|df|_g^2\,dv_g - \int_\Omega f^2\,d\mu\geqslant 0.
\end{equation}
Then for any $\psi\in W^{1,2}(\Omega)$ such that $(\phi - \psi)\in W^{1,2}_0(A)$ one has
$$
\int_\Omega|d\psi|^2 - \int_\Omega\psi^2\,d\mu \geqslant \int_\Omega|d\phi|^2 - \int_\Omega \phi^2\,d\mu - C\delta||d(\phi - \psi)||_{L^2(M)}
$$ 
\end{lemma}
\begin{proof}
The inequality~\eqref{positive:ineq} implies that
\begin{equation}
\label{H0:eq1}
\int\limits_\Omega |d(\phi-\psi)|^2\,dv_g - \int_\Omega (\phi-\psi)^2\,d\mu\geqslant 0.
\end{equation}
Moreover, pairing up the inequality~\eqref{quasimode:ineq} with $\phi-\psi$, we obtain
$$
\int_\Omega \langle d\phi,d(\phi-\psi)\rangle\,dv_g - \int_\Omega \phi(\phi-\psi)\,d\mu \leqslant C\delta||d(\phi-\psi)||_{L^2(M)},
$$
or, equivalently,
\begin{equation}
\label{H0:eq3}
\int_\Omega \langle d\phi,d\psi\rangle\,dv_g - \int_\Omega \phi\psi\,d\mu \geqslant \int_\Omega|d\phi|^2\,dv_g -\int_\Omega\phi^2\,d\mu - C\delta||d(\phi-\psi)||_{L^2(M)}.
\end{equation}

Summing up~\eqref{H0:eq1} and two copies of~\eqref{H0:eq3} yields
\begin{equation*}
\begin{split}
\int_\Omega |d\phi|^2 + |d\psi|^2\,dv_g -& \int_\Omega (\phi^2+\psi^2)\,d\mu\geqslant \\
 &\geqslant 2\int_\Omega|d\phi|^2\,dv_g- \int_\Omega \phi^2\,d\mu -2C\delta||d(\phi-\psi)||_{L^2(M)}.
\end{split}
\end{equation*}
Rearranging the terms completes the proof.
\end{proof}

\begin{definition}
\label{def:good}
We say that the point $p$ is {\em good} if there exists an open neighbourhood $\Omega_p$ and a subsequence $\delta_m\to 0$ such that $\forall f\in W^{1,2}_0(\Omega_p)$
$$
\int_{\Omega_p}|df|^2\,dv_g - \lambda_1(M,c,\mu_{\delta_m})\int_{\Omega_p}f^2\,d\mu_{\delta_m}\geqslant 0.
$$ 
Otherwise, we say that the point $p$ is {\em bad}.
\end{definition}

The following is essentially~\cite[Proposition 4.6]{KNPP2}. The proof is the same.
\begin{proposition}
There is at most one bad point.
\end{proposition}

The following is an adaptation of~\cite[Proposition 4.7]{KNPP2}

\begin{proposition}
Let $p$ be a good point. Then there exists a neighbourhood $U_p$ and a sequence $\delta_m\to 0$ such that $u_{\delta_m}\to u$ strongly in $W^{1,2}(U_p)$.
\end{proposition}
\begin{proof}
The proof follows closely that of~\cite[Proposition 4.7]{KNPP2}. We outline the main steps below.  

Let $\delta_m$ be the sequence as in Definition~\ref{def:good} and set $u_{\delta_m}=u_m$, $\mu_{\delta_m} = \mu_m$. After a choice of a subsequence, the energy measures $|du_m|_g^2\, dv_g$ converge weakly in $(C^0)^*$ to a Radon measure $\nu$.
Arguing by contradiction, we may assume that $p$ is in the support of the defect measure $d\nu - |d u|_g^2 \, dv_g$. In particular, we can choose $B_r(p)\subset B_R(p)\subset\Omega_p$ and $\varepsilon>0$ such that  
$$
\nu(B_{r}(q)) - \int_{B_{r}}|d u|^2\,dv_g>2\varepsilon
$$
Furthermore, by passing to a subannulus of $A_{r,R}:=B_R(p)\setminus B_r(p)$ if necessary we may assume that for all large enough $m$ one has
\begin{equation}
\label{C1}
\int_{B_r}|d u_m|^2\,dv_g\geqslant \int_{B_R}|d u|^2\,dv_g + \varepsilon
\end{equation}
\begin{equation}
\label{C2}
\mu_m(A_{r,R})\leqslant \frac{\varepsilon}{9N\Lambda_1(M,c)},
\end{equation}
where $N$ was  defined in Theorem \ref{KS:thm}.

The rest of the proof is almost identical to~\cite[pp. 25-26, after (C4)]{KNPP2}. We apply~\cite[Corollary 3.12]{KNPP2} to each component $u^i_m$, $u^i$. This yields sequence $w^i_m$, such that $w^i_m \to u$ in $W^{1,2}(B_R)\cap L^\infty$ and $w^i_m > u^i_m$ on $\partial B_r$;  $w^i_m < u^i_m$ on $\partial B_R$. This allows us to apply~\cite[Lemma 3.10]{KNPP2} to get a quasi-open $B_r\subset A^i_m\subset B_R$ such that $w^i_m-u^i_m\in W^{1,2}_0(A^i_m)$, which in turn allows to apply Lemma~\ref{lemma_H0} with $A=A^i_m$, $\phi = u^i_m$, $\psi = w^i_m$ and $\mu = \lambda_1(\mu_m)\mu_m$. Note that the definition of a good point is chosen precisely  so that the non-negativity hypothesis~\eqref{positive:ineq} of Lemma~\ref{lemma_H0} is satisfied. All this results in the following inequality
$$
\int_{A^i_m}|dw^i_m|^2dv_g - \lambda_1(\mu_m)(w^i_m)^2\,d\mu_m\geqslant \int_{A^i_m}|du^i_m|^2dv_g - \lambda_1(\mu_m)(u^i_m)^2\,d\mu_m - C\delta_m,
$$
where in the last step we used that both $u^i_m$ and $w^i_m$ are uniformly bounded in $W^{1,2}\cap L^\infty$, so that the correction term involving $\|d(u^i_m-w^i_m)\|_{L^2(M)}$ can be absorbed in $C$. Rearranging the terms yields
$$
\int_{B_R}|dw^i_m|^2dv_g - \lambda_1(\mu_m)\int_{A^i_m} (w^i_m)^2- (u^i_m)^2\,d\mu_m\geqslant \int_{B_r}|du^i_m|^2dv_g  - C\delta_m
$$
Since $|(w_m^i)^2-(u_m^i)^2|$ is bounded in $L^\infty$ by $3$, by condition~\eqref{C2}, we can replace the domain of integration in the middle term by $B_R$ with a loss of at most $\frac{\varepsilon}{3N}$. Then, summing up for all $i$ yields
$$
\int_{B_R}|dw_m|^2dv_g - \lambda_1(\mu_m)\int_{B_R} |w_m|^2- |u_m|^2\,d\mu_m + \frac{\varepsilon}{3}\geqslant \int_{B_r}|du_m|^2dv_g  - C\delta_m
$$
Furthermore, since $w_m\to u$ in $W^{1,2}\cap L^\infty$ and $\delta_m\to 0$, so for large enough $m$ one has
$$
\int_{B_R}|du|^2dv_g - \lambda_1(\mu_m)\int_{B_R} |u|^2- |u_m|^2\,d\mu_m + \frac{2\varepsilon}{3}\geqslant \int_{B_r}|du_m|^2dv_g 
$$
Finally, recalling that $|u_m|^2= |u|^2=1$ we obtain that the integrand in middle term vanishes, so
$$
\int_{B_R}|du|^2dv_g + \frac{2\varepsilon}{3}\geqslant \int_{B_r}|du_m|^2dv_g, 
$$
which contradicts~\eqref{C1}.
\end{proof}

To complete the proof for a fixed compact $K$, one applies the previous proposition on the finite subcover of $K$. To show that there is a subsequence independent of the choice of $K$, one uses Cantor diagonal process on the countable compact exhaustion of $M\setminus\{p\}$.

\subsubsection*{Second proof of Lemma~\ref{ujconvergence:lemma}} The functions in $W^{1,2}(M)$ are only defined up to a measure zero set. As a result, in order to talk about fine properties of such functions one has to introduce the machinery of quasi-continuous representatives and quasi-open sets. If we could ensure that our maps $u_\delta$, $u$ are a priori continuous, there would be no need for quasi-continuous representatives and the quasi-open sets $A^i_m$ would simply be open (recall that quasi-open sets arise in the proof as superlevel sets of quasi-continuous functions). We show, in fact, that $u_\delta$ can be approximated by smooth maps and that $u$ is a priori smooth. 

Recall that by~\cite[Section 4]{SU} smooth sphere-valued maps are dense in $W^{1,2}(M^2,\mathbb{S}^n)$. Furthermore, since $u_\delta$ satisfy~\eqref{udelta:ineq} and the map $W^{1,2}\to L^2(\mu_\delta)$ is bounded it is possible to find $w_\delta\in C^\infty(M,\mathbb{S}^n)$ such that
\begin{equation}
\label{wdelta:ineq}
\int_M\langle dw_\delta,dv\rangle\,dv_g - \lambda_1(M,\mC,\mu_\delta)\int_M\langle w_\delta, v\rangle\,d\mu_\delta\leqslant 2\delta \|dv\|^2_{L^2(M)}
\end{equation}
and $\|u_\delta-w_\delta\|_{W^{1,2}}\to 0$ as $\delta\to 0$. Since~\eqref{udelta:ineq} is the only property of $u_\delta$ used in the first proof and the strong convergence of $u_\delta$ is equivalent to the strong convergence of $w_\delta$, it is sufficient to show Lemma~\eqref{ujconvergence:lemma} for smooth maps $w_\delta$.

It turns out that $u$ is automatically harmonic and, in particular, smooth. Indeed, since Lemma \ref{main:lemma} implies that the maps $u_{\delta}$ form a Palais--Smale-type sequence for the energy functional on sphere-valued maps, this should follow from a variant of the results in \cite{Be93}.
In our setting, however, it is possible to give a simple self-contained proof, via the following proposition.
\begin{proposition}
\label{wk.lim} 
Let $u_j \colon (M,g) \to \overline{\mathbb{B}}^{n+1}$ be a sequence of maps with $E_g(u_j)\leqslant E_0$ and  $\nu_j\ge 0$ be such that
\begin{equation}
\label{p.s}
Q(u_j)(v):=\int_M\langle du_j,dv\rangle\,dv_g-\nu_j\int_M \langle u_j,v\rangle d\mu_j\leq \epsilon_j \|dv\|_{L^2(M)}
\end{equation}
for all $v\in W^{1,2}(M,\mathbb{R}^{n+1})$, where $\epsilon_j\to 0$. Assume further that $u_j$ converge weakly in $W^{1,2}$ to a map $u\in W^{1,2}(M,\mathbb{S}^n)$. Then $u$ is harmonic.
\end{proposition}
\begin{proof} Given a map $w\colon (M,g)\to \overline{\mathbb{B}}^{n+1}$ and $1\leqslant a<b\leqslant n+1$, consider the $1$-form
$$
\alpha^{ab}(w):=w^adw^b-w^bdw^a.
$$
It's not hard to check that if $w$ is a map to the sphere, then $w$ is harmonic if and only if the codifferential $\delta_g(\alpha^{ab}(w))=0$ for all $a$ and $b$ \cite[Section 3.5]{Hel02}.

One has that $\forall\varphi\in C^\infty(M)$
\begin{eqnarray*}
\int_M\langle \alpha^{ab}(w),d\varphi\rangle\,dv_g&=&\int_M\langle dw^b, w^ad\varphi\rangle-\langle dw^a,w^bd\varphi\rangle\,dv_g\\
&=&\int_M\langle dw^b, d(w^a\varphi)\rangle-\langle dw^a, d(w^b\varphi)\rangle\,dv_g\\
&=&Q(w)(w^a\varphi e_b)-Q(w)(w^b\varphi e_a),
\end{eqnarray*}
where $e_a,e_b$ are the corresponding base vectors in $\mathbb{R}^{n+1}$. Here the notation $w^a\phi e_b$ is used to denote a function $M \to \mathbb{R}^{n+1}$ with the $b$-th component equal to $w^a\phi$ and all other components equal to zero.
Thus, for maps $u_j$ satisfying \eqref{p.s}, one has
\begin{equation}
\label{alpha.div}
\int_M\langle \alpha^{ab}(u_j),d\varphi\rangle\,dv_g\leqslant \epsilon_j\|d(\varphi u_j)\|_{L^2}\leqslant C\epsilon_j (\|d\varphi\|_{L^2}+\|\varphi\|_{C^0}\|du_j\|_{L^2}).
\end{equation}

Up to a choice of a subsequence one can additionally assume that the convergence $u_j\to u$ is strong in $L^2$. Then for the limiting map $u$ one has the following
\begin{equation*}
\begin{split}
&\int_M\langle u_j^adu_j^b-u^adu^b,d\varphi\rangle\,dv_g=\int_M\langle (u_j^a-u^a)du_j^b+u^a d(u_j^b-u^b),d\varphi\rangle\,dv_g =\\
&=\int_M (u_j^a-u^a)\langle du_j^b,d\varphi\rangle+(u_j^b-u^b)(u^a\Delta\varphi-\langle d\varphi, du^a\rangle)\,dv_g \leqslant \\
&\leqslant \|u_j-u\|_{L^2}(\|du_j\|_{L^2}\|d\varphi\|_{L^\infty}+\|du\|_{L^2}\|d\varphi\|_{L^\infty}+\|\Delta \varphi\|_{L^2})
\leqslant \\
&\leqslant \|u_j-u\|_{L^2}(2E_0^{1/2}\|d\varphi\|_{L^\infty}+\|\Delta \varphi\|_{L^2}),
\end{split}
\end{equation*}
and the right-hand side clearly vanishes as $j\to\infty$ for fixed $\varphi\in C^{\infty}(M)$, since $u_j\to u$ in $L^2$. 
Combining this with~\eqref{alpha.div}, it's clear that $\delta_g(\alpha^{ab}(w))=0$ for all $a,b$, so $u\colon M\to \mathbb{S}^n$ is indeed a harmonic map.
\end{proof}

The rest rest of the proof proceeds as before, only now all quasi-open sets can be assumed to be open.
\begin{remark}
\label{simplification:rmk}
Proposition~\ref{wk.lim} also allows to remove the necessity for quasi-open sets in~\cite{KNPP2}. There, one has a sequence of maps $\phi_{N_m,k}$ by eigenfunctions to the closed ball $\overline{\mathbb{B}}^{n+1}$ such the limiting map $\phi_k$ is a map to the sphere $\mathbb{S}^n$. In particular, elliptic regularity implies that the maps $\phi_{N_m,k}$ are continuous and the reason quasi-open sets appeared is that it is not clear that the limiting map $\phi_k$ is also continuous. However, Proposition~\ref{wk.lim} ensures that $\phi_k$ is smooth and, as a result, it allows for some simplifications in the proofs of~\cite[Section 4]{KNPP2}.
\end{remark}

\subsubsection*{Proof of Proposition~\ref{ujconvergence:prop}}

Again, up to a choice of a subsequence, we can assume that the measures $|du_\delta|^2_g\,dv_g$ converge in $*$-weak topology. Let us denote the limit by $\mu$. Lemma~\ref{ujconvergence:lemma} implies that 
$$
\mu = |du|^2_g\,dv_g + w\delta_p 
$$
for some $w\geqslant 0$. 

First, we claim that after a further choice of a subsequence $\nu_\delta:=\lambda_1(\mu_\delta)\mu_\delta\rightharpoonup^*\mu$. Indeed, assume that  
$\nu_\delta\rightharpoonup^*\nu$. Then since $u_\delta$ are uniformly bounded in $W^{1,2}(M,\mathbb{S}^n)$ the inequality~\eqref{main:eq3} implies that $\forall \varphi\in C^1(M)$ one has $\int_M\varphi\,d\nu = \int_M\varphi\,d\mu$. Then since $C^{1}(M)$ is dense in $C^0(M)$ for any $\psi\in C^0(M)$ there exists a sequence $\varphi_k\in C^1(M)$ such that $\varphi_k\to \psi$ in $C^0(M)$ and thus
$$
\int_M\psi\,d\mu = \lim_{k\to\infty}\int_M\varphi_k\,d\mu = \lim_{k\to\infty}\int_M\varphi_k\,d\nu = \int_M\psi\,d\nu,
$$
i.e. $\mu=\nu$.

Second, we claim that $w=0$. By~\cite[Proposition 1.1] {Kokarev} the eigenvalues are upper-semicontinuous with respect to $*$-weak convergence. Since $\nu_\delta(M)\to\mu(M)$ one has 
\begin{equation}
\label{evlimit:eq}
\bar\lambda_1(M,\mC,\mu)\geqslant \lim_{\delta\to 0}\bar\lambda_1(M,\mC,\nu_\delta) = \bar\lambda_1(M,\mC,\mu_\delta) = \Lambda_1(M,\mC).
\end{equation}
At the same time, if $|du|^2_g\not\equiv 0$ and $w\ne0$, then by~\cite[Lemma 2.1]{Kokarev} one has $\bar\lambda_1(M,\mC,\mu) = 0$, which contradicts~\eqref{evlimit:eq}. If $|du|^2_g\equiv 0$, then by~\cite[Lemma 3.1]{Kokarev} one has $\Lambda_1(M,\mC)=\lim_{\delta\to 0}\bar\lambda_1(M,\mC,\mu_\delta) = 8\pi$. This is in contradiction with Petrides' rigidity result~\cite{Petrides}, which states that for $M\ne\mathbb{S}^2$ one has $\Lambda_1(M,\mC)>8\pi$. Thus, the only remaining option is $w=0$.

The rest of the proof easily follows. If $w=0$, then $|du_\delta|^2_g\,dv_g\rightharpoonup^*|du|^2_g\,dv_g$, i.e. $u_\delta$ converges strongly in $W^{1,2}(M,\mathbb{S}^n)$ to $u$. By Proposition~\ref{wk.lim} $u$ is harmonic, therefore, $\mu$ is admissible and $\bar\lambda_1(M,\mC,\mu)\leqslant \Lambda_1(M,\mC,\mu)$. Combining with~\eqref{evlimit:eq} one has $\bar\lambda_1(M,\mC,\mu)= \Lambda_1(M,\mC)$, therefore, $\mu$ is $\bar\lambda_1$-conformally maximal and, as a result, $u$ is a $\bar\lambda_1$-conformally maximal harmonic map.

\subsection{Convergence of $\mu_\delta$}
\label{mudeltaconvergence:sec}
 We have already seen in the proof of Proposition~\ref{ujconvergence:prop} that $\mu_\delta$ *-weakly converge to a $\bar\lambda_1$-conformally maximal measure. In this section we show how to improve this convergence to a strong $W^{-1,2}$-convergence.

\begin{proposition} 
\label{convergence:prop}
Let $\mu_{\delta}$ be an admissible measure with
 $\bar\lambda_1(M,\mC,\mu_{\delta})\geqslant \Lambda_1-\delta^2$, and let $u_{\delta}\colon M\to \mathbb{S}^n$
 be a map with vanishing $\mu_\delta$-average such that $2E(u_\delta)\leqslant \Lambda_1(M,\mC)$.
 If $u\colon M\to \mathbb{S}^n$ is a $\bar\lambda_1$-conformally maximal harmonic map, then there exists $C,\delta_0>0$ depending only on $n$, $(M,\mC)$ and $\|du\|_{L^\infty}$ such that
$$
\|\lambda_1(M,c,\mu_\delta) \mu_{\delta}-|du|_g^2dv_g\|_{W^{-1,2}(M,g)}\leqslant C(\|u-u_{\delta}\|_{W^{1,2}(M,g)}+\delta).
$$
whenever $\|u-u_{\delta}\|_{W^{1,2}(M,g)}+\delta\leqslant \delta_0$.
\end{proposition}
\begin{remark}
\label{Linfty:rmk}
If $M\ne\mathbb{S}^2$, then the space of $\bar \lambda_1$-maximal maps $u$ is compact in $C^\infty$-topology \cite[Theorem E1]{Kokarev}  and, thus, the dependence on 
$\|du\|_{L^\infty}$ can be eliminated.
\end{remark}

\begin{proof} Since $u\colon M\to \mathbb{S}^n$ is harmonic, one has
$$
\int_M |du|^2\langle u,v\rangle\,dv_g=\int_M \langle du,dv\rangle\,dv_g
$$
for all maps $v$. Thus, it follows that
\begin{equation}
\label{star.bd}
\left|\int_M |du|_g^2\langle u,v\rangle dv_g-\int_M\langle du_{\delta},dv\rangle\,dv_g\right|\leqslant \|d(u-u_{\delta})\|_{L^2}\|dv\|_{L^2}.
\end{equation}
Combining this with Lemma \ref{main:lemma}, it follows that 
\begin{equation}
\label{pre.est}
\left|\int_M |du|_g^2\langle u,v\rangle dv_g-\lambda_1(\mu_\delta)\int_M \langle u_{\delta},v\rangle d\mu_{\delta}\right|\leqslant (\|d(u-u_{\delta})\|_{L^2}+\delta)\|dv\|_{L^2}.
\end{equation}
Let us define the average
$$
(u)_\delta\colon =\frac{1}{\mu_\delta(M)}\int_M u d\mu_{\delta}=\frac{1}{\mu_\delta(M)}\int_M (u-u_{\delta})d\mu_{\delta},
$$
Set $v=(u^i-u_{\delta}^i)u_{\delta}$, then  one has 
$$
\|dv\|^2_{L^2} \leqslant 2\left(\|d(u_{\delta}^i-u^i)\|_{L^2}^2+\|u_{\delta}^i-u^i\|_{L^{\infty}}\|du_{\delta}\|^2\right)\leqslant C,
$$
where we used $2E(u_\delta)\leqslant \Lambda_1(M,\mC)$ and $|u|^2 = |u_\delta|^2\equiv 1$. 
Similarly,
$$
\int_M |du|_g^2\langle u,v\rangle dv_g\leqslant C\|du\|^2_{L^\infty}\|(u^i-u_{\delta}^i)\|_{L^2}.
$$
Thus,~\eqref{pre.est} implies that once $\bar\lambda_1(\mu_\delta)\geqslant \frac{1}{2}\Lambda_1(M,\mC)$ one has
$$
|(u)_\delta^i|= \left|\frac{\lambda_1(\mu_\delta)}{\bar\lambda_1(\mu_\delta)}\int \langle u_{\delta},v\rangle d\mu_{\delta}\right|\leqslant C(\|(u^i-u_{\delta}^i)\|_{W^{1,2}} + \delta)
$$
We choose $\delta_0$ so that $|(u)_\delta|\leqslant \frac{1}{2}$ and $\bar\lambda_1(\mu_\delta)\geqslant \frac{1}{2}\Lambda_1(M,\mC)$.

Furthermore, for any $v\colon M\to \mathbb{R}^{n+1}$, one has 
\begin{eqnarray*}
\left|\int_M \langle u_{\delta}-(u-(u)_\delta),v\rangle d\mu_{\delta}\right|&=&
\left|\int \langle u_{\delta}-(u-(u)_\delta),v-(v)_\delta\rangle d\mu_\delta\right|\\
&\leqslant &\|u-(u)_\delta-u_{\delta}\|_{L^2(\mu_{\delta})}\|v-(v)_\delta\|_{L^2(\mu_{\delta})}\\
&\leqslant &\lambda_1(\mu_{\delta})^{-1}\|d(u-u_{\delta})\|_{L^2}\|dv\|_{L^2}.
\end{eqnarray*}
Thus, adding $\lambda_1(\mu_\delta)\int_M \langle u_{\delta}-(u-\bar{u}),v\rangle d\mu_{\delta}$ to the left-hand side of \eqref{pre.est}, we obtain the estimate
\begin{equation}
\label{final.est}
\left|\int_M |du|^2\langle u,v\rangle dv_g-\lambda_1(\mu_\delta)\int \langle (u-(u)_\delta),v\rangle d\mu_{\delta}\right|\leqslant C(\|u-u_{\delta}\|_{W^{1,2}}+\delta)\|dv\|_{L^2}
\end{equation}
for all $v$. 

Recall that $|(u)_\delta|\leqslant \frac{1}{2}$, so for any $\varphi \in W^{1,2}$ we can take
$$
v=\varphi \frac{u-(u)_\delta}{|u-(u)_\delta|^2}.
$$
Since $\frac{1}{2}\leqslant |u-(u)_\delta|\leqslant \frac{3}{2}$, we have $\|dv\|_{L^2}\leqslant C\|\varphi\|_{W^{1,2}}$.
As a result, inequality~\eqref{final.est} implies
\begin{equation}
\label{final.est2}
\left|\int_M |du|^2\frac{(1-\langle u,(u)_\delta\rangle)}{|u-(u)_\delta|^2}\varphi dv_g-\lambda_1(\mu_\delta)\int_M \varphi d\mu_{\delta}\right|\leqslant C(\|u-u_{\delta}\|_{W^{1,2}}+\delta)\|\varphi\|_{W^{1,2}}.
\end{equation}
Finally, since $\frac{1}{2}\leqslant |u-(u)_\delta|\leqslant \frac{3}{2}$ one has
$$
\left|\frac{1-\langle u,(u)_\delta\rangle}{|u-(u)_\delta|^2}-1\right|\leqslant C|(u)_\delta|\leqslant C(\|u-u_{\delta}\|_{W^{1,2}}+\delta).
$$
Applying this inequality in the first term on the l.h.s. of the inequality~\eqref{final.est2}, we deduce that
$$
\left|\int_M |du|^2\varphi dv_g-\lambda_1(\mu_\delta)\int\varphi d\mu_{\delta}\right|\leqslant C(\|u-u_{\delta}\|_{W^{1,2}}+\delta)\|\varphi\|_{W^{1,2}},
$$
as desired.
\end{proof}
\begin{proof}[Proof of Theorem \ref{thm:qualconf}]
Combining Proposition~\ref{convergence:prop} with Proposition~\ref{ujconvergence:prop} immediately yields the result.
\end{proof}


\section{Quantitative stability of conformally maximal metrics II: non-trivial Jacobi fields}
\label{quantII:sec}
\subsection{Slow convergence implies infinitesimal deformations}
We continue to study the conformally maximizing sequence $\mu_\delta$ of admissible measures satisfying 
$\bar\lambda_1(\mu_\delta)=\Lambda_1(M,\mC)-\delta^2$ and the corresponding sequence of maps $u_\delta\in W^{1,2}(M,\mathbb{S}^n)$ with zero $\mu_\delta$-average and $\bar\lambda_1(\mu_\delta)\leqslant 2E(u_\delta)\leqslant \Lambda_1(M,\mC)$. We already know that if $M\ne\mathbb{S}^2$, then up to a choice of a subsequence $u_\delta$ converge strongly in $W^{1,2}$ to a $\bar\lambda_1$-conformally maximal harmonic map $u\colon M\to\mathbb{S}^n$. By Proposition~\ref{convergence:prop}, the obstruction to the quantitative stability estimate is the slow convergence of $\|u_\delta-u\|_{W^{1,2}}$ to $0$. The goal of this section is to show that slow convergence implies the presence of non-trivial Jacobi fields along $u$.    

First of all, one has freedom in the choice of $u_\delta$. Up to the replacement $u_\delta\mapsto Au_\delta$, $A\in SO(n+1)$ we can assume that $A=id$ minimizes $\|u-Au_\delta\|_{L^2}$. This implies that for any $B\in\mathfrak so(n+1)$, 
\begin{equation}
\label{non-trivial:eq}
\int_M\langle Bu_\delta,u-u_\delta\rangle\,dv_g = 0.
\end{equation}

According to Lemma~\ref{main:lemma} one has the estimate
\begin{equation}
\label{pal.smal}
\left|\int_M\langle du_j,dv\rangle\, dv_g-\lambda_1(\mu_\delta)\int \langle u_j,v\rangle d\mu\right|\leqslant \delta\|dv\|_{L^2}
\end{equation}
for any $v\in W^{1,2}(M,\mathbb{R}^{n+1})$. 

Define 
$$
\beta_\delta^2:=\|u-u_\delta\|_{L^2}^2
$$
and note that
\begin{equation*}
\begin{split}
&\|d(u-u_\delta)\|_{L^2}^2=\int_M \left(|du|^2-2\langle du,du_\delta\rangle+|du_\delta|^2\right)\,dv_g =\\
&=\int_M \left( 2\langle du,d(u-u_\delta)\rangle+|du_\delta|^2-|du|^2\right)\,dv_g \leqslant 2\int\langle du,d(u-u_\delta)\rangle\,dv_g =\\
&=2\int_M |du|^2\langle u,u-u_\delta\rangle\,dv_g
=\int_M |du|^2|u-u_\delta|^2\,dv_g\leqslant \|du\|_{L^\infty}^2\beta_\delta^2.
\end{split}
\end{equation*}
Thus, setting 
$$
w_\delta:=\frac{u-u_\delta}{\beta_\delta},
$$
one has
$$
\|w_\delta\|_{L^2}=1\text{ and }\|w_\delta\|_{W^{1,2}}\leqslant C,
$$
where the constant $C$ does not depend on $u$ by Remark~\ref{Linfty:rmk}.
Furthermore, note that
\begin{eqnarray*}
\int_M |\langle u,w_\delta\rangle|\,dv_g&=&\beta_\delta^{-1}\int_M( 1-\langle u,u_\delta\rangle)\,dv_g\\
&=&\beta_\delta^{-1}\int_M\frac{1}{2}|u-u_\delta|^2\,dv_g=\frac{1}{2}\beta_\delta\to 0
\end{eqnarray*}
as $\delta \to 0$. Thus, passing to a subsequence, we find $w\in W^{1,2}(M,\mathbb{R}^{n+1})$ such that
$$
w_\delta\to w\text{ weakly in }W^{1,2}\text{, strongly in }L^2,\text{ and }\langle w,u\rangle \equiv 0.
$$

Now, suppose that the convergence $u_\delta\to u$ is slow, that is
$$
\lim_{\delta\to 0}\frac{\delta}{\beta_\delta}=0,
$$
and let $v\in W^{1,2}(M, \mathbb{R}^{n+1})$ be a variation field with $\langle v,u\rangle \equiv 0$. Recall that $v$ is called a Jacobi field along $u$ if $I_u(v,v')=0$ for all variation fields $v'\in C^\infty(M, \mathbb{R}^{n+1})$ with $\langle v',u\rangle \equiv 0$, where 
$$
I_u(v,w):=\int_M\langle dv,dw\rangle-|du|^2\langle v,w\rangle\,dv_g.
$$
We claim that under the slow convergence assumption $w$ is a Jacobi field along $u$. Indeed if $v\in C^\infty(M, \mathbb{R}^{n+1})$ one has 
\begin{equation*}
\begin{split}
&|I_u(w,v)|=\lim_{\delta\to 0}\beta_\delta^{-1}\left|\int_M\left(\langle d(u-u_\delta),dv\rangle-|du|^2\langle u-u_\delta,v\rangle\right)\,dv_g\right|=\\
&\text{($u$ is harmonic) }=\lim_{\delta\to 0}\beta_\delta^{-1}\left|\int_M \left(|du|^2\langle u_\delta,v\rangle-\langle du_\delta,dv\rangle\right)\,dv_g\right|\leqslant\\
&\text{(by \eqref{pal.smal}) }\leqslant \lim_{\delta\to 0}\left(\beta_\delta^{-1}\left|\int_M |du|^2\langle u_\delta,v\rangle dv_g-\lambda_1(\mu_\delta)\int_M \langle u_\delta,v\rangle d\mu_\delta\right|+\frac{\delta}{\beta_\delta}\|dv\|_{L^2}\right)=\\
&\text{(since $u\perp v$ and $\frac{\delta}{\beta_\delta}\to 0$)}=\lim_{\delta\to 0}\beta_\delta^{-1}\left|\int_M\langle u_\delta-u,v\rangle\left[|du|^2dv_g-\lambda_1(\mu_\delta)d\mu_\delta\right]\right|\leqslant\\
&\text{(since $v$ is smooth)}\leqslant \lim_{\delta\to 0}\||du|^2-\lambda_1(\mu_\delta)d\mu_\delta\|_{W^{-1,2}}\frac{\|u_\delta-u\|_{W^{1,2}}}{\beta_\delta}
=0,
\end{split}
\end{equation*}
since $|du|^2-\lambda_1(\mu_\delta)d\mu_\delta\to 0$ in $W^{-1,2}$ and $\|u_\delta-u\|_{W^{1,2}}\leqslant C\beta_\delta$.

For any harmonic map $u\colon M\to\mathbb{S}^n$ there are many trivial Jacobi fields generated by the rotations of the target sphere. Namely, if $B\in\mathfrak{so}(n+1)$, then $Bu$ is a Jacobi field along $u$. Nevertheless, by~\eqref{non-trivial:eq} one has $\forall B\in\mathfrak{so}(n+1)$
\begin{equation*}
\begin{split}
&\int_M\langle w,Bu\rangle\,dv_g = \lim_{\delta\to 0}\beta^{-1}_\delta \int_M\langle u-u_\delta,Bu\rangle\,dv_g =\\
&=\lim_{\delta\to 0}\beta^{-1}_\delta \int_M\langle u-u_\delta,B(u-u_\delta)\rangle + \langle u-u_\delta,Bu_\delta\rangle\,dv_g =\\
&\text{(since $B\in\mathfrak{so}(n+1)$)} = -\lim_{\delta\to 0}\beta^{-1}_\delta \int_M \langle B(u-u_\delta),u_\delta\rangle\,dv_g = \\
&\text{(by~\eqref{non-trivial:eq})} = \lim_{\delta\to 0}\beta^{-1}_\delta \int_M \langle u-u_\delta,u\rangle\,dv_g =  \int_M\langle w,u\rangle\,dv_g=0,
\end{split}
\end{equation*}
and hence $w$ is a non-trivial Jacobi field along $u$.
Combining these considerations with Proposition~\ref{convergence:prop} we arrive at the following.

\begin{proposition}
\label{Jacobi:prop}
Let $\mu_\delta$ be a sequence of admissible measures such that
$$
\bar\lambda_1(\mu_\delta)\geqslant \Lambda_1(M,\mC)-\delta^2
$$
and for some $\bar\lambda_1$-conformally maximal measure $\mu_0$ one has $\lambda_1(\mu_\delta)\mu_\delta\to \lambda_1(\mu_0)\mu_0$ strongly in $W^{-1,2}(M)$ as $\delta\to 0$. Then either
$$
\|\lambda_1(\mu_\delta)\mu_\delta-\lambda_1(\mu_0)\mu_0\|_{W^{-1,2}}\leqslant C\delta
$$
for some $C<\infty$, or there exists a $\bar\lambda_1$-conformally maximal harmonic map, such that $\lambda_1(\mu_0)\mu=|du|^2dv_g$ andt $u$ admits a nontrivial Jacobi field (i.e., a Jacobi field not of the form $Bu$ for $B\in \mathfrak{so}(n+1)$).
\end{proposition}

This yields the following quantitative stability estimate.
\begin{corollary}
\label{Jacobi:cor}
Let $M$ be a closed surface with a fixed conformal class $\mC$ and let $g\in \mC$ be a background metric. Assume that for any $n$ and any $\bar\lambda_1$-conformally maximal harmonic map $u\colon (M,\mC)\to\mathbb{S}^n$ there are no non-trivial Jacobi fields along $u$. Then there exist constants $C,\delta_0>0$ depending only on $M,\mC,g$ such that for any admissible measure $\mu$ satisfying $\Lambda_1(M,\mC)-\bar\lambda_1(M,\mC,\mu)\leqslant \delta_0$ there exists a $\bar\lambda_1$-conformally maximal measure $\mu_0$ such that 
$$
\|\lambda_1(M,\mC,\mu)\mu-\lambda_1(M,\mC,\mu_0)\mu_0\|_{W^{-1,2}(M,g)}\leqslant C\sqrt{\Lambda_1(M,\mC)-\bar\lambda_1(M,\mC,\mu)}.
$$ 
\end{corollary} 
\begin{proof}
Since the harmonic map $\mathrm{Id}\colon\mathbb{S}^2\to\mathbb{S}^2$ admits non-trivial Jacobi fields generated by conformal automorphisms, the conditions of the Corollary imply that $M\ne\mathbb{S}^2$.

Assuming the contrary to the conclusion, there exists a sequence of admissible measures $\mu_\delta$ such that $\Lambda_1(M,\mC)-\bar\lambda_1(M,\mC,\mu_\delta)\leqslant\delta^2$, but for any $\bar\lambda_1$-conformally maximal measure $\mu_0$ one has 
$$
\delta^{-1}\|\lambda_1(M,\mC,\mu_\delta)\mu_\delta-\lambda_1(M,\mC,\mu_0)\mu_0\|_{W^{-1,2}(M,g)}\to \infty.
$$
Let $n$ be as in Proposition~\ref{existence:prop}, and let $u_\delta\in W^{1,2}(M,\mathbb{S}^n)$ be maps satisfying the conclusion of Proposition~\ref{existence:prop} with respect to $\mu_{\delta}$. By Proposition~\ref{ujconvergence:prop} $u_\delta$ in turn converge (up to a choice of a subsequence) strongly in $W^{1,2}(M,\mathbb{S}^n)$ to a harmonic map $u\colon M\to\mathbb{S}^n$ of spectral index $1$ such that $2E(u)=\Lambda_1(M,\mC)$ and the measures $\lambda_1(M,\mC,\mu_\delta)\mu_\delta$ converge in $W^{-1,2}$ to a $\bar\lambda_1$-conformally maximal measure $\mu_0 = |du|^2\,dv_g$. Finally, the application of Proposition~\ref{Jacobi:prop} results in a contradiction.
\end{proof}

\subsection{Examples} The intuition behind Proposition~\ref{Jacobi:prop} is the following: if the convergence is slow, then the map $u$ is not the closest $\bar\lambda_1$-conformally maximal harmonic map to $u_\delta$, i.e. the problem occurs when there are other maximal maps in the vicinity of $u$. If one indeed knows that this is the case, then one could try to perturb $u$ to a nearby maximal map in order to improve the stability estimate. Unfortunately, the presence of Jacobi fields does not necessarily imply that there are other harmonic maps close to $u$. For that one would need to show that the Jacobi field is integrable, which, in general, is a notoriously difficult problem.
Nevertheless, the conditions of Corollary~\ref{Jacobi:cor} hold for some conformal classes, which we demonstrate below. Recall that a harmonic map $u\colon(M,\mC)\to\mathbb{S}^n$ is called linearly full if its image linearly spans $\mathbb{R}^{n+1}$.

\begin{example}
Let $M$ be a projective plane $\mathbb{RP}^2$ with its unique (up to a diffeomorphism) conformal class. According to~\cite{LiYau} one has $\Lambda_1(\mathbb{RP}^2)=12\pi$. Then by~\cite{EjiriEquivariant} the only linearly full harmonic map from $\mathbb{RP}^2$ to $\mathbb{S}^n$ of energy $6\pi$ is the Veronese immersion $v\colon\mathbb{RP}^2\to\mathbb{S}^4$. Consider the antipodal lift $\tilde v\colon \mathbb{S}^2\to\mathbb{S}^4$. By~\cite[Corollary 10]{MU}, the dimension of the space of normal Jacobi fields along $\tilde v$ is $14$. Adding to it the dimension of the space of tangential Jacobi fields, see e.g.~\cite[Section 7]{Moore}, one obtains that  $\nul_E(\tilde v)=20$; here $\nul_E$ denotes the energy nullity \cite[Definition 3.1]{KRP2}, which is precisely
the dimension of the space of all Jacobi fields along a corresponding harmonic map.  Applying~\cite[Theorem 3.19]{KRP2} one obtains that $\nul_E(v)=10$, which is the same as $\dim\mathfrak{so}(5)$, i.e. there are no non-trivial Jacobi fields along $v$. This settles the case of linearly full $\bar\lambda_1$-conformally maximal maps. Let $v_m\colon\colon \mathbb{RP}^2\to\mathbb{S}^{m}$ be a non-linearly full maximal map, then $m\geqslant 4$ and $v_m = j_{4,m}\circ v$, where $j_{n,m}\colon \mathbb{S}^n\to\mathbb{S}^m$ is a totally geodesic embedding. By an argument analogous to~\cite[Proposition 3.11]{KRP2} one has that $\nul_E(v_m) = \nul_E(v) + (m-4)\nul_S(v)$, where $\nul_S$ denotes the spectral nullity \cite[Definition 2.7]{KRP2}. Since in our case 
$\nul_S(v)={\rm mult}(\lambda_1)=5$, one has $\nul_E(v_m) = 10 + 5(m-4) = \dim\mathfrak{so}(m+1)-\dim\mathfrak{so}(m-4)$, which is precisely the dimension of trivial Jacobi fields along $v_m$.
\end{example}

We remark that it is also possible to verify that the conditions of Corollary~\ref{Jacobi:cor} are satisfied for the conformal class containing the metric induced by the immersion $\tilde\tau_{3,1}$. However, since we have already shown the stability estimate in this case using the methods of Section~\ref{quant_stability:sec}, the argument is omitted.

At the same time, there are conformal classes  where this condition does not hold. We have already mentioned that it does not hold for $\mathbb{S}^2$. 

\begin{example}
\label{Jacobi_torus:ex}
Consider a torus $\mathbb{T}^2$ endowed with a conformal class of a flat metric $g_{a,b}$ on $\mathbb{R}^2/\Gamma$, where $\Gamma = \mathbb{Z}(1,0) + \mathbb{Z}(a,b)$, $a^2+b^2=1$. Then according to~\cite{ESIR} the flat metric is $\bar\lambda_1$-conformally maximal. Consider the harmonic map $u\colon\mathbb{T}^2\to \mathbb{S}^2$ by first eigenfunctions given by $u(x,y) = (\sin(2\pi x),\cos(2\pi x), 0)$. Then $u$ is $\bar\lambda_1$-maximal. At the same time, $v(x,y) = (0,0,\sin(2\pi y))$ is a non-trivial Jacobi field along $u$.  
\end{example}

\begin{example}
\label{Jacobi_Bolza:ex}
Let $M=\Sigma_2$ be a genus $2$ surface and $c$ be a conformal class of a Bolza surface. Then according to~\cite{NayataniShoda} the hyperelliptic projection $\Pi\colon(\Sigma_2,c)\to\mathbb{S}^2$ is a $\bar\lambda_1$-conformally maximal map. The map $\Pi$ has $6$ branch points on $\mathbb{S}^2$ and, therefore, perturbing branch points results in a $12$-dimensional family of deformations of $\Pi$. Since the Teichm\"uller space of genus $2$ surfaces has real dimension $6$, there is a $6$-dimensional family of hyperelliptic projections $\Pi_{\alpha}\colon (\Sigma_2,c)\to\mathbb{S}^2$. This family gives rise to a $6$-dimensional space of Jacobi fields along $\Pi$. Since $\dim\mathfrak{so}(3)=3$, there exist non-trivial Jacobi fields along $\Pi$.
\end{example}

\begin{remark}[Jacobi fields vs comparison families]
\label{rem:Jaco}
One can observe that in comparison with Theorem~\ref{stable:thm}, Corollary~\ref{Jacobi:cor} does not currently add to the list of conformal classes where the stability estimate holds. Nevertheless, in theory Corollary~\ref{Jacobi:cor} is easier to apply. Indeed, one only needs to study the Hessian of $\bar\lambda_1$-conformally maximal harmonic maps, whereas in order to apply Theorem~\ref{stable:thm} one needs to construct an explicit family and then additionally study the hessian of $\bar\lambda_1$-conformally maximal harmonic maps restricted to this family. However, the big downside of Corollary~\ref{Jacobi:cor} is that it deals with all $\bar\lambda_1$-conformally maximal {\em harmonic maps}, whereas Theorem~\ref{stable:thm} deals with all $\bar\lambda_1$-conformally maximal {\em measures}. Thus, Corollary~\ref{Jacobi:cor} might not be applicable if there are many $\bar\lambda_1$-conformally maximal harmonic maps corresponding to the same $\bar\lambda_1$-conformally maximal measure, which is exactly the situation in Example~\ref{Jacobi_torus:ex}. However, this situation seems to be specific to surfaces of lower genus. If one additionally restricts oneself to conformal classes where the degenerate situation of Example~\ref{Jacobi_Bolza:ex} does not occur (there are many of those, see e.g.~\cite[Corollary 1.8]{MS}), then Open problem \ref{open:stab}
stated in the introduction appears to be natural.
\end{remark}
\begin{remark} Finally, we remark that if the hypothesis of Corollary~\ref{Jacobi:cor} is satisfied for the conformal class $\mC_0$, then it is also satisfied for all conformal classes $\mC$ sufficiently close to $\mC_0$ in $C^\infty$-topology. Indeed, suppose that $\mC_m\to\mC_0$ and $u_m\colon (M,\mC_m)\to\mathbb{S}^n$ are $\bar\lambda_1$-conformally maximal harmonic maps. Note that $n$ can be chosen independently of $m$, because the multiplicity of $\lambda_1$ is bounded only in terms of the topology of $M$, see e.g.~\cite{Nadirashvili_mult}. Then, similarly to the proof of Proposition~\ref{ujconvergence:prop}, one can show that up to a choice of a subsequence $u_m\to u$, where $u\colon (M,\mC_0)\to \mathbb{S}^n$ is a $\bar\lambda_1$-conformally maximal harmonic map. If $u_m$ admit non-trivial Jacobi fields, then their limit is a non-trivial Jacobi field along $u$.
\end{remark}



\section{Stability for the second eigenvalue}
\label{sec:higher}
\subsection{Stability for higher eigenvalues: new challenges}
In the present section we discuss how our arguments can be adapted to prove stability for $\Lambda_2(M,\mC)$. In fact, we only present the results for the second eigenvalue, since the min-max energy characterization is not yet proved for $\Lambda_k(M,\mC)$ with $k\geqslant 3$. However, it is the only obstruction to obtaining the stability estimates for higher eigenvalues.

The main difference between the theory for $\Lambda_1(M,\mC)$ and $\Lambda_k(M,\mC)$, $k >1$,  is that there are examples for which 
 the maximizing sequence does not converge to an admissible measure. Namely, there are at least two distinct types of $\bar\lambda_k$-conformally maximizing sequences of admissible measures:
\begin{itemize}
\item[(1)] {\em Regular sequence} that converges to an admissible $\bar\lambda_k$-conformally maximal measure. In this case one expects the same behaviour as for $k=1$, i.e. that the convergence is strong in $W^{-1,2}$-norm.
\item[(2)] {\em Bubbling sequence} that converges to $\mu = \mu_l + \sum_{i=1}^{k-l}w_i\delta_{p_i}$, where $l<k$ and $\mu_l$ is a $\bar\lambda_l$-conformally maximal measure. In this case one can not have strong $W^{-1,2}$ convergence since $\mu\not\in W^{-1,2}$. 
\end{itemize}  

In this section we verify that for $k=2$, up to a choice of a subsequence, there are no other maximizing sequences. While the optimal type of convergence for bubbling sequences is not immediately clear, we present the results for the *-weak convergence that appears to be natural in this setting. 

In the exposition below we omit the proofs that are analogous to the case $k=1$. 

\subsection{Stability for $\Lambda_2(\mathbb{S}^2)$}
\label{subsec:stab2S2}
The proof of stability for $\Lambda_1(\mathbb{S}^2)$ in Section~\ref{S2stability:sec} relies on Hersch's trick~\cite{Hersch}. Similarly, the proof of stability for $\Lambda_2(\mathbb{S}^2)$ relies on Nadirashvili's generalization of Hersch's lemma obtained in~\cite{NadirashviliS2}, see also~\cite{GNP09, PetridesS2, KS, Kim}. We describe Nadirashvili's construction below.

Let $\mZ$ be a collection of spherical caps in $\mathbb{S}^2$, i.e. non-empty intersections of $\mathbb{S}^2$ with affine half-spaces in $\mathbb {R}^3$. For each $Z\in \mZ$ there exists a conformal reflection $\tau_{\partial Z}$ across $\partial Z$  interchanging $Z$ and $\mathbb{S}^2\setminus Z$. We define $R_Z\colon \mathbb{S}^2\to Z$ to be 
\begin{equation*}
R_Z(x) = 
\begin{cases}
x,&\text{ if } x\in \mathbb{S}^2\setminus Z,\\
\tau_{\partial Z}(x), &\text{ if } x\in Z.
\end{cases}
\end{equation*}

Let $\mu$ be an admissible measure and let $\phi_1$ be its first eigenfunction. Then one can show that there exists $Z'\in\mZ$ and a conformal automorhism $\Phi$ of $\mathbb{S}^2$ such that
$$
\int_{\mathbb{S}^2}\Phi\circ R_{Z'}\,d\mu = \int_{\mathbb{S}^2}\phi_1(\Phi\circ R_{Z'})\,d\mu = 0\in\mathbb{R}^3.
$$
Equivalently, one can phrase these conditions in terms of the measure $\Phi_*\mu$. Indeed, the first eigenfunction of $\Phi_*\mu$ is $(\Phi^{-1})^*\phi_1$ and one has
$$
0=\int_{\mathbb{S}^2}\Phi\circ R_{Z'}\,d\mu = \int_{\mathbb{S}^2}(\Phi\circ R_{Z'}\circ \Phi^{-1})\,\Phi_*d\mu,
$$
$$
0 = \int_{\mathbb{S}^2}\phi_1(\Phi\circ R_{Z'})\,d\mu = \int_{\mathbb{S}^2}((\Phi^{-1})^*\phi_1)(\Phi\circ R_{Z'}\circ\Phi^{-1})\,\Phi_*d\mu.
$$
Thus, the condition of Lemma~\ref{main:lemma} are satisfied for $\Phi_*\mu$ with $u=R_{Z}: = \Phi\circ R_{Z'}\circ \Phi^{-1}$.
\begin{proposition} 
\label{S2stability2:prop}
Let $\mu$ be an admissible measure on $\mathbb{S}^2$ and $g$ be a round metric on $\mathbb{S}^2$. Then there exists a conformal automorphism $\Phi\colon\mathbb{S}^2\to\mathbb{S}^2$ and a spherical cap $Z$ satisfying 
$$
\area(Z)\leqslant 4\pi-\frac{1}{4}\bar\lambda_2(\mu)
$$
such that the measure $\nu_Z = |dR_Z|^2_gdv_g$ satisfies
$$
\||dR_Z|^2_g\,dv_g - \lambda_2(\mu)\Phi_*\mu\|_{(C^0\cap W^{1,2}(\mathbb{S}^2))^*}\leqslant (12\pi)^{1/2}\sqrt{16\pi-\bar\lambda_2(\mu)}
$$
In particular, if $p\in \mathbb{S}^2$ is the center of the spherical cap $Z$, one has that 
\begin{equation}
\|2dv_g+8\pi \delta_p-\lambda_2(\mu)\Phi_*\mu\|_{(C^1(\mathbb{S}^2))^*}\leq C_1\sqrt{16\pi-\bar\lambda_2(\mu)}
\end{equation}
with an explicit constant $C_1$.
\end{proposition}
\begin{proof} By the preceding discussion, we can find a conformally equivalent measure $\tilde{\mu}=\Phi_*\mu$ and a spherical cap $Z$ such that $\tilde{\mu}$ satisfies the conclusions of Lemma \ref{main:lemma} for $k=2$ with respect to the map $u=R_Z$.

Thus, the inequality~\eqref{main:eq3} implies that 
$$
\||dR_Z|^2_g\,dv_g - \lambda_2(\mu)\Phi_*\mu\|_{(C^0\cap W^{1,2}(\mathbb{S}^2))^*}\leqslant \|R_Z\|_{W^{1,2}(\mathbb{S}^2)}\sqrt{2E(R_Z)-\bar\lambda_2(\mu)}.
$$
It remains to note that $2E(R_Z)\leqslant 16\pi$ and $\|R_Z\|_{W^{1,2}}\leqslant 12\pi$. More precisely,
$$
2E(R_Z) =2\int_{\mathbb{S}^2\setminus Z}\,dv_g + \int_{Z}|d\tau_{\partial Z}|_g^2\,dv_g = 4\area (\mathbb{S}^2\setminus Z) = 16\pi - 4\area(Z).
$$
Since $\bar\lambda_2(\mu)\leqslant 2E(u)$, one arrives at $\area(Z)\leqslant 4\pi-\frac{1}{4}\bar\lambda_2(\mu)$.

To prove the final $(C^1(\mathbb{S}^2))^*$ estimate, we may assume that $16\pi-\bar\lambda_2(\mu)\leqslant 8\pi$, so that the spherical cap $Z$ is a geodesic disk $D_r(p)$ of area
$$
\area(D_r(p))\leqslant 4\pi-\frac{1}{4}\bar\lambda_2(\mu)\leqslant 2\pi,
$$
so that the radius $r\leq \pi$ satisfies an estimate of the form
\begin{equation}
\label{concentration_rate:eq}
r\leqslant c_0\sqrt{\area(Z)}\leqslant c_0\sqrt{16\pi-\bar\lambda_2(\mu)}.
\end{equation}
Then for any $\varphi \in C^1(S^2)$ one has that
$$
\left|\int_{\mathbb{S}^2\setminus Z} \varphi|dR_Z|_g^2\,dv_g - 2\int_{\mathbb{S}^2} \varphi\,dv_g \right|= 2\left|\int_Z\varphi\,dv_g\right|\leqslant 2\|\varphi\|_{C^0}\area(Z);
$$
\begin{equation*}
\begin{split}
& \left|\int_{Z} \varphi|dR_Z|_g^2\,dv_g - 8\pi\varphi(p)\right| \leqslant \int_{Z} |\varphi-\varphi(p)||dR_Z|_g^2\,dv_g +
2\area(Z)|\varphi(p)|\leqslant\\
&\leqslant r\|\varphi\|_{C^1}\int_C|dR_Z|_g^2\,dv_g + 2\area(Z)\|\varphi\|_{C^0}\leqslant C_2\sqrt{16\pi-\bar\lambda_2(\mu)}\|\varphi\|_{C^1}
\end{split}
\end{equation*}
Summing up the two inequalities yields the desired bound.
\end{proof}
\begin{proof}[Proof of Theorem \ref{thm:stab2S2}]
The result follows immediately from Proposition \ref{S2stability2:prop} in view of the normalization $\lambda_2(\mu):=\lambda_2(\mathbb{S}^2,\mu)=2$.
\end{proof}
We remark that it is possible to develop the theory of stable comparison families for $\Lambda_2(M,\mC)$ following Section~\ref{quant_stability:sec}. Unfortunately, we do not have any examples of such families apart from the Nadirashvili's family used to prove Proposition~\ref{S2stability2:prop}, which is already used to obtain quantitative stability for $\Lambda_2(\mathbb{S}^2)$. As a result, the details of these definitions are omitted.

\subsection{Min-max characterization}
Since the min-max characterization of $\Lambda_2(M,\mC)$ is proved in~\cite{KS} the proof of the following proposition is the same as the proof of Proposition~\ref{existence:prop}.

\begin{proposition}
\label{existence2:prop}
Let $M$ be a surface with a fixed conformal class $\mC$ and let $\mu$ be an admissible measure on $M$. Then there exists $n>0$ such that for any admissible measure $\mu$ and any function $\psi\in W^{1,2}(M)$ there is a map $u\in W^{1,2}(M,\mathbb{S}^n)$ such that 
$$
\int_M u\,d\mu = \int_M \psi u\,d\mu = 0\in \mathbb{R}^{n+1}
$$
and $2E(u)\leqslant \Lambda_2(M,\mC)$. In particular, if $\psi$ is chosen to be  the $\lambda_1(M,\mC,\mu)$-eigenfunction, then
$$
\bar\lambda_2(M,\mC,\mu)\leqslant 2E(u)\leqslant \Lambda_2(M,\mC).
$$
\end{proposition}

\subsection{Convergence of the maps} Let $\mu_\delta$ be a sequence of admissible measures such that $\bar\lambda_2(M,\mC,\mu_\delta)\to \Lambda_2(M,\mC)$ as $\delta\to 0$. Then by Proposition~\ref{existence2:prop} there exist maps $u_\delta\colon (M,\mC)\to\mathbb{S}^n$ such that 
$$
\int_M u_\delta\,d\mu_\delta = \int_M \psi_\delta u_\delta\,d\mu_\delta = 0\in \mathbb{R}^{n+1},
$$
where $\psi_\delta$ is any  $\lambda_1(M,\mC,\mu_\delta)$-eigenfunction. As a result, Lemma~\ref{main:lemma} can be applied with $k=2$.

\begin{proposition}
\label{ujconvergence2:prop}
Let $M\ne\mathbb{S}^2$. Then up to a choice of a subsequence one of the following holds.
\begin{enumerate}
\item $u_\delta$ converge strongly in $W^{1,2}(M,\mathbb{S}^n)$ to a $\bar\lambda_2$-maximal harmonic map $u$.
\item There exists a point $p\in M$ such that $u_\delta|_K$ converge strongly in $W^{1,2}(K,\mathbb{S}^n)$ to a $\bar\lambda_1$-maximal harmonic map $u|_K$  for any compact $K\Subset M\setminus\{p\}$. Furthermore, $\lim_{\delta\to 0} E(u_\delta) = E(u) + 4\pi$.
\end{enumerate}
In particular, if $\Lambda_2(M,\mC)>\Lambda_1(M,\mC)+8\pi$, then assertion $(1)$ holds.
\end{proposition}
The following result is analogous to Lemma \ref{ujconvergence:lemma}.
\begin{lemma}
\label{ujconvergence2:lemma}
There exist two points $p_1,p_2\in M$ such that up to a choice of a subsequence for any compact $K\Subset M\setminus\{p_1,p_2\}$ the sequence $u_\delta|_K$ converges to $u|_K$ strongly in $W^{1,2}(K,\mathbb{S}^n)$.
\end{lemma}

\begin{proof}[Proof of Proposition~\ref{ujconvergence2:prop}]
Up to a choice of a subsequence, we can assume that $|du_\delta|^2_g\,dv_g\rightharpoonup^*\mu$. Lemma~\ref{ujconvergence2:lemma} implies that 
$$
\mu = |du|^2_g\,dv_g + w_1\delta_{p_1} + w_2\delta_{p_2} 
$$
for some $w_1,w_2\geqslant 0$. Similarly to the proof of Proposition~\ref{ujconvergence:prop} one has $\nu_\delta:=\lambda_2(M,\mC,\mu_\delta)\mu_\delta \rightharpoonup^*\mu$.

The following lemma is a generalization of~\cite[Lemma 2.1, 3.1]{Kokarev}.

\begin{lemma}
\label{kokarev:lemma}
Let $\mu_i$ be a sequence of admissible measures such that 
$$
\mu_i\rightharpoonup^*\mu = \mu_r + w_1\delta_{p_1} + w_2\delta_{p_2} ,
$$
where $\mu_r$ is a smooth absolutely continuous measure.
\begin{itemize}
\item If $w_1,w_2,\mu_r\ne 0$, then
$$
\limsup\bar\lambda_2(M,\mC,\mu_i) = 0.
$$
\item If $\mu_r=0$, $w_1\ne 0$ then
$$
\limsup\bar\lambda_2(M,\mC,\mu_i) \leqslant  16\pi.
$$
\item If $\mu_r\ne 0$, $w_1\ne 0$, $w_2=0$, then 
$$
\limsup\lambda_2(M,\mC,\mu_i)\leqslant \min\left\{\lambda_1(M,c,\mu_r), \frac{8\pi}{w_1}\right\}.
$$
In particular,
\begin{equation}
\label{kokarev:ineq1}
\limsup\bar\lambda_2(M,\mC,\mu_i) \leqslant  8\pi + \bar\lambda_1(M,\mC,\mu_r)
\end{equation}
with equality only if $8\pi = w_1\lambda_1(M,\mC,\mu_r)$.
\end{itemize}
\end{lemma}

We postpone the proof of the lemma until the end of the section. Let us use it to finish the proof of Proposition~\ref{ujconvergence:prop}. By Proposition~\ref{wk.lim} the map $u$ is harmonic, i.e. it is smooth. Thus, Lemma~\ref{kokarev:lemma} can be applied to the sequence $\nu_\delta$. Since Petrides' bound $\Lambda_1(M,\mC)>8\pi$ (see~\cite{Petrides}) implies that $\Lambda_2(M,\mC)>16\pi$ and $\bar\lambda_2(M,\mC,\nu_\delta)\to\Lambda_2(M,\mC)>16\pi$, Lemma~\ref{kokarev:lemma} implies that $u$ is not constant and at least one of $w_1,w_2$ is zero. Thus, there are two cases.

{\bf Case 1:} $w_1=w_2=0$. Then $u_\delta$ converge to $u$ strongly in $W^{1,2}(M,\mathbb{S}^n)$. By upper-semicontinuity of the eigenvalues one has $\bar\lambda_2(M,\mC,|du|^2_g\,dv_g) = \Lambda_2(M,\mC)$ and, thus, $u$ is a $\bar\lambda_2$-conformally maximal harmonic map.

{\bf Case 2:} $w_2=0$, $w_1\ne 0$ and $u$ is not constant. Then by Lemma~\ref{kokarev:lemma} one has
$$
\Lambda_2(M,\mC)\leqslant \bar\lambda_1(M,\mC,|du|^2_g\,dv_g) + 8\pi\leqslant \Lambda_1(M,\mC)+8\pi\leqslant \Lambda_2(M,\mC),$$
where the last inequality is a well-known property of conformal eigenvalues, see e.g.~\cite{CE}. In particular, one has $\bar\lambda_1(M,\mC,|du|^2_g\,dv_g)=\Lambda_1(M,\mC)$, i.e. $u$ is $\bar\lambda_1$-conformally maximal harmonic map; $\Lambda_1(M,\mC) + 8\pi = \Lambda_2(M,\mC)$; and the equality in~\eqref{kokarev:ineq1} holds. Thus, 
$8\pi = w_1\lambda_1(M,\mC,|du|^2_g\,dv_g) = w_1$, i.e. $w_1=8\pi$. Since $E(u)$ is half the total measure of $|du|^2_g\,dv_g$, the equality $\lim E(u_\delta) = E(u)+4\pi$ follows.
\end{proof}

\begin{proof}[Proof of Lemma~\ref{kokarev:lemma}]
The idea is to construct test-functions associated to the three components of $\mu$, such that for different components these test-function have disjoint support. 

For each point $p_j$, $j=1,2$ there exists a neighbourhood of $p_j$ and a conformally flat metric $g_j$ defined in this neighbourhood. In the following $B_r(p_j)$ denotes the (open) ball around $p_j$ of radius $r$ in the metric $g_j$. It is defined for $r$ small enough. It is easy to see that up to a choice of a subsequence there exist $\gamma^j_i,\varepsilon^j_i\to 0$ such that $\dfrac{\gamma^j_i}{\varepsilon^j_i}\to 0$,
$$
\mu_i\left(B_{\varepsilon^j_i}(p_j)\setminus B_{\gamma^j_i}(p_j)\right) = o(1);\, \mu_i\left(B_{\gamma^j_i}(p_j)\right) = w_j + o(1);\, \mu_r\left(B_{\varepsilon^j_i}(p_j)\right) = o(1)
$$
and for a fixed $i_0$ one has
\begin{equation}
\label{kineq1}
\mu_i\left( B_{\varepsilon^j_{i_0}}(p_j)\setminus B_{\gamma^j_i}(p_j) \right)\to \mu_r\left(B_{\varepsilon^j_{i_0}}(p_j)\right).
\end{equation}
This is a simple measure-theoretic fact, for the proof see e.g.~\cite[Lemma 5.1]{KNPP2}. Set $\delta^j_i = \sqrt{\gamma^j_i\varepsilon^j_i}$, then $\frac{\gamma^j_i}{\delta^j_i} = \frac{\delta^j_i}{\varepsilon^j_i}\to 0$. In particular, there exist cut-off functions $0\leqslant\psi^j_i\leqslant 1$, $j=1,2,r$ such that $\|d\psi^j_i\|_{L^2}\to 0$ and for $j=r$
\begin{equation*}
\psi^r_i = 
\begin{cases}
1\,\,\,\text{on  } M\setminus (B_{\varepsilon^1_i}(p_1)\cup B_{\varepsilon^2_i}(p_2)),\\
0\,\,\,\text{on  } B_{\delta^1_i}(p_1)\cup B_{\delta^2_i}(p_2),
\end{cases}
\end{equation*}
and for $j=1,2$
\begin{equation*}
\psi^j_i = 
\begin{cases}
1\,\,\,\text{on  } B_{\gamma^j_i}(p_j), \\
0\,\,\,\text{on  } M\setminus B_{\delta^j_i}(p_j).
\end{cases}
\end{equation*}
These cut-offs have disjoint support and serve as multipliers for our test-functions.

{\bf Test-functions associated to $\mu_r$.} Suppose that $\mu_r\ne 0$. Let $V^r_k$ be a direct sum of the eigenspaces corresponding to $\lambda_m(M,\mC,\mu_r)$, $m=0,\ldots, k$. Since $\mu_r$ is smooth and $V^r_k$ is finite-dimensional, there exists $C_m>0$ such that for any $f\in V^r_k$ one has $\|f\|_{L^\infty}\leqslant C_m\|f\|_{L^2(M,\mu_r)}$. Consider a function $\psi^r_i f$, then one has
\begin{equation*}
\begin{split}
&\|d(\psi^r_if)\|^2_{L^2}\leqslant (1+\alpha)\|df\|^2_{L^2} + (1+\alpha^{-1})\|f\|^2_{L^\infty}\|d\psi^r_i\|^2_{L^2}\leqslant\\
&\left[(1+\alpha)\lambda_k(M,c,\mu_r) + (1+\alpha^{-1})C^2_m\|d\psi^r_i\|^2_{L^2}\right]\|f\|^2_{L^2(M,\mu_r)}
\end{split}
\end{equation*}
Setting $\alpha=\|d\psi^r_i\|_{L^2} = o(1)$ one obtains
\begin{equation}
\label{kineq0}
\|d(\psi^r_if)\|^2_{L^2}\leqslant (\lambda_k(M,c,\mu_r)+o(1))\|f\|^2_{L^2(M,\mu_r)}.
\end{equation}
In order for $\psi^r_if$ to be a good test-function one needs to obtain the same estimate with $\|f\|^2_{L^2(M,\mu_r)}$ replaced by $\|\psi_i^rf\|^2_{L^2(M,\mu_i)}$. Fix $\eta>0$ and $i_0$ such that $\mu_r\left(B_{\varepsilon^j_{i_0}}(p_j)\right)<\eta$ for $j=1,2$. Then for $f_1,f_2\in V_k^r$ one has
\begin{equation*}
\begin{split}
&\left |\int f_1f_2(\psi^r_i)^2\,d\mu_i - \int f_1f_2\,d\mu_r\right|\leqslant 
\left |\int f_1f_2(\psi^r_i)^2\,d\mu_i - \int f_1f_2(\psi^r_{i_0})^2\,d\mu_i\right| + \\
&+ \left | \int f_1f_2(\psi^r_{i_0})^2\,d\mu_i-\int 
f_1f_2(\psi^r_{i_0})^2\,d\mu_r\right| + \left |\int f_1f_2(\psi^r_{i_0})^2\,d\mu_r - \int f_1f_2\,d\mu_r\right|\leqslant \\
&\leqslant \|f_1\|_{L^\infty}\|f_2\|_{L^\infty}\left(\sum_{j=1}^2\mu_i\left( B_{\varepsilon^j_{i_0}}(p_j)\setminus B_{\delta^j_i}(p_j) \right) + \mu_r\left(B_{\varepsilon^j_{i_0}}(p_j)\right)\right) +\\
&+\left | \int f_1f_2(\psi^r_{i_0})^2\,d\mu_i-\int f_1f_2(\psi^r_{i_0})^2\,d\mu_r\right| \to 2\|f_1\|_{L^\infty}\|f_2\|_{L^\infty}\mu_r\left(B_{\varepsilon^j_{i_0}}(p_j)\right)\leqslant C\eta,
\end{split}
\end{equation*}
where in the limit we used~\eqref{kineq1} and that the support of $f_1f_2(\psi^r_{i_0})^2$ is disjoint from $p_1,p_2$. Since $\eta>0$ is arbitrary we conclude that 
\begin{equation}
\label{kineq2}
\int f_1f_2\,d\mu_r= (1+o_{f_1,f_2}(1))\int f_1f_2(\psi^r_i)^2\,d\mu_i,
\end{equation}
where the index in the $o(1)$ indicates that the speed of convergence depends on $f_1,f_2$. Since $f_1, f_2$ lie in a finite-dimensional space, this estimate can be interpreted as the convergence of a sequence of quadratic forms on $V_k^r$. 
To be more precise,
let $f_i$, $i=1,\ldots, k+1$ be an $L^2(\mu_r)$-orthonormal basis of $V^r_k$ and $f=\sum a_if_i$ be a function in $V^r_k$. Then one has
\begin{equation*}
\begin{split}
&\int (\psi_i^rf)^2\,d\mu_i = \sum_{p,q=1}^{k+1} a_pa_q\int f_pf_q(\psi_i^r)^2\,d\mu_i =\\ 
& (1+o(1))\sum_{p,q=1}^{k+1}a_pa_q\int f_pf_q\,d\mu_r = (1+o(1))\|f\|^2_{L^2(\mu_r)},
\end{split}
\end{equation*}
where the advantage of taking the basis is that now we apply~\eqref{kineq2} to finitely many pairs and, thus, $o(1)$ does no longer depend on the function $f$. Combining with~\eqref{kineq0} one obtains 
\begin{equation}
\label{tfr}
\|d(\psi^r_if)\|^2_{L^2}\leqslant (\lambda_k(M,\mC,\mu_r)+o(1))\|\psi^r_if\|^2_{L^2(\mu_i)}.
\end{equation}
Note that this construction works for any number $k$ and the min-max characterisation is not required.

{\bf Test-functions associated to bubble points.} Fix $j=1,2$ and assume that $w_j\ne 0$. We show how to use the min-max characterization to construct good test-functions concentrated near the bubble point $p_j$. They are constructed inductively, such that each function is $L^2(\mu_i)$-orthogonal to the previous ones. The first function is $f^j_{i,0}:=\psi^j_i$, so that 
$$
\|df^j_{i,0}\|^2_{L^2} = \int |d\psi^j_i|^2_g\,dv_g = o(1).
$$
At the same time, 
$$
\|f^j_{i,0}\|^2_{L^2(\mu_i)}\geqslant\mu_i(B_{\gamma_i^j}(p_j))=w_j+o(1)>0. 
$$
Combining these two inequalities yields
\begin{equation}
\label{tfj0}
\|df^j_{i,0}\|^2_{L^2}\leqslant o(1)\|f^j_{i,0}\|^2_{L^2(\mu_i)}.
\end{equation}

To construct the second function $f^j_{i,1}$ one could use the Hersch trick. For example, this is the approach in 
~\cite[Lemma 3.1]{Kokarev},~\cite[Claim 12]{Petrides},~\cite[Theorem 1.1.3]{Girouard}. In order to emphasize the role of the min-max characterization, we phrase our construction in a different way. This makes the construction of $f^j_{i,2}$ an easy modification of the argument below. 

Let $\pi\colon B_{\delta^j_i}(p_j)\to\Omega^j_i\subset\mathbb{S}^2$ be the inverse stereographic projection and $g_0$ be the standard metric on $\mathbb{S}^2$. Define the sequence of measures $\mu^j_i$ on $\mathbb{S}^2$ by the formula $\mu^j_i = \pi_*((\psi^j_i)^2\mu_i)$. Since $\pi$ is conformal, it preserves the Dirichlet integrals. In particular, it is easy to see that admissibility of $\mu_i$ implies the admissibility $\mu^j_i$. 
By Proposition~\ref{existence:prop} there exists a map $u_{i,1}^j\colon \mathbb{S}^2\to \mathbb{S}^n$ with vanishing $\mu^j_i$-average and $2E(u_{i,1}^j)\leqslant \Lambda_1(\mathbb{S}^2)=8\pi$. In fact, Hersch trick implies that $n=2$ and $u_{i,1}^j$ is a conformal automorphism of $\mathbb{S}^2$, but it is not important for the argument. Since
$$
\int_{\mathbb{S}^2}|u_{i,1}^j|^2\,d\mu^j_i = \int_M(\psi^j_i)^2\,d\mu_i = w_j+o(1)
$$ 
there exists a component $v^j_{i,1}:=(u_{i,1}^j)^{m^j_{i,1}}$ of $u_{i,1}^j$ such that
\begin{equation}
\label{kineq3}
\int_{\mathbb{S}^2}|dv^j_{i,1}|_{g_0}^2\,dv_{g_0}\leqslant\frac{8\pi}{w_j+o(1)}\int_{\mathbb{S}^2}(v^j_{i,1})^2\,d\mu^j_i,\quad \int_{\mathbb{S}^2}(v^j_{i,1})^2\,d\mu^j_i\geqslant \|d\psi^j_i\|^{\frac{1}{2}}_{L^2}.
\end{equation}
Indeed, otherwise there exist $\eta>0$ such that for all $k=1,\ldots, n+1$ one has
$$
\int_{\mathbb{S}^2} \left[(u_{i,1}^j)^k\right]^2\,d\mu_i^j\leqslant \frac{w_j-\eta}{8\pi}\|d(u_{i,1}^j)^k\|^2_{L^2} + \|d\psi^j_i\|^{\frac{1}{2}}_{L^2}.
$$
Summing over all $k$ yields
$$
w_j +o(1) = \int_{\mathbb{S}^2} |u_{i,1}^j|^2\,d\mu_i^j\leqslant 2E(u^j_{i,1})\frac{w_j-\eta}{8\pi} + (n+1)\|d\psi^j_i\|^{\frac{1}{2}}_{L^2}\to w_j-\eta,
$$
which is a contradiction.

We define $f^j_{i,1}:= \psi^j_i\pi^*v^j_{i,1}$ on $B_{\delta_j^i}(p_j)$ and $0$ outside. Then one has
$$
\int_M f^j_{i,1}f^j_{i,0}\,d\mu_i = \int_M (\pi^*v^j_{i,1})(\psi^j_i)^2\,d\mu_i = \int_{\mathbb{S}^2}v^j_{i,1}\,d\mu^j_i = 0,
$$
where the last equality holds because $\mu^j_i$-average of $u_{i,1}^j$ vanishes.
Furthermore, since $\psi_i^j, |v^j_{i,1}|\leqslant 1$ one has
\begin{equation*}
\begin{split}
&\int_M|df^j_{i,1}|^2_g\,dv_g = \int_{\mathbb{S}^2}|d((\pi^{-1})^*\psi^j_i)v^j_{i,1}|^2_{g_0}\,dv_{g_0}\leqslant \\
&\leqslant (1+\alpha)\int_{\mathbb{S}^2} |dv^j_{i,1}|^2_{g_0}\,dv_{g_0} + \left(1+\frac{1}{\alpha}\right)\int_{M} |d\psi^j_i|_{g}^2\,dv_{g}.
\end{split}
\end{equation*}
Setting $\alpha = \|d\psi_i^j\|_{L^2}$ and using~\eqref{kineq3} one obtains
\begin{equation*}
\begin{split}
&\int_M|df^j_{i,1}|^2_g\,dv_g \leqslant \frac{8\pi}{w_j+o(1)}\int_{\mathbb{S}^2}(v^j_{i,1})^2\,d\mu^j_i + (1+o(1))\|d\psi_i^j\|_{L^2} = \\
& =(1+o(1))\frac{8\pi}{w_j}\int_{\mathbb{S}^2}(v^j_{i,1})^2\,d\mu^j_i = (1+o(1))\frac{8\pi}{w_j}\int_M(\pi^*v^j_{i,1})^2(\psi_i^j)^2\,d\mu_i = \\
& = (1+o(1))\frac{8\pi}{w_j}\int_M (f^j_{i,1})^2\,d\mu_i.
\end{split}
\end{equation*}
As a result, we arrive at a $2$-dimensional space $V^j_1:=\mathrm{Span}\{f^j_{i,0},f^j_{i,1}\}$ of functions supported in $B_{\delta^j_i}(p_j)$ such that for any $f\in V^j_1$ one has 
\begin{equation}
\label{tfj1}
\|df\|^2_{L^2}\leqslant (1+o(1))\frac{8\pi}{w_j}\|f\|^2_{L^2(\mu_i)}.
\end{equation}

The third function $f^j_{i,2}$ is constructed in the same way as $f^j_{i,1}$. We use Proposition~\ref{existence2:prop} with $\psi = v^j_{i,1}$, $\mu=\mu^j_i$ and find a map $u^j_{i,2}\colon \mathbb{S}^2\to\mathbb{S}^n$ such that 
\begin{equation}
\label{kineq5}
\int_{\mathbb{S}^2}u^j_{i,2}\,d\mu_i = \int_{\mathbb{S}^2}u^j_{i,2}v^j_{i,1}\,d\mu_i = 0
\end{equation}
and $2E(u)\leqslant \Lambda_2(\mathbb{S}^2)=16\pi$. Similarly, there exists a component $v^j_{i,2}:=(u^j_{i,2})^{m^j_{i,2}}$ such that
\begin{equation}
\label{kineq4}
\int_{\mathbb{S}^2}|dv^j_{i,2}|_{g_0}^2\,dv_{g_0}\leqslant\frac{16\pi}{w_j+o(1)}\int_{\mathbb{S}^2}(v^j_{i,2})^2\,d\mu^j_i,\quad \int_{\mathbb{S}^2}(v^j_{i,2})^2\,d\mu_{i}^j\geqslant \|d\psi^j_i\|^{\frac{1}{2}}_{L^2}.
\end{equation}
We define $f^j_{i,2}:= \psi^j_i\pi^*v^j_{i,2}$ on $B_{\delta_j^i}(p_j)$ and $0$ outside. Then all the arguments for $f^j_{i,1}$ carry over to $f^j_{i,2}$. In particular, $f^j_{i,2}\perp f^j_{i,0}$ in $L^2(\mu_i)$ and
$$
\int_M|df^j_{i,2}|^2_g\,dv_g \leqslant (1+o(1))\frac{16\pi}{w_j}\int_M (f^j_{i,2})^2\,d\mu_i.
$$
The additional property is that $f^j_{i,2}\perp f^j_{i,1}$ in $L^2(\mu_i)$. Indeed, by~\eqref{kineq5}
$$
\int_M f^j_{i,1}f^j_{i,2}\,d\mu_i = \int_M \pi^*(v^j_{i,1}v^j_{i,2})(\psi^j_i)^2\,d\mu_i = \int_{\mathbb{S}^2}v^j_{i,1}v^j_{i,2}\,d\mu^j_i = 0.
$$
As a result, we arrive at a $3$-dimensional space $V^j_2:=\mathrm{Span}\{f^j_{i,0},f^j_{i,1}, f^j_{i,2}\}$ of functions supported in $B_{\delta^j_i}(p_j)$ such that for any $f\in V^j_2$ one has 
\begin{equation}
\label{tfj2}
\|df\|^2_{L^2}\leqslant (1+o(1))\frac{16\pi}{w_j}\|f\|^2_{L^2(\mu_i)}.
\end{equation}

{\bf Estimates on the eigenvalues $\lambda_2(M,\mC,\mu_i)$.} Suppose that $\mu_r\ne 0$, $w_1,w_2\ne 0$. Consider the $3$-dimensional space $\psi_i^rV^r_0\oplus V_0^1\oplus V_0^2$. Since the supports of distinct summands are disjoint, the inequalities~\eqref{tfr},\eqref{tfj0} imply that 
$$
\limsup\lambda_2(M,\mC,\mu_i)=0.
$$
At the same time, $\mu_i(M)\to \mu_r(M) + w_1+w_2<\infty$, therefore, the same is true for the normalized eigenvalues $\bar\lambda_2(M,c,\mu_i)$.

Suppose $\mu_r=0$, but $w_1, w_2\ne 0$. Consider the $3$-dimensional spaces $V_0^1\oplus V_1^2$ and $V_2=V_1^1\oplus V^2_0$. Similarly to the previous case, the inequalities~\eqref{tfj0},\eqref{tfj1} imply that
$$
\limsup\lambda_2(M,\mC,\mu_i)\leqslant\min\left\{\frac{8\pi}{w_1},\frac{8\pi}{w_2}\right\}.
$$
In particular,
$$
\limsup\bar\lambda_2(M,\mC,\mu_i)\leqslant\min\left\{\frac{8\pi(w_1+w_2)}{w_1},\frac{8\pi(w_1+w_2)}{w_2}\right\}\leqslant 16\pi.
$$

Suppose $\mu_r=0$, $w_2=0$, but $w_1\ne 0$. Consider the $3$-dimensional space $V^1_2$. Then the inequality~\eqref{tfj2} implies that
$$
\limsup\lambda_2(M,\mC,\mu_i)\leqslant \frac{16\pi}{w_1}.
$$
Since $\mu_i(M)\to w_1$ one has
$$
\limsup\bar\lambda_2(M,\mC,\mu_i)\leqslant16\pi.
$$

Finally, suppose $\mu_r\ne 0$, $w_1\ne 0$, $w_2=0$. Consider the $3$-dimensional spaces $\psi^r_iV^r_1\oplus V^1_0$ and $\psi^r_iV^r_0\oplus V^1_1$. Then the inequalities~\eqref{tfr},\eqref{tfj0},\eqref{tfj1} imply
$$
\limsup\lambda_2(M,\mC,\mu_i)\leqslant\min\left\{\lambda_1(M,\mC,\mu_r),\frac{8\pi}{w_1}\right\}.
$$
Since $\mu_i(M)\to \mu_r(M) + w_1$ one has
\begin{equation*}
\begin{split}
&\limsup\bar\lambda_2(M,\mC,\mu_i)\leqslant\min\left\{\lambda_1(M,\mC,\mu_r),\frac{8\pi}{w_1}\right\}(\mu_r(M) + w_1)\leqslant\\
&\leqslant \lambda_1(M,\mC,\mu_r)\mu_r(M) + 8\pi.
\end{split}
\end{equation*}
\end{proof}
\begin{remark}
The only non-trivial part of the proof is the min-max characterization of $\Lambda_2(M,\mC)$. In particular, the same proof could be generalized to $k>2$ once the min-max characterization of $\Lambda_k(M,\mC)$ is proved. 
\end{remark}

\subsection{Convergence of $\mu_\delta$}

The contents of Section~\ref{mudeltaconvergence:sec} carry over with only minor modifications to the case of a regular $\bar\lambda_2$-conformally maximizing sequence. Below we state the analog of Theorem~\ref{thm:qualconf}. 

\begin{theorem} [Qualititative stability for $\Lambda_2(M,\mC)$]
\label{thm:higher_conf_stab}
Let $M\ne \mathbb{S}^2$ be a closed surface with a fixed conformal class $\mC$ and $g\in \mC$ be a background metric. Let $\mu_j$ be a sequence of admissible measures, such that $\bar\lambda_2(M,\mC,\mu_j)\to\Lambda_2(M,\mC)$. If $\Lambda_2(M,\mC)>\Lambda_1(M,\mC)+8\pi$ or if the sequence $\lambda_2(M,\mC,\mu_j)\mu_j$ converges to an absolutely continuous measure, then there exists  $\bar\lambda_2$-conformally maximal measure $\mu$ such that $\lambda_2(M,\mC,\mu_j)\mu_j$ converges to $\lambda_2(M,\mC,\mu)\mu$ strongly in $W^{-1,2}(M,g)$.
\end{theorem}

For bubbling sequence one has the *-weak convergence by definition. Since for any $1\geqslant\alpha>0$ the embedding $C^\alpha(M)\to C^0(M)$ is compact, *-weak convergence implies strong $(C^\alpha(M))^*$-convergence. The case of $\alpha=1$ is of particular importance, see Section~\ref{Bubble_quant:sec}.

\subsection{Quantitative stability for regular maximizing sequences via Jacobi fields} The content of Section~\ref{quantII:sec} easily carries over to the case of regular $\bar\lambda_2$-conformally maximizing sequences. We state the corresponding result.

\begin{corollary}
\label{Jacobi2:cor}
Let $(M,\mC)$ be such that $\Lambda_2(M,\mC)>\Lambda_1(M,\mC)+8\pi $ and $g\in \mC$ be a background metric. Assume that for any $n$ and any $\bar\lambda_2$-conformally maximal harmonic map $u\colon (M,\mC)\to\mathbb{S}^n$ there are no non-trivial Jacobi fields along $u$. Then there exist constants $C,\delta_0>0$ depending only on $M,\mC,g$ such that for any admissible measure $\mu$ satisfying $\Lambda_2(M,\mC)-\bar\lambda_2(M,\mC,\mu)\leqslant \delta_0$ there exists a $\bar\lambda_2$-conformally maximal measure $\mu_0$ such that 
$$
\|\lambda_2(M,\mC,\mu)\mu-\lambda_2(M,\mC,\mu_0)\mu_0\|_{W^{-1,2}(M,g)}\leqslant C\sqrt{\Lambda_2(M,\mC)-\bar\lambda_2(M,\mC,\mu)}.
$$ 
\end{corollary} 

\subsection{Quantitative stability for bubbling sequences}
\label{Bubble_quant:sec}
In this section we briefly discuss the quantitative stability estimates for bubbling $\bar\lambda_2$-conformally maximizing sequences. General estimates of this type do not immediately follow from our methods and require additional considerations. As a result, we do not present any concrete results, but, instead, outline the difficulties and formulate an open problem.

First, the proof of Proposition~\ref{S2stability2:prop} suggests that in order to obtain the stability estimate in terms of the natural quantity $\sqrt{\Lambda_2(M,\mC)-\bar\lambda_2(M,\mC,\mu)}$ one needs to consider the norm in $(C^1(M))^*$. Second, the same proof highlights one of the challenges for the general conformal class. Namely, that one needs to obtain a bound on the concentration scale of the maximizing sequence of the form~\eqref{concentration_rate:eq}. This observation raises another interesting question: does the location of the concentration point have any effect on the stability estimates, in particular, would attaching a bubble at a branch point affect the concentration scale? Finally, it is clear that any obstruction to stability of $\bar\lambda_1$-conformally maximal metrics should still be an obstruction to the stability of bubbling sequences. Putting all these observations together, one arrives at the following problem. 

\begin{open}
Let $(M,\mC)\ne\mathbb{S}^2$ be such that there are no $\bar\lambda_2$-conformally maximal metrics. Assume that for any $\bar\lambda_1$-conformally maximal harmonic map $u$ there are non non-trivial Jacobi fields along $u$. Then there exist constants $C,\delta_0>0$ depending only on $(M,\mC)$ such that for any admissible measure $\mu$ satisfying $\Lambda_2(M,\mC)-\bar\lambda_2(M,c,\mu)\leqslant \delta_0$ there exists a $\bar\lambda_1$-conformally maximal measure $\mu_0$ and a point $p\in M$ such that 
$$
\|\lambda_2(M,\mC,\mu)\mu-\lambda_1(M,\mC,\mu_0)\mu_0-8\pi\delta_p\|_{(C^1(M))^*}\leqslant C\sqrt{\Lambda_2(M,\mC)-\bar\lambda_2(M,\mC,\mu)}.
$$ 
\end{open}

\section{Stability of global maximizers for the first eigenvalue}

\label{sec:glob_stab}

In the preceding sections 
(see, in particular, Theorem \ref{thm:qualconf})  we have observed  that any $\bar\lambda_1$-conformally maximizing sequence of unit-area metrics $g_j\in \mC$ in a fixed conformal class $\mC$ on $M\neq S^2$ converges subsequentially in $W^{-1,2}$ to a conformally maximizing metric $g_{\max}\in \mC$. In this section, we establish an analogous stability result for globally $\bar\lambda_1$-maximizing sequences $g_j$ of varying conformal type, and describe some circumstances under which the rate of convergence can be quantified.

\subsection{Qualitative stability}
 Given a closed surface $M$, denote by $\Met_{\can}(M)$ the space 
$$
\Met_{\can}(M):=\{g\in \Met(M)\mid \area(M,g)=1,\text{ }K_g\equiv 2\pi \chi(M)\}
$$
of unit-area metrics on $M$ of constant curvature $K\equiv 2\pi \chi(M)$, where $\chi(M)$ denotes the Euler characteristic.
 By the uniformization theorem, $\Met_{\can}(M)$ is in one-to-one correspondence with the space of conformal classes of metrics on $M$. Note that the diffeomorphism group $\Diff(M)$ acts naturally on pairs $(g,\mu)\in \Met_{can}(M)\times \left(C^0(M)\right)^*$, by
$$
\Phi\cdot (g,\mu)=(\Phi^*g, (\Phi^{-1})_*\mu),
$$
such that
$$
\bar{\lambda}_k(\Phi\cdot (g,\mu))=\bar{\lambda}_k(g,\mu).
$$

As we'll see below, the global quantitative stability result in Theorem~\ref{glob:qual:stab} is a relatively straightforward consequence of Theorem \ref{thm:qualconf} and the results of Petrides \cite{Petrides} and Matthiesen--Siffert \cite{MS} establishing compactness of (the moduli space of) conformal classes with $\Lambda_1(M,\mC)$ sufficiently close to $\Lambda_1(M)$. 

%

\begin{proof}[Proof of Theorem~\ref{glob:qual:stab}] Since the moduli space of conformal classes on $\mathbb{RP}^2$ is trivial, the theorem reduces trivially to Theorem \ref{thm:qualconf} in this case, so in what follows we consider the case $\chi(M)\leqslant 0$. 

By the combined work of \cite[Theorem 2]{Petrides}, \cite{MS17}, and \cite{MS}, we know that for any sequence $g_j\in \Met_{\can}(M)$ satisfying $\Lambda_1(M,[g_j])\to \Lambda_1(M)$, there exists a subsequence (unrelabelled) and diffeomorphisms $\Phi_j\in \Diff(M)$ such that $\tilde{g}_j:=\Phi_j^*g_j$ converges smoothly to a metric $g_0\in \Met_{\can}(M)$ such that 
$$\Lambda_1(M,[g_0])=\Lambda_1(M).$$

Let $\tilde{\mu}_j:=(\Phi_j^{-1})_*\mu_j$, so that
$$\bar{\lambda_1}([\tilde{g}_j],\tilde{\mu}_j)=\bar{\lambda}_1([g_j],\mu_j)\to \Lambda_1(M).$$
By the smooth convergence $\tilde{g}_j\to g_0$, we see that there is a sequence $\delta_j\to 0$ for which
$$(1+\delta_j)^{-1}\tilde{g}_j\leq g_0\leq (1+\delta_j)\tilde{g}_j,$$
and in particular, it follows that
$$
\lambda_1(M,[g_0],\tilde{\mu}_j)\geq (1+\delta_j)^{-2}\lambda_1(M,[\tilde{g}_j],\tilde{\mu}_j).
$$
Thus, the admissible probability measures $\tilde{\mu}_j$ satisfy
$$\lim_{j\to\infty}\lambda_1(M,[g_0],\tilde{\mu}_j)\geq \lim_{j\to\infty}(1+\delta_j)^{-2}\lambda_1(M,c_j,\tilde{\mu}_j)=\Lambda_1(M);$$
i.e., $\tilde{\mu}_j$ is a $\bar\lambda_1$-conformally maximizing sequence within the maximizing conformal class $[g_0]$. It then follows from Theorem \ref{thm:qualconf} that $\tilde{\mu}_j\to dv_{g_{\max}}$ in $W^{-1,2}(M,g_0)$ for a 
$\bar\lambda_1$-maximizing metric $g_{\max}$ in $[g_0]$, completing the proof of the theorem.
\end{proof} 

\subsection{Quantitative stability} In some cases, we can upgrade the qualitative convergence statement to a quantitative one, bounding the distance from a given metric $g$ to a $\lambda_1$-maximizing metric $g_{\max}$ in terms of the spectral gap $\Lambda_1(M)-\overline{\lambda}_1(M,g)$. When the minimal immersion $M\to \mathbb{S}^n$ inducing the $\bar\lambda_1$-maximizing metric $g_{\max}$ is not contained in an equatorial $2$-sphere $\mathbb{S}^2\subset \mathbb{S}^n$, we have seen already in Theorem \ref{stable:thm} and Proposition \ref{stable:prop} that such a bound holds for metrics $g$ \emph{within the maximizing conformal class} $g\in [g_{max}]$. The question then becomes whether a metric $g$ \emph{not conformal} to any $\lambda_1$-maximizing metric lies a distance at most $C\sqrt{\Lambda_1(M)-\overline{\lambda}_1(M,g)}$ from some maximizing metric $g_{\max}$, in a $W^{-1,2}$ sense.

Our main result in this direction gives a positive answer provided the branched minimal immersion $u\colon M\to \mathbb{S}^n$ by first eigenfunctions realizing $\Lambda_1(M)$ has \emph{maximal Morse index}
\begin{equation}\label{max.ind}
\ind_A(u)=n+1+\dim(\mathcal{M}(M))
\end{equation}
as a critical point of the area functional, where 
$$
\mathcal{M}(M):=\Met_{\can}(M)/\Diff(M)
$$
denotes the moduli space of conformal structures on $M$. We postpone the proof to the end of the section.

\begin{theorem}\label{max.ind.thm} Suppose that the (possibly branched) minimal immersion $u\colon M\to \mathbb{S}^n$ by first eigenfunctions inducing each globally $\bar{\lambda}_1$-maximizing metric $g_{max}$ on $M$ has Morse index
$$
\ind_A(u)=n+1+\dim(\mathcal{M}(M)).
$$
Then the global quantitative stability estimate holds for $\Lambda_1(M)$ in the sense of Definition~\ref{def:quantstabglob}.
\end{theorem}

In general, if $u\colon M\to \mathbb{S}^n$ is a (branched) minimal immersion by first eigenfunctions, then by \cite[Proposition 1.6]{KRP2}, we know that $u$ has Morse index $\ind_E(u)$ at most $n+1$ as a critical point of the \emph{energy} functional in the conformal class $[u^*(g_{\mathbb{S}^n})]$, and we know from the proof of Proposition \ref{stable:prop} above that equality $\ind_E(u)=n+1$ holds provided $u(M)$ is not contained in a totally geodesic $\mathbb{S}^2\subset \mathbb{S}^n$, with the $(n+1)$-parameter family of conformal dilations of $\mathbb{S}^n$ giving an energy-decreasing family. By the results of Ejiri and Micallef in \cite{EjiriMicallef}, it follows that the area-index $\ind_A(u)$ of such a map is bounded above by
$$\ind_A(u)\leqslant n+1+\dim(\mathcal{M}(M));$$
moreover, equality implies that $u$ is a true immersion, with no branch points.

On the other hand, in view of the min-max characterization of conformally $\bar{\lambda}_1$-maximizing metrics in \cite{KS}, it is natural to expect that the minimal immersions arising from maximization of $\bar{\lambda}_1$ over all metrics may be constructed via a $(n+1)+\dim(\mathcal{M}(M))$-parameter min-max construction for the area functional, roughly corresponding to adding a Teichm\"{u}ller parameter in the construction of \cite{KS}. From this perspective, the maximal index assumption \eqref{max.ind} is a natural one, perhaps corresponding to a Morse-Bott nondegeneracy condition for the area functional at the critical value $\frac{1}{2}\Lambda_1(M)$. 

One can show directly that the maximal index condition \eqref{max.ind} is satisfied in all examples where the $\bar{\lambda}_1$-maximizing metric is known and the associated minimal immersion $u\colon M\to \mathbb{S}^n$ is not given by a branched cover of $\mathbb{S}^2$. In particular, we have the following corollaries. 

\begin{remark}
\label{rem:tensornorms}
In what follows, and throughout the section, we denote by $C^k(g_0)$ the norm
$$\|T\|_{C^k(g_0)}:=\max_{x\in M}\sum_{j=0}^k|D_{g_0}^jT|_{g_0}$$
on tensor fields $T$ (possibly taking values in a Euclidean space $\mathbb{R}^L$), where $D_{g_0}$ is the Levi-Civita connection for $g_0$; note that the norms $C^k(g_0)$ and $C^k(g_1)$ are equivalent for any pair of reference metrics $g_0,g_1$. We define the $L^p$ and Sobolev norms of $T$ similarly.
\end{remark}

\begin{corollary}\label{quant:nadirashvili} 
Let $g_1\in \Met_{\can}(\mathbb{T}^2)$ be a unit-area flat metric and $\mu$ an admissible measure on $\mathbb{T}^2$ with $\bar{\lambda}_1(\mathbb{T}^2,[g_1],\mu)\geqslant\frac{8\pi^2}{\sqrt{3}}-\delta_0(\mathbb{T}^2)$. Then there exists a metric $g_0$ isometric to $\mathbb{R}^2/\Gamma_{eq}$, where $\Gamma_{eq}$ is the lattice generated by $(1,0)$ and $\left(\frac{1}{2},\frac{\sqrt{3}}{2}\right)$ such that
$$
\left\|\lambda_1(\mathbb{T}^2,[g_1],\mu)\mu-\frac{16\pi^2}{3}dv_{g_0}\right\|_{W^{-1,2}(\mathbb{T}^2,{g_0})}^2\leqslant C \left(\frac{8\pi^2}{\sqrt{3}}-\bar{\lambda}_1(\mathbb{T}^2,[g_1],\mu)\right)
$$
and
$$
\left\|g_1-\frac{2}{\sqrt{3}}g_0\right\|_{C^1(g_0)}^2\leqslant C\left(\frac{8\pi^2}{\sqrt{3}}-\bar{\lambda}_1(\mathbb{T}^2,[g_1],\mu)\right).
$$
\end{corollary}

\begin{proof} By virtue of Nadirashvili's proof \cite{NadirashviliT2} that the equilateral flat torus $\mathbb{R}^2/\Gamma_{eq}$ is the unique global maximizer of $\bar{\lambda}_1(\mathbb{T}^2,g)$, the result will follow from Theorem \ref{max.ind.thm}, provided the homothetic minimal embedding
$$
u\colon \mathbb{R}^2/\Gamma_{eq}\to \mathbb{S}^5
$$
of the equilateral flat torus $\mathbb{R}^2/\Gamma_{eq}$ into $\mathbb{S}^5$ by first eigenfunctions has maximal Morse index
\begin{equation}\label{t.imax}
\ind_A(u)=(n+1)+\dim(\mathcal{M}(\mathbb{T}^2))=(5+1)+2=8.
\end{equation}
Indeed, \eqref{t.imax} is known to hold for the ``Bryant--Itoh--Montiel--Ros" torus $u\colon\mathbb{R}^2/\Gamma_{eq}\to \mathbb{S}^5$; see, for instance, \cite[Proposition 3.4, Remark 3.7]{KW}. More explicitly, viewing $\mathbb{S}^5$ as a unit sphere in $\mathbb{C}^3$, one can write
$$u(x)=\frac{1}{\sqrt{3}}(\phi_1,\phi_2,\phi_3),$$
where
$$\phi_1(x)=e^{2\pi i \left(x_1-\frac{x_2}{\sqrt{3}}\right)},\text{ }\phi_2(x)=e^{4\pi i \frac{x_2}{\sqrt{3}}},\text{ and }\phi_3(x)=e^{2\pi i\left(x_1+\frac{x_2}{\sqrt{3}}\right)}.$$
One can find an $8$-dimensional space of area-decreasing variations
$$\mathcal{V}_8:=\left\{(e-\langle u,e\rangle u)\mid e\in \mathbb{R}^6\right\}\oplus \mathrm{Span}\left\{ (-\phi_1,0,\phi_3), (-\phi_1,2\phi_2,-\phi_3)\right\}$$
corresponding to the deformations of $u$ by conformal dilations of $\mathbb{S}^5$ and the deformations along a natural two-parameter family of maps $\psi_{ab}\colon \mathbb{T}^2\to \mathbb{S}^5$ considered by Montiel and Ros in \cite{MR}.  These maps give  homothetic immersions of the flat metrics $T_{a,b}:=\mathbb{R}^2/\left(\mathbb{Z}(1,0)\oplus\mathbb{Z}(a,b)\right)$ into $\mathbb{S}^5$ by the first six nontrivial eigenfunctions of the Laplacian on $T_{a,b}$.

\end{proof}

\begin{corollary}\label{quant:klein} Let $g_1\in \Met_{\can}(\mathbb{K})$ be a unit-area flat metric and $\mu$ an admissible measure on the Klein bottle $\mathbb{K}$, with $\bar{\lambda}_1(\mathbb{K},[g_1],\mu)\geqslant \lambda_1(\mathbb{K})-\delta_0(\mathbb{K})$. Then there exists a unit-area $\bar{\lambda}_1$-maximizing metric $g_{\max}$ on $\mathbb{K}$ (as described in \cite{EGJ, JNP}) conformal to a unit-area flat metric $g_0\in \Met_{\can}(\mathbb{K})$ such that
$$
\|\lambda_1(\mathbb{K},[g_1],\mu)\mu-\Lambda_1(\mathbb{K})dv_{g_{\max}}\|_{W^{-1,2}(M,g_0)}^2\leqslant C\left(\Lambda_1(\mathbb{K})-\bar{\lambda}_1(\mathbb{K},[g_1],\mu)\right)
$$
and
$$
\|g_1-g_0\|_{C^1(g_0)}^2\leqslant C\left(\Lambda_1(\mathbb{K})-\bar{\lambda}_1(\mathbb{K},[g_1],\mu)\right).
$$
\end{corollary}
\begin{proof} Again, by Theorem \ref{max.ind.thm}, it suffices to show that the minimal immersion $u\colon \mathbb{K}\to \mathbb{S}^4$ by first eigenfunctions inducing the maximizing metric $g_{\max}$ has maximal Morse index
\begin{equation}\label{k.imax}
\ind_A(u)=(n+1)+\dim(\mathcal{M}(\mathbb{K}))=5+1=6.
\end{equation}

To prove \eqref{k.imax}, we employ the arguments of \cite{KW}, with minor modifications for the nonorientable setting. Let $u\colon\mathbb{K}\to \mathbb{S}^4$ be the minimal immersion described in \cite{EGJ, JNP} inducing the $\bar{\lambda}_1$-maximizing metric $g_{\max}$, and choose $a>0$ such that $g_{\max}$ is conformal to the flat metric
$$
(\mathbb{K},g_0)=\mathbb{R}^2/\Gamma_a,
$$
where $\Gamma_a\subset \mathrm{Isom}(\mathbb{R}^2)$ is the group generated by
$$
\gamma_1\cdot (x,y)=(x,y+a)\text{ and }\gamma_2\cdot (x,y)=(x+\pi,-y).
$$
Denote by $v\colon \mathbb{R}^2\to \mathbb{S}^4$ the induced conformal minimal immersion $v=u\circ \pi,$
where $\pi\colon \mathbb{R}^2\to \mathbb{R}^2/\Gamma_a$ is the obvious projection. 

Now, let $\mathcal{N}(v)\subset v^*(TS^4)$ denote the normal bundle associated with $v$, and let $\alpha$ be the second fundamental form. Following \cite{KW}, consider the normal sections $N_1,N_2\in\Gamma(\mathcal{N}(v))$ given by
$$N_1=\alpha(\partial_x,\partial_x)=-\alpha(\partial_y,\partial_y),\text{ and }N_2=\alpha(\partial_x,\partial_y).$$
In the notation of \cite{KW}, $N_1$ and $N_2$ coincide (up to a constant factor) with the sections $\Omega_1$ and $\Omega_2$ giving the real and imaginary parts, respectively, of the section $\Omega$ of the complexified normal bundle given by the normal projection of $\frac{\partial^2v}{\partial z^2}$. As in \cite{KW}, we note that the Codazzi and Ricci equations (together with the minimality of $v$) yield
\begin{equation}\label{codazzi}
D_{\partial_x}^{\perp}N_1+D_{\partial_y}^{\perp}N_2=D_{\partial_y}^{\perp}N_1-D_{\partial_x}^{\perp}N_2=0,
\end{equation}
and
\begin{equation}\label{jac.n}
\mathcal{L}N_1=2N_1+2\rho^{-4}\left[(|N_1|^2-|N_2|^2)N_1+2\langle N_1,N_2\rangle N_2\right],
\end{equation}
where $\rho:=|\partial v/\partial x|=|\partial v/\partial y|$ and $\mathcal{L}\colon \Gamma(\mathcal{N}(v))\to \Gamma(\mathcal{N}(v))$ is the Jacobi operator associated to the minimal immersion $v$.

Since 
$$(\gamma_1)_*(\partial_x)=(\gamma_2)_*(\partial_x)=\partial_x$$
and
$$(\gamma_1)_*(\partial_y)=-(\gamma_2)_*(\partial_y)=\partial_y,$$
we see that $N_1$ is invariant under the action of $\Gamma_a$, and therefore descends to a section $\bar{N}_1\in \Gamma(\mathcal{N}(u))$ of the normal bundle for the embedding $u\colon \mathbb{K}\to S^4$. The section $N_2$ is invariant under the translations $(x,y)\mapsto (x,y+a)$ and $(x,y)\mapsto (x+2\pi,y)$, but changes by a sign $N_2\circ \gamma_2=-N_2$ under the action of $\gamma_2$, and therefore does not descend to a section on the Klein bottle $\mathbb{K}$.

It is a standard consequence of \eqref{codazzi} that
$$(\partial_x-i\partial_y)\left[\frac{1}{2}(|N_1|^2-|N_2|^2)+i\langle N_1,N_2\rangle\right]=0,$$
and since $N_1$ and $N_2$ are both periodic with respect to the translations $(x,y)\mapsto (x,y+a)$ and $(x,y)\mapsto (x+2\pi,y)$, it follows that
$$(|N_1|^2-|N_2|^2)+2i\langle N_1,N_2\rangle\equiv A+Bi\in \mathbb{C}$$
is constant. Since $\langle N_1,N_2\rangle\circ \gamma_2=\langle N_1,-N_2\rangle=-\langle N_1,N_2\rangle$, it follows immediately that $B=\langle N_1,N_2\rangle\equiv 0$, and \eqref{jac.n} gives
\begin{equation}
\mathcal{L}N_1=2N_1+2A\rho^{-4}N_1.
\end{equation}

For the minimal Klein bottle given by maximization of $\bar{\lambda}_1$, we can compute the constant $A=|N_1|^2-|N_2|^2$ explicitly. Recalling from \cite{EGJ, JNP} that the minimal immersion $v\colon \mathbb{R}^2\to \mathbb{S}^4$ has the form
$$
v(x,y)=(\varphi_0(y),\varphi_1(y)e^{ix},\varphi_2(y)e^{2ix}),
$$
where $\varphi_0^2+\varphi_1^2+\varphi_2^2=1$ and $\varphi_1^2+4\varphi_2^2=(\varphi_0')^2+(\varphi_1')^2+(\varphi_2')^2$, one may check by direct computation that
$$A=|N_1|^2-|N_2|^2=\varphi_1^2+16\varphi_2^2-(\varphi_1^2+4\varphi_2^2)^2-(\varphi_1')^2-4(\varphi_2')^2.$$
That is, in the notation of \cite[p. 7]{JNP}, we have
$$A=-\kappa_0=-3E_1=4\varphi_2(0)^2(3-4\varphi_2(0)^2)=9/4,$$
since the solutions satisfy $\varphi_2(0)=\sqrt{3/8}$.

Hence, for the minimal Klein bottle in question, the section $N_1\in \Gamma(\mathcal{N}(v))$ satisfies
$$
\mathcal{L} N_1=2N_1+\frac{9}{2}\rho^{-4}N_1,
$$
and, in particular, descends to a normal section $\bar{N}_1\in \Gamma(\mathcal{N}(u))\subset \Gamma(u^*(T\mathbb{S}^4))$ over the embedding $u\colon\mathbb{K}\to \mathbb{S}^4$ satisfying
\begin{equation}
\mathcal{L}\bar{N}_1=2\bar{N}_1+\frac{9}{2}\rho^{-4}\bar{N}_1.
\end{equation}
Thus, the second variation of area $\delta^2A(u)(\bar{N}_1,\bar{N}_1)$ at $u$ along $\bar{N}_1$ satisfies 
\begin{eqnarray*}
\delta^2A(u)(\bar{N}_1,\bar{N}_1)&=&-\int_{\mathbb{K}} \langle \bar{N}_1,\mathcal{L}\bar{N}_1\rangle\,dv_{g_{\max}}\\
&=&-\int_{\mathbb{K}}\left[(2+\frac{9}{2}\rho^{-4})\right]|\bar{N}_1|^2 \,dv_{g_{\max}}\\
&<&-2\|\bar{N}_1\|_{L^2(g_{\max})}^2.
\end{eqnarray*}
Arguing as in \cite[Theorem 3.1]{KW}, we therefore deduce that the lowest eigenvalue $\lambda_0$ of the second variation $\delta^2A(u)$ must satisfy $\lambda_0<-2$. On the other hand, (cf. \cite[Lemma 2.2]{KW}), the conformal deformations give a $5$-dimensional subspace of $\Gamma(\mathcal{N}(u))$ within the $(-2)$-eigenspace of $\delta^2A(u)$, which together with the $\lambda_0$-eigenspace gives a $6$-dimensional subspace of $\Gamma(\mathcal{N}(u))$ on which $\delta^2A(u)$ is negative definite. In particular, it follows that $\ind_A(u)\geq 6$, as desired.
\end{proof}

We turn to the proof of Theorem \ref{max.ind.thm}. The main ingredient is the following lemma, showing that the assumption of maximal index gives rise to a nice family of comparison maps associated to conformal classes $[g]$ sufficiently close to the $\bar\lambda_1$-maximizing conformal class. It is convenient to formulate the result of this lemma in terms of the ``Teichm\"uller" space of conformal classes $\mM_0(M):=\Met_{\can}(M)/\Diff_0(M)$, where $\Diff_0(M)$ is the group of diffeomorphisms isotopic to the identity. If $M$ is oriented, then the space $\mM_0(M)$ agrees with the classical Teichm\"uller space of complex structures on $M$. The advantage of $\mM_0(M)$ is that it is a manifold whereas $\mM(M)$ is only an orbifold.

\begin{lemma}\label{max.ind.lem} Let $u\colon M\to \mathbb{S}^n$ be a minimal immersion by the first eigenfunctions whose induced metric $u^*(g_{S^n})$ is conformal to the constant curvature metric $g_0\in \Met_{\can}(M)$. Assume that $u$ has the maximal Morse index, i.e. that
$$
\ind_A(u)=n+1+\dim(\mathcal{M}(M)).
$$
Then there is a neighborhood $\mathcal{U}$ of the class $\langle g_0\rangle$ in the 
Teichm\"uller space $\mathcal{M}_0(M)$ and a family of immersions
$$
\mathcal{U}\ni \tau \mapsto F_{\tau}\in C^{\infty}(M,\mathbb{S}^n)
$$
such that the constant curvature metrics $g_{\tau}\in \Met_{\can}(M)$ conformal to $F_{\tau}^*(g_{\mathbb{S}^n})$ satisfy $\langle g_{\tau}\rangle=\tau$, for which the following holds: denoting by $F_{\tau,a}:=G_a\circ F_{\tau}$ the composition with the conformal automorphism $G_a\in \mathrm{Conf}(\mathbb{S}^n)$ as in Remark~\ref{canon:ex}, there exist $C<\infty$ and $\delta_1>0$ such that for every $(a,\tau)\in \mathbb{B}^{n+1}\times \mathcal{U}$, either
$$
\area(F_{\tau,a})\leqslant \area(u)-\delta_1,
$$
or
\begin{equation}\label{main.est.1}
\|g_{\tau}-g_0\|_{C^1(g_0)}^2+\|F_{\tau,a}-u\|_{C^2(g_0)}^2\leqslant C [\area(u)-\area(F_{\tau,a})].
\end{equation}
\end{lemma}

\begin{remark} The constants $C<\infty$ and $\delta_1>0$ in the lemma depend only on the conformal structure $\langle g_0\rangle \in \mathcal{M}(M)$, since each conformal structure carries at most one minimal immersion by first eigenfunctions, up to isometries (by \cite[Theorem 1]{MR}). In particular, for those minimal immersions $u$ arising from maximization of $\bar{\lambda}_1$, the constants $C=C(M)$ and $\delta_1(M)$ may be taken independent of the conformal structure $\langle g_0\rangle \in \mathcal{M}(M)$ as well, since the collection of conformal structures achieving $\Lambda_1(M)$ is compact (recall that we are assuming $M\ne\mathbb{S}^2$ throughout this section).
\end{remark}

\begin{proof} 
Recall that the maximal index condition implies by~\cite{EjiriMicallef} that $du$ never vanishes, so the assumption that $u$ is unbranched follows from the index condition.

Denote by $g_0\in \Met_{\can}(M)$ the unit-area constant curvature metric conformal to $u^*(g_{\mathbb{S}^n})$. As discussed above, since $u$ is a map by first eigenfunctions for $\Delta_{u^*(g_{\mathbb{S}^n})}$, we know that $u$ has Morse index 
\begin{equation}\label{e-ind.char}
\ind_E(u)=n+1
\end{equation}
as a critical point of the energy functional $E_{g_0}$, with an $(n+1)$-parameter family of area-decreasing deformations given by the conformal variations
\begin{equation}\label{vconf}
\mathcal{V}_{conf}:=\left\{e-\langle e,u\rangle u \mid e\in \mathbb{R}^{n+1}\right\}.
\end{equation}

Now, for any variation vector field $v\in \Gamma(u^*T\mathbb{S}^n))$, denote by $\eta(v)$ the symmetric two-tensor
$$
\eta(v)=du^tdv+dv^tdu=\left.\frac{d}{dt}\right|_{t=0}\left[\left(\frac{u+tv}{|u+tv|}\right)^*(g_{\mathbb{S}^n})\right],
$$
and denote by $\eta(v)^T$ its $g_0$-traceless part
$$\eta(v)^T:=\eta(v)-\langle du,dv\rangle_{g_0}g_0.$$
A straightforward computation (cf., e.g., \cite{Moore}) shows that the second variation $\delta^2A(u)(v,v)$ of area along $v\in \Gamma(u^*(T\mathbb{S}^n))$ is given by
\begin{equation}\label{area.2var}
\delta^2A(u)(v,v)=\delta^2E_{g_0}(u)(v,v)-\frac{1}{2}\int_M\frac{1}{2}|du|_{g_0}^{-2}\left|\eta(v)^T\right|^2_{g_0}\,dv_{g_0}.
\end{equation}
Recall that $\delta^2A(u)(v,v)$ depends only on the component of $v$ orthogonal to the tangent space $du(TM)$; indeed, for any tangent vector field $X\in \Gamma(TM)$ generating a family of diffeomorphisms $\Phi_t\in \Diff_0(M)$, we have
\begin{equation*}
\begin{split}
&\delta^2A(u)\left(v+du(X),v+du(X)\right)=\left.\frac{d^2}{dt^2}\right|_{t=0}\area\left(\frac{u+tv}{|u+tv|}\circ\Phi_t\right)=\\
&=\left.\frac{d^2}{dt^2}\right|_{t=0}\area\left(\frac{u+tv}{|u+tv|}\right)=\delta^2A(u)(v,v),
\end{split}
\end{equation*}
by the invariance of the area functional under reparametrization.

By the assumption of maximal index
$$
\ind_A(u)=I_{\max}:=n+1+\dim(\mathcal{M}(M)),
$$
we can find a space of sections $\mathcal{V}\subset \Gamma(u^*(T\mathbb{S}^n))$ of dimension $\dim(\mathcal{V})=I_{\max}$ such that
\begin{equation}\label{2var.bds}
\max_{0\neq v\in \mathcal{V}}\frac{\delta^2A(u)(v,v)}{\|v\|_{C^2(g_0)}^2}\leqslant -c<0
\end{equation}
for some $c=c(u(M))>0$. Moreover, we take $\mathcal{V}$ to be of the form
$$
\mathcal{V}=\mathcal{V}_{\conf}\oplus \mathcal{W},
$$
where $\mathcal{V}_{\conf}$ is the $(n+1)$-dimensional space of conformal variations given by \eqref{vconf} and $\mathcal{W}\subset \Gamma(u^*(T\mathbb{S}^n))$ is a complementary $\dim(\mathcal{M}(M))$-dimensional space of variations.

With $\mathcal{V}$ as above, consider the linear map
$$
T\colon \mathcal{V}\to S:=\Gamma(\mathrm{Sym}^2(T^*M))
$$
onto the symmetric two-tensors, given by 
$$
T(v):=2|du|_{g_0}^{-2}\eta(v)=2|du|_{g_0}^{-2}(du^tdv+dv^tdu).
$$
Now, recall (e.g., from \cite{FM}) that the space $S$ of symmetric $2$-tensors admits a direct sum decomposition (determined by $g_0$)
$$
S=S_0\oplus S_1\oplus S_2
$$
into the components
$$
S_0:=\{\varphi g_0\mid \varphi \in C^{\infty}(M)\},
$$
$$
S_1:=\{\mathcal{L}_Xg_0-\mathrm{div}_{g_0}(X)g_0 \mid X\in \Gamma(TM)\},
$$
and
$$
S_2:=\{h\in S\mid \langle h,g_0\rangle \equiv 0,\text{ }\mathrm{div}_{g_0}(h)=0\}.
$$
The latter summand $S_2$ corresponds to those variations of $g_0$ within the unit-area constant curvature metrics $\Met_{\can}(M)$ which are $L^2(g_0)$-orthogonal to the orbit $\{\Phi^*g_0\mid \Phi\in \Diff_0(M)\}$ of $g_0$ under the action of $\Diff_0(M)$, and may, in particular, be identified with the tangent space $T_{\langle g_0\rangle}\mathcal{M}_0(M)$ to the Teichm\"uller space $\mathcal{M}_0(M)$ at $\langle g_0\rangle$.

Denote by $T_2\colon \mathcal{V}\to S_2$ the composition of $T\colon\mathcal{V}\to S$ with the projection $S\to S_2$. We claim now that $T_2$ is \emph{surjective}. To see this, consider the kernel $\mathcal{V}_0:=\ker(T_2)$. It follows from the definition of $T_2$ that each $v\in \mathcal{V}_0$ satisfies
$$
\eta(v)=\varphi g_0+\frac{1}{2}|du|_{g_0}^2\mathcal{L}_Xg_0
$$
for some $\varphi\in C^{\infty}(M)$ and $X\in \Gamma(TM)$ which are determined uniquely (and linearly) by $v$, provided we select $X$ in the $L^2(g_0)$-orthogonal complement to the space of conformal Killing vector fields for $g_0$. Moreover, it is straightforward to check that
$$
\mathcal{L}_X(\rho g_0)=\rho\mathcal{L}_X(g_0)+X(\rho)g_0
$$
for any conformal factor $0<\rho\in C^{\infty}(M)$, so we can write
$$
\eta(v)=\varphi_1g_0+\frac{1}{2}\mathcal{L}_X\left(|du|_{g_0}^2g_0\right)=\varphi_1g_0+\mathcal{L}_X\left(u^*(g_{\mathbb{S}^n})\right)
$$
for some $\varphi_1\in C^{\infty}(M)$ and $X\in \Gamma(TM)$. On the other hand, it is easy to check that
$$\eta(du(X))=\mathcal{L}_X(u^*(g_{S^n})),$$
so that
$$\eta(v-du(X_v))=\eta(v)-\eta(du(X_v))=\varphi_1g_0\in S_0.$$
In particular, since $du(X)$ is tangential to $du(M)$, it follows from \eqref{area.2var} that
\begin{equation*}
\begin{split}
&\delta^2A(u)(v,v)=\delta^2A(u)\left(v-du(X),v-du(X)\right)=\\
&=\delta^2E_{g_0}(u)\left(v-du(X),v-du(X)\right)-\frac{1}{2}\int_M|du|_{g_0}^{-2}\left|\eta(v-du(X))^T\right|^2_{g_0}\,dv_{g_0}\\
&=\delta^2E_{g_0}(u)\left(v-du(X),v-du(X)\right).
\end{split}
\end{equation*}
Thus, it follows that the second variation of energy $\delta^2E_{g_0}(u)$ is negative definite on the space
$$
\widetilde{\mathcal{V}}_0:=\left\{v-du(X_v)\mid v\in \mathcal{V}_0\right\}\subset \Gamma(u^*(T\mathbb{S}^n)),
$$
and since $\ind_E(u)=n+1$, it follows that $\dim\left(\widetilde{\mathcal{V}}_0\right)\leqslant n+1$. Moreover, it is easy to see that the map $v\mapsto du(X_v)$ is injective on $\mathcal{V}_0$, since $\delta^2A(u)(\cdot,\cdot)$ vanishes for the tangent vector field $du(X_v)$. Hence,
$$
\dim(\mathcal{V}_0)=\dim(\ker(T_2))\leqslant n+1.
$$
In particular, it follows that
$$
\ker(T_2)=\mathcal{V}_{\conf},
$$
and since $\dim(S_2)=\dim (\mathcal{M}(M))$, we deduce that
$$T_2: \mathcal{W}\to S_2$$
gives an isomorphism.

Now, for each $v\in \mathcal{W}$, let 
$$
F_v:=\frac{u+v}{|u+v|}\colon M\to \mathbb{S}^n,
$$
and for $v$ in a small neighborhood $0\ni W\subset \mathcal{W}$ of $0$ (so that $F_v$ remains an isomorphism), let $g_v\in \Met_{\can}(M)$ be the unique constant-curvature unit-area metric conformal to $F_v^*(g_{\mathbb{S}^n})$. A straightforward computation reveals that
\begin{equation}
\left.\frac{d}{dt}\right|_{t=0}(g_{tv_0})=2\left(|du|_{g_0}^{-2}\eta(v_0)\right)^T
\end{equation}
In particular, for $v$ in a sufficiently small neighborhood $0\ni W'\subset \mathcal{W}$, since $T_2\colon \mathcal{W}\to S_2$ is an isomorphism, the assignment
$$
U\ni v\mapsto \langle g_v\rangle \in \mathcal{M}_0(M)
$$
gives a homeomorphism onto a neighborhood of $\langle g_0\rangle$ in the Teichm\"uller space. 

Letting $\mathbb{B}^{n+1}\ni a \mapsto G_a$ denote the family of conformal dilations of $\mathbb{S}^n$, note that the subspace of variations in $\Gamma(u^*(T\mathbb{S}^n))$ given by differentiating the families
$G_a\circ F_v$ at $(a,v)=0\in \mathbb{B}^{n+1}\times \mathcal{W}$ is precisely
$$
\mathcal{V}=\mathcal{V}_{\conf}+\mathcal{W}.
$$
In particular, it follows from \eqref{2var.bds} that for $(a,v)$ in a neighborhood $0\in \mathcal{O}\subset \mathcal{V}$ of $0$ in $\mathcal{V}$,
$$
\|G_a\circ F_v -u\|_{C^2(g_0)}^2\leqslant C[\area(u)-\area(G_a\circ F_v)].
$$
Moreover, since the maps $G_a$ are conformal, for any $(a_0,v_0)\in \mathbb{B}^{n+1}\times \mathcal{W}$, we see that
\begin{eqnarray*}
\left.\frac{d}{dt}\right|_{t=0}g_{tv_0}=2\left(|du|_{g_0}^{-2}\eta(v_0)\right)^T
=2|du|_{g_0}^{-2}\eta\left(\left.\frac{d}{dt}\right|_{t=0}[G_{t a_0}\circ F_{tv_0}]\right)^T
\end{eqnarray*}
from which it follows that 
$$
\|g_v-g_0\|_{C^1(g_0)}\leqslant C \|G_a\circ F_v-u\|_{C^2(g_0)}
$$
for $(a,v)$ in a sufficiently small neighborhood $\mathcal{O}$ of $0$ in $\mathbb{B}^{n+1}\times \mathcal{W}$. In particular, for $(a,v)\in \mathcal{O}\subset \mathcal{W}$, it follows that
\begin{equation}\label{dist.ctrl}
\|g_v-g_0\|_{C^1(g_0)}^2+\|G_a\circ F_v-u\|_{C^2(g_0)}^2\leqslant C[\area(u)-\area(G_a\circ F_v)].
\end{equation}

Finally, letting $\mathcal{O}$ be such a neighborhood of $0$ in $\mathbb{B}^{n+1}\times \mathcal{W}$, we claim that there exists $\delta_1>0$ and a neighborhood $U$ of $0$ in $\mathcal{W}$ such that if $(a,v)\in \mathbb{B}^{n+1}\times U$ and $\area(G_a\circ F_v)\geqslant \area(u)-\delta_1$, then $(a,v)\in \mathcal{O}$. Indeed, if no such $\delta_1$ and $W$ existed, then we could find $v_j\to 0$ in $\mathcal{W}$ and $a_j\to a\neq 0$ in $\overline{\mathbb{B}}^{n+1}$ for which $\lim_{j\to\infty}\area(G_{a_j}\circ F_{v_j})\geqslant \area(u)$. If $a\in \mathbb{B}^{n+1}$, then $\lim_{j\to\infty}\area(G_{a_j}\circ F_{v_j}) = \area(G_a\circ u)$, which contradicts the fact that $u$ is the unique maximizer of area among the maps $G_a\circ u$. The case $a\in \partial\mathbb{B}^{n+1}$ is more subtle and we postpone the detailed argument until Section~\ref{sec:can_family}. In Corollary~\ref{cor:usc.} we show that in this case $\lim_{j\to\infty}\area(G_{a_j}\circ F_{v_j})\leqslant \lim_{t\to 1}\area(G_{ta}\circ u)$, which contradicts the fact that $\area(G_{ta}\circ u)$ is strictly decreasing in $t$, see the proof of Proposition~\ref{stable:prop}.

Combining the conclusions of the preceding three paragraphs, we see that we can find a neighborhood $U$ of $0$ in $\mathcal{W}$ such that 
$$U\ni v\mapsto \langle g_v\rangle \in \mathcal{M}_0(M)$$
gives a homeomorphism onto its image $\mathcal{U}\subset \mathcal{M}_0(M)$, and such that for every $v\in U$ and $a\in \mathbb{B}^{n+1}$, either
$$
\|g_v-g_0\|_{C^1(g_0)}^2+\|G_a\circ F_v-u\|_{C^2(g_0)}^2\leqslant C[\area(u)-\area(G_a\circ F_v)]
$$
or
$$
\area(G_a\circ F_v)\leqslant \area(u)-\delta_1.
$$
Thus, we see that the family of maps
$$
\mathcal{U}\ni \langle g_v\rangle \mapsto F_v\in C^{\infty}(M,\mathbb{S}^n)
$$
satisfies all the desired properties, completing the proof of the lemma.

\end{proof}

With the preceding lemma in place, we can now complete the proof of Theorem \ref{max.ind.thm}.

\begin{proof}[Proof of Theorem \ref{max.ind.thm}]

Let $\mathcal{C}_{\max}\subset \mathcal{M}_0(M)$ denote the set of (equivalence classes of) conformal structures $\langle g\rangle$ achieving the maximum $\Lambda_1(M,[g])=\Lambda_1(M)$. By assumption, for each $\langle g_0\rangle \in \mathcal{C}_{\max}$, the minimal immersion $u\colon M\to \mathbb{S}^n$ arising from maximization of $\bar{\lambda}_1$ satisfies the hypotheses of Lemma \ref{max.ind.lem}. Thus, for each $\langle g_0\rangle \in \mathcal{C}_{\max}$, there is a neighborhood $\mathcal{U}_{\langle g_0\rangle}$ of $\langle g_0\rangle$ in $\mathcal{M}_0(M)$ and a family of maps $\mathcal{U}\ni \tau \mapsto F_{\tau}\in C^{\infty}(M,\mathbb{S}^n)$ satisfying the conclusions of the lemma. Namely, the constant curvature metric $g_{\tau}\in \Met_{\can}(M)$ conformal to $F_{\tau}^*(g_{\mathbb{S}^n})$ lies in $\tau\in \mathcal{M}_0(M)$, $F_{\langle g_0\rangle}=u_0$ is the minimal immersion inducing the $\bar{\lambda}_1$-maximizing metric in $\langle g_0\rangle$, and for every $(a,\tau)\in \mathbb{B}^{n+1}\times \mathcal{U}$ with
$$
\area(G_a\circ F_{\tau})\geq \area(u_0)-\frac{\delta_0(M)}{2}=\frac{1}{2}[\Lambda_1(M)-\delta_0(M)],
$$
we have
$$
\|g_{\tau}-g_0\|_{C^1(g_0)}^2+\|G_a\circ F_{\tau}-u_0\|_{C^2(g_0)}^2\leqslant C\left[\frac{1}{2}\Lambda_1(M)-\area(G_a\circ F_{\tau})\right].
$$

Now, let $g_1\in \Met_{\can}(M)$ and let $\mu$ be an admissible measure such that
$$
\bar{\lambda}_1(M,[g_1],\mu)\geqslant\Lambda_1(M)-\delta_0.
$$
By the qualitative convergence result of Theorem \ref{glob:qual:stab}, provided $\delta_0=\delta_0(M)>0$ sufficiently small, it follows that $\langle g_1\rangle\in \mathcal{U}_{\langle g_0\rangle}$ for some $\langle g_0\rangle\in \mathcal{C}_{\max}$. Thus, we have a map $F=F_{\langle g_1\rangle}\colon M\to \mathbb{S}^n$ and a representative $g_0\in \Met_{\can}(M)$ of $\langle g_0\rangle \in \mathcal{C}_{\max}$ (possibly after replacing $\left(F_{\langle g_1\rangle},g_0\right)$ with $\left(F_{\langle g_1\rangle}\circ \Phi, \Phi^*g_0\right)$ for an appropriate $\Phi\in \Diff(M)$) such that for every $a\in \mathbb{B}^{n+1}$, if
$$
\area(G_a\circ F)\geqslant \area(u_0)-\frac{\delta_0(M)}{2}=\frac{1}{2}[\Lambda_1(M)-\delta_0(M)],
$$
then
\begin{equation}\label{map.comps}
\|g_1-g_0\|_{C^1(g_0)}^2+\|G_a\circ F-u_0\|_{C^2(g_0)}^2\leqslant C[\Lambda_1(M)-2\area(G_a\circ F)].
\end{equation}

Since $\mu$ is admissible, we know that there exists $a\in \mathbb{B}^{n+1}$ for which $F_a:=G_a\circ F$ satisfies
$$
\int_MF_a\,d\mu=0\in \mathbb{R}^{n+1},
$$
and since $u_a\colon (M,g_1)\to \mathbb{S}^n$ is conformal, it follows that
$$
\area(F_a)=E_{g_1}(F_a)\geqslant \frac{1}{2}\bar{\lambda}_1(M,[g_1],\mu) \geqslant \frac{1}{2}[\Lambda_1(M)-\delta_0].
$$
Hence, \eqref{map.comps} holds, and we have
\begin{equation}\label{pre.main.est}
\|g_1-g_0\|_{C^1(g_0)}^2+\|F_a-u_0\|_{C^2(g_0)}^2 \leqslant C[\Lambda_1(M)-\bar{\lambda}_1(M,[g_1],\mu)].
\end{equation}
Moreover, by Lemma \ref{main:lemma}, we see that
\begin{equation*}
\begin{split}
&\left\||dF_a|_{g_1}^2dv_{g_1}-\lambda_1([g_1],\mu)\mu\right\|_{W^{-1,2}(g_1)}^2 \leqslant \\
&\leqslant\|dF_a\|_{L^{\infty}(g_1)}^2[2\area(F_a)-\bar{\lambda}_1(M,[g_1],\mu)]\leqslant\\
&\leqslant \|dF_a\|_{L^{\infty}(g_1)}^2[\Lambda_1(M)-\bar{\lambda}_1(M,[g_1],\mu)].
\end{split}
\end{equation*}
Finally, it follows from \eqref{pre.main.est} that, provided $\Lambda_1(M)-\bar{\lambda}_1(M,[g_1],\mu)$ is sufficiently small,
$$
\|dF_a\|_{L^{\infty}(g_1)}^2\leqslant C\|dF_a\|_{L^{\infty}(g_0)}^2\leqslant C,
$$
$$
\|\cdot\|_{W^{-1,2}(g_0)}\leqslant C\|\cdot\|_{W^{-1,2}(g_1)},
$$
and
\begin{equation*}
\begin{split}
&\left\||dF_a|^2_{g_1}dv_{g_1}-|du_0|_{g_0}^2dv_{g_0}\right\|_{W^{-1,2}(g_0)}\leqslant \\
&\leqslant C(\|g_1-g_0\|_{C^0_{g_0}}+\|F_a-u_0\|_{C^1(g_0)}) \leqslant C[\Lambda_1(M)-\bar{\lambda}_1(M,[g_1],\mu)]^{\frac{1}{2}}.
\end{split}
\end{equation*}
Combining the preceding observations, we therefore see that
$$
\left\||du_0|_{g_0}^2dv_{g_0}-\lambda_1([g_1],\mu)\mu\right\|_{W^{-1,2}(g_0)}\leqslant C \sqrt{\Lambda_1(M)-\bar{\lambda}_1(M,[g_1],\mu)}.
$$
In particular, since 
$$
|du_0|_{g_0}^2dv_{g_0}=\Lambda_1(M)dv_{g_{\max}},
$$
where $g_{\max}$ is the globally $\bar{\lambda}_1$-maximizing unit area metric conformal to $g_0$, by combining the preceding estimate with \eqref{pre.main.est}, we see that
\begin{equation*}
\begin{split}
\|\lambda_1(M,[g_1],\mu)\mu-\Lambda_1(M)dv_{g_{\max}}\|_{W^{-1,2}(g_0)}^2+\|g_1-g_0\|_{C^1_{g_0}}^2\leqslant \\\leqslant C[\Lambda_1(M)-\bar{\lambda}_1(M,[g_1],\mu)],
\end{split}
\end{equation*}
as desired. 

Finally, note that in the special case in which $\mu=dv_g$ for a metric $g$ conformal to $g_1$, it follows that
$$
\|\lambda_1(M,g)g-\Lambda_1(M)g_{\max}\|^2_{W^{-1,2}(g_0)}\leqslant C [\Lambda_1(M)-\bar{\lambda}_1(M,[g_1],\mu)].
$$

\end{proof}

\subsection{Limiting behaviour of the canonical family}
\label{sec:can_family}
Let $F\colon M\to \mathbb{S}^n$ be an immersion with induced metric $g_F=F^*(g_{\mathbb{S}^n})$. Given $y\in \mathbb{S}^n$ and $r>0$, consider the area density
$$
\Theta_F(y,r):=\frac{\area_{g_F}(\{|F-y|<r\})}{r^2},
$$
and recall the following standard monotonicity result for immersions $M\to \mathbb{R}^{n+1}$ with bounded mean curvature.

\begin{lemma}[see e.g. ~\cite{Simon}, section 17]
\label{mono} 
Let $F\colon M\to \mathbb{R}^{n+1}$ be an immersion with induced metric $g_F=F^*(g_{\mathbb{S}^n})$. Then for $r>0$ one has
$$
\frac{d}{dr}\left[e^{r\|H_F\|_{\infty}}\Theta_F(y,r)\right]\geqslant 0,
$$
where $H_F=\Delta_{g_F}F$ is the mean curvature of $F$ as an immersion into $\mathbb{R}^{n+1}$ and $\|H_F\|_{\infty}$ is its $L^\infty$-norm.
\end{lemma}
An immediate corollary is the well-definedness of the density
$$
\Theta_F(y,0):=\lim_{r\to 0}\Theta_F(y,r),
$$
and for any $0<r<s$ one has
$$
\Theta_F(y,r)\leqslant e^{\|H_F\|_{\infty}(s-r)}\Theta_F(y,s).
$$

Recall that the canonical family $\mathbb{B}^{n+1}\ni a\mapsto G_a\in \Conf(\mathbb{S}^n)$ is given by
$$
G_a(x)=(1-|a|^2)\frac{x+a}{|x+a|^2}+a.
$$
We then have the following.

\begin{proposition} Let $F\colon M\to \mathbb{S}^n$ be an immersion with mean curvature $H_F=\Delta_{g_F}F$ as a submanifold of Euclidean space, and let $a\in \mathbb{B}^{n+1}$ with $1-|a|<\delta^3<1/2$. Then there exists a universal constant $C_0$ (independent of $F$, $n$) such that, letting $\alpha=-a/|a|$, we have
\begin{equation}
\label{l.bd}
\area(G_a\circ F)\geqslant (1+|a|)^2\frac{\Theta_F(\alpha,0)}{e^{\|H_F\|_{\infty}\delta}}-C_0\left(1+\delta^2e^{C_0\|H_F\|_{\infty}}\right)\frac{\delta^2}{|a|^2}\area(F)
\end{equation}
and
\begin{equation}\label{u.bd}
\area(G_a\circ F)\leqslant 4e^{\|H_F\|_{\infty}\delta}\Theta_F(\alpha,\delta)+C_0\left(1+\delta e^{C_0\|H_F\|_{\infty}}\right)\frac{\delta^2}{|a|^2}\area(F).
\end{equation}
\end{proposition}  

\begin{proof}
By a direct computation, we have
\begin{equation*}
\begin{split}
\area(G_a\circ F)&=\int_M\frac{(1-|a|^2)^2}{|F+a|^4}\,dv_{g_F}=\int_M\frac{(1-|a|^2)^2}{\left(|a||F-\alpha|^2+(1-|a|)^2\right)^{2}}dv_{g_F}=\\
&=\int_{\{|F-\alpha|<\delta\}}\frac{(1-|a|^2)^2}{\left(|a||F-\alpha|^2+(1-|a|)^2\right)^{2}}dv_{g_F}+\\
&+\int_{\{|F-\alpha|\geqslant\delta\}}\frac{(1-|a|^2)^2}{\left(|a||F-\alpha|^2+(1-|a|)^2\right)^{2}}dv_{g_F},
\end{split}
\end{equation*}
so that
\begin{equation*}
\begin{split}
&\left|\area(G_a\circ F)-\int_{\{|F-\alpha|<\delta\}}\frac{(1-|a|^2)^2}{\left(|a||F-\alpha|^2+(1-|a|)^2\right)^{2}}\right|\leqslant \\
&\leqslant \frac{(1-|a|^2)^2}{|a|^2\delta^4}\area(F)\leqslant \frac{4\delta^2}{|a|^2}\area(F),
\end{split}
\end{equation*}
where in the last line we used the bound $1-|a|<\delta^3$.
Next, applying the coarea formula for $|F-\alpha|$ and integrating by parts, we see that
\begin{equation*}
\begin{split}
&\int_{\{|F-\alpha|<\delta\}}\left(|a||F-\alpha|^2+(1-|a|)^2\right)^{-2}dv_{g_F}=\\
&=\int_0^{\delta}(|a|s^2+(1-|a|)^2)^{-2}\frac{d}{ds}\area(\{|F-\alpha|<s\}) ds=\\
&=\frac{\Theta_F(\alpha,\delta)\delta^2}{(|a|\delta^2+(1-|a|)^2)^2}
-\int_0^{\delta}\Theta_F(\alpha,s)s^2\frac{d}{ds}(|a|s^2+(1-|a|)^2)^{-2} ds.
\end{split}
\end{equation*}
Set $v_a(s):=|a|s^2+(1-|a|)^2>0$. Combining the previous computations gives 
\begin{equation}
\label{eq:idk}
\begin{split}
&\left|\area(G_a\circ F) - 2(1-|a|^2)^2\int_0^{\delta}\Theta_F(\alpha,s)s^2v_a^{-3}v_a'(s) ds\right|\leqslant 
\frac{4\delta^2}{|a|^2}\area(F)+\\
&+(1-|a|^2)^2\frac{\Theta_F(\alpha,\delta)\delta^2}{(|a|\delta^2+(1-|a|)^2)^2}\leqslant \frac{4\delta^2}{|a|^2}\area(F)+\frac{4\delta^4}{|a|^2}\Theta_F(\alpha,\delta),
\end{split}
\end{equation}
where we used the bound $(1-|a|^2)^2\leqslant 4(1-|a|)^2<4\delta^6$ in the last step. 
By Lemma~\ref{mono} one has
$$
\Theta_F(\alpha,\delta)\leqslant e^{2\|H_F\|_{\infty}}\Theta_F(\alpha,2)=e^{2\|H_F\|_{\infty}}\frac{\area(F)}{4}.
$$
Substituting this into~\eqref{eq:idk} yields
\begin{equation*}
\begin{split}
&\left|\area(G_a\circ F) - 2(1-|a|^2)^2\int_0^{\delta}\Theta_F(\alpha,s)s^2v_a^{-3}v_a'(s) ds\right|\leqslant \\
&\leqslant \left(4+\delta^2e^{2\|H_F\|_{\infty}}\right)\frac{\delta^2}{|a|^2}\area(F).
\end{split}
\end{equation*}

Finally, we estimate the integral term in the l.h.s. The monotonicity statement of Lemma~\ref{mono} gives that for all $s\in [0,\delta]$ one has
\begin{equation}
\label{eq:idk2}
e^{-\|H_F\|_{\infty}\delta}\Theta_F(\alpha,0)\leqslant\Theta_F(\alpha,s)\leq e^{\|H_F\|_{\infty}\delta}\Theta_F(\alpha,\delta).
\end{equation}
At the same time, 

\begin{eqnarray*}
\int_0^{\delta}s^2\frac{v_a'(s)}{v_a^3}ds&=&\frac{1}{|a|}\int_0^{\delta}(v_a(s)-(1-|a|)^2)\frac{v_a'(s)}{v_a^3}ds=\\
&=&\frac{1}{|a|}\int_0^{\delta}\frac{d}{ds}\left(-\frac{1}{v_a(s)}+\frac{(1-|a|)^2}{2v_a(s)^2}\right)ds=\\
&=&\frac{1}{|a|}\left(\frac{(1-|a|)^2}{2v_a(\delta)^2}-\frac{1}{v_a(\delta)}+\frac{1}{v_a(0)}-\frac{(1-|a|)^2}{2v_a(0)^2}\right)=\\
&=&\frac{1}{2|a|(1-|a|)^2}+\frac{(1-|a|)^2-2[|a|\delta^2+(1-|a|)^2]}{2|a|[|a|\delta^2+(1-|a|)]^2},
\end{eqnarray*}
so that
$$
\left|\int_0^{\delta}s^2\frac{v_a'(s)}{v_a^3}ds-\frac{1}{2|a|(1-|a|)^2}\right|\leqslant \frac{1}{|a|^2\delta^2}.
$$
Combining this with~\eqref{eq:idk2} yields (note that $v_a', v_a>0$)
\begin{equation*}
\begin{split}
&\area(G_a\circ F)\geqslant -\left(4+\delta^2e^{2\|H_F\|_{\infty}}\right)\frac{\delta^2}{|a|^2}\area(F) +\\
&+2(1-|a|^2)^2\int_0^{\delta}\Theta_F(\alpha,s)s^2\frac{v_a'(s)}{v_a^3}ds\geqslant -\left(4+\delta^2e^{2\|H_F\|_{\infty}}\right)\frac{\delta^2}{|a|^2}\area(F)+\\
& +2(1-|a|^2)^2e^{-\|H_F\|_{\infty}\delta}\Theta_F(\alpha,0)\left(\frac{1}{2|a|(1-|a|)^2}-\frac{1}{|a|^2\delta^2}\right)\\
&\geqslant (1+|a|)^2e^{-\|H_F\|_{\infty}\delta}\Theta_F(\alpha,0)-C_0\left(1+\delta^2e^{C_0\|H_F\|_{\infty}}\right)\frac{\delta^2}{|a|^2}\area(F),
\end{split}
\end{equation*}
where in the last line we used Lemma~\ref{mono} to bound the term $\dfrac{\Theta_F(\alpha,0)}{|a|^2\delta^2}$ in terms of $\area(F)$.
Similarly, we find that
$$
\area(G_a\circ F)\leqslant \frac{(1+|a|)^2}{|a|}e^{\|H_F\|_{\infty}\delta}\Theta_F(\alpha,\delta)+4\left(1+\delta^2e^{2\|H_F\|_{\infty}}\right)\frac{\delta^2}{|a|^2}\area(F).
$$
To complete the proof of~\eqref{u.bd} we write
$$
\frac{(1+|a|)^2}{|a|} \leqslant \frac{(\delta^3 + 2|a|)^2}{|a|}\leqslant 4 + 4\delta^3 + \frac{\delta^6}{|a|}
$$
and we once again use Lemma~\ref{mono} to bound the term $\Theta_F(\alpha,\delta)\left(4\delta^3 + \dfrac{\delta^6}{|a|}\right)$ in terms of $\area(F)$.
\end{proof} 

As an immediate corollary, we see that for any fixed immersion $F\colon M\to \mathbb{S}^n$ and $\alpha\in \mathbb{S}^n$, 
\begin{equation}\label{conf.dens.char}
\lim_{t\to 1}\area(G_{t\alpha}\circ F)=4\Theta_F(-\alpha,0).
\end{equation}

Finally, we can use this to analyze the boundary behavior of the area of the canonical family for a family of immersions converging in $C^2$.

\begin{corollary}
\label{cor:usc.}
 Let $F_j\colon M\to \mathbb{S}^n$ be a sequence of immersions converging in $C^2$ to an immersion $F\colon M\to \mathbb{S}^n$, and let $a_j\in \mathbb{B}^{n+1}$ be a sequence of points such that
$$
a_j\to \beta\in \mathbb{S}^n.
$$
Then 
\begin{equation}
\label{main.est.1}
4\Theta_F(-\beta,0)\geqslant \limsup_{j\to\infty}\area(G_{a_j}\circ F_j).
\end{equation}
In particular,
\begin{equation}\label{main.est.2}
\lim_{t\to 1}\area(G_{t\beta}\circ F)\geqslant \limsup_{j\to\infty}\area(G_{a_j}\circ F_j).
\end{equation}
\end{corollary}
\begin{proof}
Since the maps $F_j$ are converging in $C^2$ to an immersion $F\colon M\to \mathbb{S}^n$, it is easy to see that the induced metrics $g_{F_j}\to g_F$ converge in $C^1$, and the mean curvatures (as immersions in $\mathbb{R}^{n+1}$)
$$
H_{F_j}=\Delta_{g_{F_j}}F_j=-\left(\det(g_{F_j})\right)^{-1/2}\partial_a\left(\det(g_{F_j})^{1/2}g_{F_j}^{ab}\partial_bF_j\right)
$$
converge $H_{F_j}\to H_F$ in $C^0$. In particular, there exists $K>0$ such that
$$
\|H_{F_j}\|_{L^{\infty}}\leqslant K\text{ and }\|H_F\|_{L^{\infty}}\leqslant K
$$
and
$$
\area(F_j)\leqslant K.
$$
Thus, applying \eqref{u.bd} to $F_j$, we see that there exists $C$ independent of $j$ such that
$$
\area(G_{a_j}\circ F_j)\leqslant 4 e^{C\delta}\Theta_{F_j}(-a_j/|a_j|,\delta)+C\delta^2
$$
for any fixed $\delta>0$ and $j$ sufficiently large that $1-|a_j|<\delta^3$. Moreover, for fixed $\delta>0$, it is easy to see that
$$
\lim_{j\to\infty}\Theta_{F_j}(-a_j/|a_j|,\delta)=\Theta_F(-\beta,\delta),
$$
so that passing to the limit as $j\to\infty$ gives
$$
\limsup_{j\to\infty}\area(G_{a_j}\circ F_j)\leqslant 4e^{C\delta}\Theta_F(-\beta,\delta)+C\delta^2.
$$
Taking $\delta\to 0$, we then find
$$
\limsup_{j\to\infty}\area(G_{a_j}\circ F_j)\leqslant 4\Theta_F(-\beta,0),
$$
giving \eqref{main.est.1}. 

Finally, it follows from \eqref{l.bd} (taking $\delta\to 0$) that
$$
\lim_{t\to 1}\area(G_{t\beta}\circ F)\geqslant 4\Theta_F(-\beta,0),
$$
which together with \eqref{main.est.1} gives the desired semicontinuity estimate \eqref{main.est.2}.

\end{proof}

%

\section{Sharpness of  the quantitative stability estimates}
\label{sec:sharp}
\subsection{Optimality of the exponent}
The goal of this section is to show that the quantitative stability estimates 
obtained in the previous sections are sharp. 
As a test case, we will revisit Theorem \ref{S2stability:thm} that gives the quantitative stability of Hersch's inequality on the sphere, and
show that  neither the exponent two on the right-hand side of \eqref{eq:S2stabnew1} nor the norm $W^{-1,2}$ can be  improved.
For the purposes of this section we denote by $g$ a round metric on a sphere $\mathbb{S}^2$ of {\it unit} area embedded in $\mathbb{R}^3$. Unless specified otherwise, all the functional spaces throughout this section will be considered on $\mathbb{S}^2$ with respect to this metric, and  the dependence on $(\mathbb{S}^2,g)$  will be omitted to simplify notation.
As before, we say that  a measure $\mu$ on $\Sp$ is balanced whenever $\int_{\Sp}F d\mu=0\in\mathbb{R}^3$, where $F \colon\mathbb{S}^2\to\mathbb{S}^2\subset\mathbb{R}^3$ is the identity map. We also denote by $E_1$ the eigenspace  corresponding to the eigenvalue $\lambda_1(\Sp,g)$ spanned by the coordinate functions.

As was mentioned in Remark \ref{rem:stabcont}, the eigenvalue functional $\mu\mapsto \lambda_1(\mathbb{S}^2, \mu)$ is not continuous in
 $W^{-1,2}:=W^{-1,2}(\Sp,g)$ norm.  
Following \cite[Proposition 4.11]{GKL} we introduce the  Orlicz-Sobolev space below.
\begin{definition}
\label{def:Orlicz}
For a function $u\in L^1(\Sp,g)$ let 
\[
\|u\|_{L^2(\mathrm{Log}L)^{-\frac{1}{2}}} : =\inf\left\{\eta>0:\ \int_{\Sp}\frac{|u/\eta|^2}{\log(2+|u/\eta|)}\,d v_g\leqslant 1\right\}.
\]
The {\it Orlicz-Sobolev space} $W^{1,2,-\frac{1}{2}}:=W^{1,2,-\frac{1}{2}}(\Sp,g)$ is defined to be the space of functions such that
$$
\|u\|_{W^{1,2,-\frac{1}{2}}} := \|u\|_{L^2(\mathrm{Log}L)^{-\frac{1}{2}}} + \|du\|_{L^2(\mathrm{Log}L)^{-\frac{1}{2}}}<\infty.
$$
\end{definition}
 In what follows,  the only property property of $W^{1,2,-\frac{1}{2}}$ we are using is that 
 there exists a constant $C_{\text{Or}}$ such that for any $u\in W^{1,2}$ one has $\| u^2\|_{W^{1,2,-\frac{1}{2}}}\leqslant C_{\text{Or}}\Vert u\Vert_{W^{1,2}}^2$, which implies the continuity of $\mu\mapsto\lambda_1(\mu):=\lambda_1(\mathbb{S}^2,\mu)$ in the dual of $W^{1,2,-\frac{1}{2}}$ (see \cite{GKL} for more details). In the next proposition, we prove that if a balanced measure $\mu$ is close enough to $v_g$ in the dual of $W^{1,2,-\frac{1}{2}}$ (and hence in $W^{-1,2}$ as well), then the exponent two  in the stability estimate \eqref{eq:S2stabnew1} is sharp on a finite codimension subspace of admissible measures. 
\begin{proposition}
\label{prop:optpower}
There exist $\eps_0,c>0$ such that, for any balanced admissible measure $\mu$, if
 \begin{equation}
 \label{ineq:as1}
 \int_{\Sp}\varphi^2\,d(\mu-v_g)\leqslant \eps_0\Vert \varphi\Vert_{W^{1,2}}^2
 \end{equation}
 for any $\varphi\in W^{1,2}$,
then 
\begin{equation}
\label{ineq:conc1}
\lambda_1\left(\mu\right)\geqslant \frac{8\pi}{1+c\left(\Vert\mu- dv_g \Vert_{W^{-1,2}} + \Vert\mu- dv_g \Vert^2_{W^{-1,2}}\right)}.
\end{equation}
Moreover, if, in addition, for any $w \in E_1$ one has that  $\int_{\Sp} w^2\,d(\mu- v_g )=0$, then
\begin{equation}
\label{ineq:conc2}
\lambda_1\left(\mu\right)\geqslant \frac{8\pi}{1+c\Vert\mu- dv_g \Vert^2_{W^{-1,2}}}.
\end{equation}
In particular, the exponent in the stability estimate~\eqref{eq:S2stabnew1} is sharp.
\end{proposition}
\begin{remark}
Assumption~\eqref{ineq:as1} is verified whenever $\Vert \mu-v_g\Vert_{\left(W^{1,2,-\frac{1}{2}}\right)^*}$ is smaller than $C_\text{Or}^{-1}\eps_0$. 
\end{remark}
\begin{remark}
Inequality~\eqref{ineq:conc1} can be interpreted as continuity of $\lambda_1(\mu)$ at $\mu=dv_g$ with respect to the $W^{-1,2}$ distance  in the class of measures satisfying~\eqref{ineq:as1}. We will see in Corollary~\ref{cor:optpower_counter} below that an additional assumption~\eqref{ineq:as1} is necessary.
\end{remark}
\begin{proof}
Since 
\[\lambda_1(\mu)=\inf_{\varphi\in W^{1,2}}\frac{\int_{\Sp} |d \varphi|_{g}^2 \,d v_g }{\int_{\Sp} \left(\varphi-\int_{\Sp}\varphi \,d\mu\right)^2 \,d\mu},\]
and both numerator and denominator are invariant up to the addition of a constant to $\varphi$, $\lambda_1(\mu)$ may be written as
\[\lambda_1(\mu)=\inf_{\varphi\in 1^\bot}\frac{\int_{\Sp} |d \varphi|_{g}^2 \,d v_g }{\int_{\Sp} \varphi^2 \,d\mu-\left(\int_{\Sp}\varphi \,d\mu\right)^2},\]
where $1^\bot=\{\varphi: \int_{\Sp} \varphi \,d v_g =0\rbrace$.  For any  $\delta>0$, in order to prove that $\lambda_1(\mu)\geqslant \dfrac{8\pi}{1+\delta}$ it suffices to show that for any $\varphi\in 1^\bot$ one has
\begin{equation}
\label{eq:g1}
\int_{\Sp} \varphi^2\,d\mu\leqslant \frac{1+\delta}{8\pi}\int_{\Sp} |d \varphi|_{g}^2\,d v_g. 
\end{equation}
We write $\varphi=w+\psi$, where $w\in E_1$, 
$\psi\in\{1, E_1\}^\bot$. Notice that, since $\lambda_1(dv_g) = \lambda_3(dv_g)=8\pi$ and $\lambda_4(dv_g)=24\pi$, it follows that  $24\pi\int_{\Sp} \psi^2\,d v_g\leqslant \int_{\Sp}|d\psi|_g^2\,d v_g$ and $8\pi \int_{\Sp}w^2\,dv_g = \int_{\Sp}|dw|_g^2\,dv_g$. Thus, one obtains
\begin{equation*}
\begin{split}
&\frac{1+\delta}{8\pi}\int_{\Sp} |d \varphi|_{g}^2\,d v_g = \frac{1+\delta}{8\pi}\int_{\Sp}\left( |d w|_{g}^2 + |d \psi|_{g}^2\right) \,d v_g\geqslant\\
&\geqslant  (1+\delta)\int_{\Sp} \left(w^2+\psi^2\right)\,dv_g + \frac{1+\delta}{12\pi}\int_{\Sp}|d\psi|_g^2\,d v_g\geqslant\\
&\geqslant \int_{\Sp}\left(\varphi^2 + \delta w^2\right)\,dv_g + \frac{1}{12\pi} \int_{\Sp}|d\psi|_g^2\,d v_g
\end{split}
\end{equation*}
%
%
Thus, to obtain~\eqref{eq:g1} it is sufficient to show that 
\begin{equation}
\label{eq:g2}
\int_{\Sp} \varphi^2\,d(\mu- v_g )\leqslant \delta\int_{\Sp} w^2\,d v_g +\frac{1}{12\pi}\int_{\Sp} |d\psi|_g^2\,d v_g.
\end{equation}
Before proving this inequality, let us first remark that since $E_1$ is finite dimensional, there exists a constant $C$ such that $\|w\|_{W^{1,\infty}}\leqslant C\|w\|_{L^2}$. Similarly, for any 
$\psi \in \{1, E_1\}^\bot$, since $\lambda_4(dv_g)=24\pi$ we have $\Vert \psi\Vert_{W^{1,2}}^2\leqslant \left(1+(24\pi)^{-1}\right)\Vert d\psi\Vert_{L^2}^2$, and for a certain constant $C>0$
$$
\|w\psi\|_{W^{1,2}}\leqslant C\|w\|_{L^2}\|d\psi\|_{L^2}.
$$

Let us first verify~\eqref{ineq:conc2}. With the hypothesis on $\mu$, $\int_{\Sp}  w^2\,d(\mu- v_g )=0$, thus,  the left-hand side of~\eqref{eq:g2} may be estimated as follows,
\begin{equation*}
\begin{split}
&\int_{\Sp}  (w+\psi)^2 \,d(\mu- v_g )=\int_{\Sp}\left( 2 w\psi+ \psi^2\right)\,d(\mu- v_g )\leqslant\\
&\leqslant 2\Vert\mu- dv_g \Vert_{W^{-1,2} }\Vert w\psi \Vert_{W^{1,2} }+\eps_0\Vert \psi\Vert_{W^{1,2}}^2\leqslant \\
&\leqslant 2C\Vert\mu- dv_g \Vert_{W^{-1,2} }\Vert w\Vert_{L^2 }\Vert d\psi \Vert_{L^2 } +\left(1+\frac{1}{24\pi}\right)\eps_0\Vert d\psi\Vert_{L^2}^2\leqslant\\
&\leqslant 24\pi C^2\Vert \mu- dv_g\Vert_{W^{-1,2}}^2\Vert w\Vert_{L^2}^2+\left(\frac{1}{24\pi}+\left(1+\frac{1}{24\pi}\right)\eps_0\right)\Vert d\psi\Vert_{L^2}^2,
\end{split}
\end{equation*}
where we used the 
arithmetic-geometric mean  inequality in the last step. As a result, we obtain~\eqref{eq:g2} with $\delta = 24\pi C^2\Vert \mu-v_g\Vert_{W^{-1,2}}^2$ as long as we choose $\eps_0=\dfrac{1}{1+24\pi}$. Substituting this $\delta$ in~\eqref{eq:g1} completes the proof.

To show~\eqref{ineq:conc1} it is sufficient to note that
$$
\int_{\Sp}  w^2\,d(\mu- v_g)\leqslant \|\mu- dv_g \|_{W^{-1,2}}\|w^2\|_{W^{1,2}}\leqslant C\|\mu- dv_g \|_{W^{-1,2}}\|w\|^2_{L^2}.
$$ 
Adding this term to the computation above, we obtain~\eqref{eq:g2} with $\delta = C\Vert \mu-dv_g\Vert_{W^{-1,2}}(1+ 24\pi C\Vert \mu-dv_g\Vert_{W^{-1,2}})$.
\end{proof}

\subsection{Optimality of the $W^{-1,2}$ norm}
We now show that the Hersch inequality is not stable in $\left(W^{1,2,-\frac{1}{2}}\right)^*$, and, consequently, is also not stable in
$\left(W^{1,2-\eps}\right)^*$  for any $\varepsilon>0$, cf. Remark \ref{rem:stabcont}. Note that  by Sobolev embedding theorem, it is thus not stable in $\left(W^{1-\varepsilon,2}\right)^*$ either. We claim that in order to show this it is sufficient to prove the following theorem.
\begin{theorem}
\label{thm:spacesharp}
There exists a sequence $(\mu_\eps)_\eps$ of admissible, balanced probability measures on $\Sp$, such that $\lambda_1(\mu_\eps)\underset{\eps\to 0}{\longrightarrow}8\pi$ and
\[\liminf_{\eps\to 0}\Vert \mu_\eps-v_g\Vert_{\left(W^{1,2,-\frac{1}{2}}\right)^*}>0.\]
\end{theorem}
Indeed, assume that there exist conformal automorphisms $\Phi_\eps\in\Conf(\Sp)$ such that $(\Phi_\eps)_*\mu_\eps\to dv_g$ in $\left(W^{1,2,-\frac{1}{2}}\right)^*$. Note that since $\bar\lambda_1(\mu_\eps)\to 8\pi$ and $\mu_\eps$ are balanced, an application of Lemma~\ref{main:lemma} with $u=\mathrm{id}\colon \Sp\to\Sp$ yields that $\mu_\eps \to dv_g$ in $W^{-1,2}$. Suppose that $\Phi_\eps$ converges (smoothly up to a choice of a subsequence) to a conformal automorphism 
$\Phi_0$. Then $\mu_\eps\to (\Phi_0^{-1})_*dv_g$ in $\left(W^{1,2,-\frac{1}{2}}\right)^*$ and, at the same time, $\mu_\eps\to dv_g$ in $W^{-1,2}$. 
Since the space $\left(W^{1,2,-\frac{1}{2}}\right)^*$ embeds in $W^{-1,2}$, one has $(\Phi_0^{-1})_*dv_g = dv_g$ in contradiction with the conclusion of Theorem~\ref{thm:spacesharp}. If $\Phi_\eps$ does not converge, then (up to a choice of a subsequence) $\Phi_\eps$ sends most of $\Sp$ into a shrinking neighbourhood of a single point. In particular, 
$\mu_\eps\rightharpoonup\delta_p$ for some $p\in\Sp$, which contradicts  $\mu_\eps \to dv_g$ in $W^{-1,2}$.

In order to prove Theorem~\ref{thm:spacesharp} we need the following lemma, where we use the notation $a_\eps\underset{\eps\to 0}{\sim} b_\eps$ to mean that $\lim\limits_{\eps\to 0}\dfrac{a_\eps}{b_\eps}=1$.
\begin{lemma}
Let $B_\eps$ be 
a disk of radius $\eps$ in $\R^2$, then
\begin{align}
\sup_{u\in W^{1,2}_0(B_1),\ \Vert d u\Vert_{L^2(B_1)}=1}\int_{B_\eps}u^2\,d x&\underset{\eps\to 0}{\sim}\frac{1}{2}\eps^2\log(1/\eps)\label{asymptot1},\\
\sup_{u\in W^{1,2}_0(B_1),\ \Vert d u\Vert_{L^2(B_1)}=1}\int_{B_\eps}u\,d x&\underset{\eps\to 0}{\sim}\eps^2\sqrt{\frac{\pi}{2}\log(1/\eps)}\label{asymptot2},
\end{align}
where $W^{1,2}_0(B_1)$ is the completion of $C^\infty_0(B_1,\R)$ with its respective norm.
\end{lemma}
\begin{proof}
We start with the proof of  \eqref{asymptot1}. The inverse of this quantity may be rewritten as the eigenvalue problem 
\[\lambda_\eps:=\inf_{u\in W^{1,2}_0(B_1)}\frac{\int_{B_1}|d u|^2\,dx}{\int_{B_\eps}u^2\,dx}.\]
By compactness of $W^{1,2}_0(B_1)\hookrightarrow L^2(B_1)$ the infimum is attained  for some function $u_\eps\in W^{1,2}_0(B_1)$.  Its spherical mean $\overline{u}_\eps(x)=\frac{1}{|\partial B_{|x|}|}\int_{\partial B_{|x|}}u_\eps\,dx$ has lower energy, so we may suppose without loss of generality that $u_\eps(x)=\psi_\eps(|x|)$ for a certain function $\psi_\eps (r)$ defined for $r\in[0,1]$.
The Euler-Lagrange equation associated to this problem is 
\[\psi_\eps''+\frac{1}{r}\psi_\eps'=-\lambda_\eps1_{r<\eps} \psi_\eps,\]
which implies that $\psi_\eps\in C^1((0,1])$ by elliptic regularity. On $B_1\setminus B_\eps$ we know that for a certain constant $k>0$:
\[u_\eps(x)=k\frac{\log(1/|x|)}{\log(1/\eps)}.\]
Notice also that at $|x|=\eps$ we have  $\frac{\partial_r u_\eps}{u_\eps}=-\frac{1}{\eps\log(1/\eps)}$, thus $u_{\eps}|_{B_{\eps}}$ is an eigenfunction associated to the first eigenvalue of the Laplacian with Robin condition of parameter $\frac{1}{\eps\log(1/\eps)}$, denoted $\lambda_1\left(B_\eps;\frac{1}{\eps\log(1/\eps)}\right)$. Now by applying Stokes' formula to $u_\eps$ on $B_1\setminus B_\eps$ we see that
\[\lambda_\eps=\frac{\int_{B_1}|\nabla u_\eps|^2}{\int_{B_\eps}u_\eps^2}=\frac{\int_{B_\eps}|\nabla u_\eps|^2+\frac{1}{\eps\log(1/\eps)}\int_{\partial B_{\eps}}u_\eps^2}{\int_{B_\eps}u_\eps^2}=\lambda_1\left(B_\eps;\frac{1}{\eps\log(1/\eps)}\right).\]
At the same time,  $\lambda_1\left(B_\eps;\frac{1}{\eps\log(1/\eps)}\right)=\eps^{-2}\lambda_1\left(B_1;\frac{1}{\log(1/\eps)}\right)\sim \frac{2}{\eps^2\log(1/\eps)}$, where we use in the first step the general scaling property
\[\lambda_1(r\Omega;\alpha)=r^{-2}\lambda_1(\Omega;r\alpha),\]
and in the second step the well-known asymptotic formula  $\lambda_1(\Omega;\alpha)\underset{\alpha\to 0^+}{\sim}\alpha \frac{|\partial\Omega|}{|\Omega|}$ (\cite{S73}, see also
\cite{LOS98, GS08}).

Similarly, let us now prove  \eqref{asymptot2}.
By compactness of $W^{1,2}_0(B_1)\hookrightarrow L^1(B_1)$, we find that the supremum is attained for a positive function $u_\eps$. Using the same spherical mean argument as earlier, we may suppose that $u_\eps$ is radial and satisfies the Euler-Lagrange equation $-\Delta u_\eps=c_\eps 1_{B_\eps}$ for a certain constant $c_\eps$. In particular, $u\in{C}^1(B_1)$ by elliptic regularity. 
Thus, up to multiplication by a scalar, $u_\eps$ is given by:
\[u_\eps(x)=\begin{cases}\log(1/|x|)&\text{ if }\eps\leqslant|x|\leqslant 1\\ \log(1/\eps)+\frac{\eps^2-|x|^2}{2\eps^2}&\text{ if }|x|\leqslant \eps.\end{cases}\]
An explicit computation yields the result.  
\end{proof}

Consider two antipodal points $x_1,x_2$ on $\Sp$, e.g. the north pole $N=x_1$ and south pole $S=x_2$. Let $r>0$ be such that $B_g(x_i,2r)$ is the corresponding hemisphere. Using the stereographic projection from $x_i$ onto the equatorial plane one can construct conformal flat metrics $g_i:=e^{2w_i}g$ on $B_g(x_i,2r)$, such that 
\begin{itemize}
\item $w_i(x)$ is bounded and only depends on $\dist_g(x,x_i)$;
\item $B_g(x_i,r) = B_{g_i}(x_i,1)$;
\item if $\rho_i(\eps)$ is such that $B_g(x_i,\eps)=B_{g_i}(x_i,\rho_i(\eps))$, then $\rho_i(\eps)\underset{\eps\to 0}{\sim}e^{w_i(x_i)}\eps$.
\end{itemize}

\begin{corollary} There exists a constant $C_0>0$ such that for any $\eps>0$ small enough one has
\begin{align}
C_0^{-1}\eps^2\log(1/\eps)&\leqslant\sup_{\| u\|_{W^{1,2}}=1}\int_{\cup_{i=1}^{2}B_g(x_i,\eps)}u^2\,d v_g\leqslant C_0\eps^2\log(1/\eps)\label{asymptot_sphere_1},\\
C_0^{-1}\eps^2\sqrt{\log(1/\eps)}&\leqslant \sup_{\Vert u\Vert_{W^{1,2}}=1}\int_{\cup_{i=1}^{2}B_g(x_i,\eps)}u\,d v_g\leqslant C_0\eps^2\sqrt{\log(1/\eps)}\label{asymptot_sphere_2}.
\end{align}

\end{corollary}
\begin{proof}

We prove \eqref{asymptot_sphere_1}, the proof of  \eqref{asymptot_sphere_2} is identical. Let $\chi_i\in C^\infty(\Sp,[0,1])$ be such that $\{\chi_i\neq 0\}\Subset B_g(x_i,r)$ and $B_g(x_i,r/2)\Subset \{\chi_i=1\}$.
The estimate \eqref{asymptot1} implies that there exists a constant $C_1>0$ such that for any small enough $\eps>0$,

\[
C^{-1}_1\eps^2\log(1/\eps)\leqslant \sup_{u\in W^{1,2}_0(B_1),\ \Vert d u\Vert_{L^2(B_1)}=1}\int_{B_\eps}u^2\,d x\leqslant C_1\eps^2\log(1/\eps).
\]
Since the functions $w_i$ are bounded, for any  $\varphi\in W^{1,2}$, one has 
\begin{align*}
&\int_{\cup_{i=1}^{2}B_g(x_i,\eps)}\varphi^2\, dv_g\leqslant C\sum_{i=1}^{2}\int_{B_{g_i}(x_i,\rho_i(\eps))}(\chi_i\varphi)^2\, dv_{g_i}\leqslant\\
&\leqslant CC_1\sum_{i=1}^{2}\rho_i(\eps)^2\log(1/\rho_i(\eps))\Vert d(\chi_i \varphi)\Vert_{L^2(\Sp,g_i)}^2\leqslant\\
&\leqslant CC_1\left(1+\Vert d\chi_i\Vert^2_{L^\infty}\right)\eps^2\log(1/\eps)(1+o(1))\Vert\varphi\Vert_{W^{1,2}}^2,
\end{align*}
where $C$ is a constant depending only on $\|w_i\|_{L^\infty}$, possibly changing from line to line.

Conversely, for each $\eps>0$ there is a function $\varphi_\eps\in C^\infty_0(B_g(x_1,r))$ such that $\int_{B_{g_1}(x_1,\rho_1(\eps))}\varphi^2_\eps\, dv_{g_1}\geqslant C_1^{-1}\rho_1(\eps)^2\log(1/\rho_1(\eps))\Vert d\varphi_\eps\Vert_{L^2}^2$. As a result,
\begin{align*}
&\int_{\cup_{i=1}^{2}B_g(x_i,\eps)}\varphi^2_\eps\, dv_{g}\geqslant CC_1^{-1}\rho_1(\eps)^2\log(1/\rho_1(\eps))\Vert d\varphi_\eps\Vert_{L^2}^2\geqslant\\
&\geqslant \frac{CC^{-1}(1+o(1))}{(1+\lambda_*(B_g(x_1,r),g))}\eps^2\log(1/\eps)\Vert \varphi_\eps\Vert_{W^{1,2}}^2,
\end{align*}
where $\lambda_*(\Omega,g)$ is the first eigenvalue of the Laplacian with Dirichlet boundary conditions on $\partial\Omega$.

The proof of~\eqref{asymptot_sphere_2} is the same, except $\varphi^2$ is replaced with $\varphi$.

\end{proof}

We are now ready to define our sequence of measures. For any $M,\eps>0$ set
\[
\nu_\eps=\frac{1_{\cup_{i=1}^2 B_g(x_i,\eps)}}{\eps^2\log(1/\eps)}dv_g,\ \mu_\eps^M=\frac{dv_g+M \nu_\eps}{1+M\nu_\eps(\Sp)}.
\]
The following lemma motivates the definition of $\nu_\eps$.
\begin{lemma}
\label{lemma:mps}
For any $M,\eps>0$ the measures $\nu_\eps$, $\mu^M_\eps$ possess the following properties,
\begin{enumerate}
\item $\|\nu_\eps\|_{W^{-1,2}}\to 0$ as $\eps\to 0$. In particular, one has 
$$
\|\mu^M_\eps-dv_g\|_{W^{-1,2}}\to 0;
$$
\item There exists $c>0$ such that
$$
\liminf_{\eps\to 0}\Vert \mu_\eps^M-dv_g\Vert_{\left(W^{1,2,-\frac{1}{2}}\right)^*}>cM>0.
$$
\end{enumerate}
\end{lemma}
\begin{proof} The upper bound~\eqref{asymptot_sphere_2} implies that
$$
\|\nu_\eps\|_{W^{-1,2}}\leqslant C_0\left(\log(1/\eps)\right)^{-\frac{1}{2}}\to 0.
$$
As a result,
$$
\|\mu^M_\eps-dv_g\|_{W^{-1,2}} = \left\|\frac{M\nu_\eps(\Sp)dv_g + M\nu_\eps}{1 +M\nu_\eps(\Sp)}\right\|_{W^{-1,2}}\to 0,
$$
since $\nu_\eps(\Sp)\leqslant C\left(\log(1/\eps)\right)^{-1}\to 0$. This completes the proof of $(1)$.

To show $(2)$ we write
\begin{align*}
&\Vert \mu_\eps^M-dv_g\Vert_{\left(W^{1,2,-\frac{1}{2}}\right)^*}\geqslant \sup_{\varphi\in W^{1,2}}\frac{\int_{\Sp}\varphi^2\, d(\mu_\eps^M-v_g)}{\Vert \varphi^2\Vert_{W^{1,2,-\frac{1}{2}}}}\geqslant \\
&\geqslant \sup_{\varphi\in W^{1,2}}\frac{M}{1+M\nu_\eps(\Sp)}\left(\frac{\int_{\Sp}\varphi^2\, d\nu_\eps}{C_\text{Or}\Vert \varphi\Vert_{W^{1,2}}^2}
-\frac{\nu_\eps(\Sp)\int_{\Sp}\varphi^2dv_g}{\Vert \varphi^2\Vert_{\left(W^{1,2,-\frac{1}{2}}\right)^*}}
\right)\geqslant\\
&\geqslant \frac{cM}{1+M\nu_\eps(\Sp)}(1-C\nu_\eps(\Sp)),
\end{align*}
where in the last step we used the lower bound~\eqref{asymptot_sphere_1}. Since $\nu_\eps(\Sp)\to 0$, the proof is complete.

\end{proof}

\begin{proposition}\label{prop:counterexample}
There are constants $M_0,C>0$ such that
\begin{enumerate}
\item $\limsup_{\eps\to 0}\lambda_1(\mu_\eps^M)\leqslant \frac{C}{M}$;
\item If $M<M_0$, then $\lambda_1(\mu_\eps^M)\to 8\pi$ as $\eps\to 0$.
\end{enumerate}
\end{proposition}
\begin{proof}

To prove $(1)$ we write
\[\lambda_1(\mu_\eps^M)^{-1}=\sup_{\int_{\Sp}|d\varphi|^2dv_g=1}\left(\int_{\Sp}\varphi^2d\mu_\eps^M-\left(\int_{\Sp}\varphi d\mu_\eps^M\right)^2\right)\]
This formula is invariant up to the addition of a constant to $\varphi$, so we may take $\varphi\in 1^\perp:=\left\{\psi:\int_{\Sp}\psi dv_g=0\right\}$ without loss of generality. Let $H=\{\varphi\in 1^\perp,\,\int_{\Sp}|d\varphi|^2dv_g=1\}$, this may be rewritten
\begin{align*}
&\lambda_1(\mu_\eps^M)^{-1}=\\
&=\frac{1}{1+M\nu_\eps(\Sp)}\sup_{\varphi\in H}\left(\int_{\Sp}\varphi^2d(v_g+M\nu_\eps)-\frac{M^2}{1+M\nu_\eps(\Sp)}\left(\int_{\Sp}\varphi d\nu_\eps\right)^2\right)\geqslant\\
&\geqslant \frac{M}{1+M\nu_\eps(\Sp)}\sup_{\varphi\in H}\int_{\Sp}\varphi^2d\nu_\eps-\frac{(1+(8\pi)^{-1})M^2\Vert \nu_\eps\Vert_{W^{-1,2}}^2}{(1+M\nu_\eps(\Sp))^2}.
\end{align*}
The second term goes to $0$ by Lemma~\ref{lemma:mps}. Moreover, according to the estimate \eqref{asymptot1}, we may find a function $\psi_\eps$ with support on $B_g(x_1,r)$ such that $\Vert d\psi_\eps\Vert_{L^2(\Sp,g_1)}=1$ and
\[
\int_{B_{g_1}(x_1,\rho_1(\eps))}\psi_\eps^2 dv_{g_1}\geqslant \frac{1}{2}C_1^{-1}\eps^2\log(1/\eps).
\]
Let $\phi_\eps$ be the same function on $B_{g}(x_2,r)$ (which is isometric to $B_{g}(x_1,r)$), then $\frac{\phi_\eps-\psi_\eps}{\sqrt{2}}\in H$ and there is a constant $c>0$ such that
\[\sup_{\varphi\in H}\int_{\Sp}\varphi^2d\nu_\eps\geqslant \int_{\Sp}\left(\frac{\phi_\eps-\psi_\eps}{\sqrt{2}}\right)^2d\nu_\eps\geqslant c.\]
Thus, as $\eps\to 0$, we obtain $\lambda_1(\mu_\eps^M)^{-1}\geqslant cM$. 

To show $(2)$ we note that~\eqref{asymptot_sphere_1} implies that for any $\varphi\in W^{1,2}$,
\[
\int_{\Sp}\varphi^2\, d\nu_\eps\leqslant C_0M\Vert \varphi\Vert_{W^{1,2}}.
\]
Therefore, the assumption~\eqref{ineq:as1} of Proposition~\ref{prop:optpower} is satisfied for $M<M_0:=C_0^{-1}\eps_0$. An application of~\eqref{ineq:conc1} concludes the proof.
%
%
%
%
\end{proof}

Let $M$ be large enough so that $\limsup_{\eps\to 0}\lambda_1(\mu_\eps^M)\leqslant \frac{C}{M}<8\pi$. Then by Lemma~\ref{lemma:mps} part (1) the sequence $(\mu_\eps^M)_{\eps\to 0}$ gives the proof of the following corollary.
\begin{corollary}\label{cor:optpower_counter}
$\mu\mapsto \lambda_1(\mu)$ is not continuous in $W^{-1,2}$ and the Proposition~\ref{prop:optpower} fails without the assumption~\eqref{ineq:as1}.
\end{corollary}

At the same time, when $M<M_0$ we obtain Theorem~\ref{thm:spacesharp}.

\begin{proof}[Proof of Theorem \ref{thm:spacesharp}]
Consider $M<M_0$ and $(\mu_\eps^M)_{\eps}$ as defined in proposition \ref{prop:counterexample}, then $\lambda_1(\mu_\eps^M)\goto 8\pi$. An application of Lemma~\ref{lemma:mps}, part (2) completes the proof the corollary.

\end{proof}


\end{document}